\documentclass[12pt,a4paper]{article}
\usepackage{amssymb}
\usepackage{amsmath}
\usepackage{amsthm}
\usepackage{thmtools}
\usepackage{enumitem}
\usepackage[all]{xy}
\usepackage{framed}
\usepackage{hyperref}
\usepackage{parskip}
\usepackage[a4paper,margin=2cm]{geometry}
\usepackage{multicol}
\usepackage{marginnote}
\usepackage{xcolor}
\usepackage{tikz-cd}
\usepackage[normalem]{ulem}

\begingroup
    \makeatletter
    \@for\theoremstyle:=definition,remark,plain\do{%
        \expandafter\g@addto@macro\csname th@\theoremstyle\endcsname{%
            \addtolength\thm@preskip\parskip
            }%
        }
\endgroup

\makeatletter
\xydef@\txt@ii#1{\vbox{\vspace*{-5pt}%
 \let\\=\cr
 \tabskip=\z@skip \halign{\relax\hfil\txtline@@{##}\hfil\cr\leavevmode#1\crcr}}}
\makeatother

\theoremstyle{definition}
\newtheorem{thm}{Theorem}[section]
\newtheorem{lem}[thm]{Lemma}
\newtheorem{proc}[thm]{Procedure}
\newtheorem{cor}[thm]{Corollary}
\newtheorem{defn}[thm]{Definition}
\newtheorem{propn}[thm]{Proposition}
\newtheorem*{thm*}{Theorem}

\newtheorem*{qn*}{Question}

\newtheorem{props}[thm]{Properties}

\newcounter{lthm}
\newtheorem{letterthm}[lthm]{Theorem}

\theoremstyle{remark}
\newtheorem{rk}[thm]{Remark}
\newtheorem*{rk*}{Remark}
\newtheorem{rks}[thm]{Remarks}
\newtheorem*{rks*}{Remarks}
\newtheorem{ex}[thm]{Example}

\newcommand{\gr}{\mathrm{gr}}
\newcommand{\Aut}{\mathrm{Aut}}
\newcommand{\id}{\mathrm{id}}

\newcommand{\chr}{\mathrm{char}}
\renewcommand{\O}{\mathcal{O}}

\newcommand\blfootnote[1]{%
  \begingroup
  \renewcommand\thefootnote{}\footnote{#1}%
  \addtocounter{footnote}{-1}%
  \endgroup
}

\numberwithin{thm}{subsection}
\numberwithin{lem}{subsection}
\numberwithin{propn}{subsection}
\numberwithin{props}{subsection}
\numberwithin{defn}{subsection}
\numberwithin{cor}{subsection}
\numberwithin{conj}{subsection}
\numberwithin{rk}{subsection}
\numberwithin{proc}{subsection}
\numberwithin{notn}{subsection}
\numberwithin{qn}{subsection}

\begin{document}

\numberwithin{equation}{section}
\binoppenalty=\maxdimen
\relpenalty=\maxdimen

\title{Bounded skew power series rings for inner \(\sigma\)-derivations}
\author{Adam Jones and William Woods}
\date{\today}
\maketitle
\begin{abstract}


\noindent We define and explore the bounded skew power series ring $R^+[[x;\sigma,\delta]]$ defined over a complete, filtered, Noetherian prime ring $R$ with a commuting skew derivation $(\sigma,\delta)$. We establish precise criteria for when this ring is well-defined, and for an appropriate completion $Q$ of $Q(R)$, we prove that if $Q$ has characteristic $p$, $\delta$ is an inner $\sigma$-derivation and no positive power of $\sigma$ is inner as an automorphism of $Q$, then $Q^+[[x;\sigma,\delta]]$ is often prime, and even simple under certain mild restrictions on $\delta$. It follows from this result that $R^+[[x;\sigma,\delta]]$ is itself prime.

\end{abstract}
\tableofcontents

\section{Introduction and motivation}

Let $R$ be a complete filtered ring and $(\sigma, \delta)$ a skew derivation (Definition \ref{defn: skew derivation}) on $R$. Under certain appropriate \emph{compatibility} conditions between $(\sigma, \delta)$ and the filtration on $R$ (see Definition \ref{defn: skew power series ring, compatible case} below), we can define the \emph{skew power series ring}
$$R[[x; \sigma, \delta]] = \left\{ \sum_{n\geq 0} r_n x^n : r_n\in R\right\},$$
also a topological ring, with continuous multiplication (extending the multiplication on $R$) uniquely defined by the family of relations $xr = \sigma(r)x + \delta(r)$ for all $r\in R$. This is of course related to the classical notion of a \emph{skew polynomial ring} $R[x;\sigma,\delta]$ \cite[\S 1.2]{MR}, which sits inside $R[[x; \sigma, \delta]]$ as a dense subring, and whose properties are far better understood. (Note that the existence of skew power series rings is non-trivial, and conditions under which they exist are not yet fully understood: see \cite[\S 1]{SchVen06} and \cite[\S 3.4]{letzter-noeth-skew}.)

Skew power series rings of \emph{automorphic type} (with $\delta = 0$) are fairly well-studied classical objects \cite[\S 1.4]{MR} \cite[\S 2.3]{Coh95}. We note also that the case where $\delta$ is (locally) nilpotent has recently been studied \cite{bergen-grzeszczuk-skew}. However, for our purposes, this will not be enough, and we will need $\delta$ to satisfy a more general \emph{topological} nilpotency property. Skew power series rings, in this level of generality, were first introduced and studied in \cite[\S 1]{SchVen06} and \cite[\S 2]{venjakob}.

The primary motivation for the current paper comes from the study of Iwasawa algebras of soluble compact $p$-adic Lie groups: see \S \ref{subsec: iwasawa} below. In particular, there is an ongoing attempt to understand the prime ideals in Iwasawa algebras (see e.g.\ \cite{ardakovInv, ardakovbrown, jones-abelian-by-procyclic, woods-catenary}), and in skew power series extensions more generally \cite{letzter-noeth-skew}. However, we also mention that this skew power series construction encompasses many other rings of interest elsewhere, including $q$-commutative power series rings \cite{letzter-wang-q-comm} and other completed quantum coordinate rings \cite{wang-quantum}.

In contrast to skew polynomial rings, whose prime ideal structure is relatively well understood \cite{goodearl-letzter}, very little is known about the prime ideal structure of skew power series rings. Almost nothing is known in generality apart from the results of \cite{letzter-noeth-skew}. 

\subsection{Iwasawa algebras}\label{subsec: iwasawa}

Iwasawa algebras are objects of fundamental importance in the representation theory of $p$-adic groups.

A \emph{compact $p$-adic Lie group} $G$ is most simply defined as a closed subgroup of some $GL_n(\mathbb{Z}_p)$ (a topological group), where $\mathbb{Z}_p$ denotes the group of $p$-adic integers. Such a group is profinite \cite[Definition 1.1]{DDMS}, so its open subgroups form a base for the neighbourhoods of the identity. Given a fixed base ring $k$, usually taken to be $\mathbb{F}_p$ or $\mathbb{Z}_p$, the \emph{Iwasawa algebra} of $G$ (over $k$) is the completion $kG$ of the usual group algebra $k[G]$: when $k = \mathbb{F}_p$ or $\mathbb{Z}_p$, this can be calculated as $$kG=\underset{U\trianglelefteq_o G}{\varprojlim}{k[G/U]},$$ where the notation $U\trianglelefteq_o G$ means that $U$ is an open normal subgroup of $G$.

If $G$ is a compact $p$-adic Lie group, and $H$ is a closed normal subgroup of $G$ satisfying $G/H \cong \mathbb{Z}_p$, then there exists a skew derivation $(\sigma, \delta)$ on $kH$ and an element $x\in kG$ such that $kG = kH[[x; \sigma, \delta]]$ \cite[Lemma 1.6 and \S 4]{SchVen06}. In fact, we can choose $x$ such that $\delta = \sigma - \mathrm{id}$ (we will say that the skew derivation $(\sigma, \delta)$ is \emph{of Iwasawa type}). In this context, an example of a foundational motivating question for us is:

\vspace{12pt}
\centerline{
\text{\parbox{.78\textwidth}{
\emph{Let $P$ be a $G$-invariant prime ideal of $kH$, so that $PkG$ is a two-sided ideal of $kG$. Is $PkG$ prime?}
}
}
}

This is a very special case of the question \cite[3.19]{letzter-noeth-skew}, but it still remains open. An important alternative way of formulating this question is, if $P$ is a $G$-invariant prime ideal of $kH$, and $R=kH/P$, then is $kG/PkG=R[[\overline{x};\overline{\sigma},\overline{\delta}]]$ a prime ring?

In studying compact $p$-adic Lie groups, it is common to focus first on a slightly more restricted class of groups. Two such classes are common: the \emph{uniform} (or \emph{uniformly powerful}) groups \cite{DDMS}, and the \emph{$p$-valuable} groups \cite{lazard}, both classes of compact $p$-adic Lie groups admitting well-behaved filtration properties. If $G$ is a soluble uniform group or a soluble $p$-valuable group, it is known that there is a series
$$1 = G_0 \lhd G_1 \lhd \dots \lhd G_n = G$$
where each $G_i/G_{i-1} \cong \mathbb{Z}_p$. In particular, $kG$ can be built up from $k$ as an $n$-fold iterated skew power series extension, where each skew derivation is of Iwasawa type.

\subsection{The filtration assumptions}

We will need to make the following assumptions often, so we give them a name. We will say that the filtered ring $(R, w_0)$ satisfies \eqref{filt} if
\begin{equation}
    \tag{filt}
    \qquad\;\;\left.
    \text{\parbox{.78\textwidth}{
    $R$ is a prime $\mathbb{Z}_p$-algebra, and $w_0: R\to \mathbb{Z}\cup\{\infty\}$ is a complete (separated) filtration on $R$, such that $\gr_{w_0}(R)$ is a Noetherian ring, finitely generated as a module over a central graded subring $A$ which contains a non-nilpotent element of non-zero degree.
    }}
    \quad\;\; \right\}
    \label{filt}
\end{equation}

(We will usually omit the word ``separated", and assume throughout that all filtrations are separated, in contrast to our main reference \cite{LVO}.)

These conditions will allow us to define a filtration $u$ on the Goldie ring of quotients $Q(R)$ with particularly nice properties: this filtration can be constructed along the lines of \cite[Theorem C]{ardakovInv} or \cite[Theorem 3.3]{jones-abelian-by-procyclic}. We state the result we will use as Theorem \ref{thm: filtered localisation} below, but as we will also analyse the procedure for building this filtration in detail, we outline it as Procedure \ref{proc: build a standard filtration}.

The condition \eqref{filt} is relatively mild in our context. For instance, let $k = \mathbb{F}_p$ or $\mathbb{Z}_p$, and let
\begin{equation}\label{eqn: iterated local skew power series ring}
S = k[[x_1]][[x_2; \sigma_2, \delta_2]] \dots [[x_n; \sigma_n, \delta_n]]
\end{equation}
be an iterated local skew power series ring in the sense of \cite{woods-SPS-dim}, a scalar local ring with maximal ideal $\mathfrak{m}$. Suppose that $\deg_\mathfrak{m}(\sigma_i - \id) \geq 1$ for all $2\leq i\leq n$. Suppose also that \emph{either} $\deg_\mathfrak{m}(\delta_i) \geq 2$ for all $2\leq i\leq n$, \emph{or} that \eqref{eqn: iterated local skew power series ring} is \emph{triangular} in the sense of \cite[Definition 2.8]{woods-SPS-dim}. Then there is a complete separated filtration $f$ on $S$ such that $\gr_f(S)$ is a commutative polynomial ring over $\mathbb{F}_p$ in $n$ variables (if $k = \mathbb{F}_p$) or $n+1$ variables (if $k = \mathbb{Z}_p$) \cite[Theorems D and E]{woods-SPS-dim}. Now, if $P$ is a non-maximal prime ideal of $S$, then the quotient filtration on $R := S/P$ satisfies \eqref{filt} \cite[Example 4.11]{jones-woods-1}.

For example, these conditions are satisfied when $S = kH$ is the Iwasawa algebra of a soluble uniform pro-$p$ group \cite[Example 4.12]{jones-woods-1}, as in \S \ref{subsec: iwasawa}. That is: if $P$ is a non-maximal prime ideal of $kH$, then $kH/P$ satisfies \eqref{filt}.

\subsection{Bounded skew power series rings and subrings}\label{subsec: main results}

(We state abbreviated versions of our main results in this introduction, and reference more detailed versions in the main text of the paper.)

As intimated in \S \ref{subsec: iwasawa}, an important open question in the representation theory of compact $p$-adic Lie groups is the question of whether $R[[x;\sigma,\delta]]$ is a prime ring when $R$ is a prime quotient of an Iwasawa algebra. The main goal of of this paper is to partially answer this question for a larger class of rings $R$.

For now, we work more generally. Let $(S,w)$ be a complete filtered ring, and let $(\sigma, \delta)$ be a \emph{commuting} skew derivation on $S$, i.e. a skew derivation satisfying $\sigma\delta = \delta\sigma$. (In the literature on \emph{$q$-skew derivations}, this is sometimes called a \emph{1-skew derivation}, but we have avoided this term in favour of a more descriptive term.)

In \cite{jones-woods-1}, the authors explored the notion of \emph{compatibility} of a skew derivation (Definition \ref{defn: compatibility}), and demonstrated how various partial skew power series rings over $S$ exist when this condition is satisfied. However, compatibility is quite a strong condition, so we will need a more nuanced criterion to suit the circumstances of our current investigation.

In \S \ref{subsec: restricted skew power series}, we will introduce the notion of a \emph{quasi-compatible} skew derivation: roughly speaking, $(\sigma,\delta)$ is quasi-compatible with $w$ if $\sigma^i\delta^j$ is a filtered endomorphism of $(S,w)$ for any $i\in\mathbb{Z}$, $j\in\mathbb{N}$, their degrees with respect to $w$ all have a common lower bound, and $\delta$ is topologically nilpotent in an appropriate sense. We demonstrate in \S \ref{sec: filtered localisation} that this allows us to define a ring structure on the set of bounded power series over $S$, culminating in our first main result:

\begin{letterthm}[Theorem \ref{thm: restricted skew power series rings exist}]\label{letterthm: restricted skew power series rings exist}
Suppose $(S,w)$ is a complete filtered ring and $(\sigma,\delta)$ is a quasi-compatible commuting skew derivation on $S$. Then the \emph{bounded skew power series ring} $$S^+[[x; \sigma, \delta]]:=\left\{\sum_{n\in\mathbb{N}} s_nx^n:s_n\in S, \text{ and there exists } M\in\mathbb{Z} \text{ such that }w(s_n)\geq M\text{ for all }n\right\}$$ is a well-defined ring. It can be given a ring topology, and the skew polynomial ring $S[x; \sigma, \delta]$ is a dense subring. Moreover, if $w$ is a positive filtration, $S^+[[x;\sigma,\delta]] = S[[x; \sigma, \delta]]$ is the (complete) ring of ordinary skew power series over $S$.\qed
\end{letterthm}

The proof that this ring exists is technical but necessary, as we will usually not be working with positive filtrations.

\begin{rk*}
In \cite{jones-woods-1}, we instead used the term bounded skew power series to refer to the ring $S^b[[x;\sigma,\delta]]$, whose elements satisfy a weaker convergence condition. These are not the same ring in general: it should be remarked that, under appropriate conditions, $S^b[[x;\sigma,\delta]]$ arises as a completion of $S^+[[x;\sigma,\delta]]$.
\end{rk*}

A key ingredient in many later results is the following data, which we use to define large subrings of $S^+[[x;\sigma,\delta]]$. Suppose now that $(S,w)$ is a general \emph{filtered $\mathbb{Z}_p$-algebra} (Definition \ref{defn: filtered Z_p-algebra}): the following are Definitions \ref{defn: X_n etc in char p} and \ref{defn: X_n etc in inner case}.

\begin{enumerate}[label=(\arabic*)]
\item If $\chr(S) = p$, set $X_n = x^{p^n}$, $\Sigma_n = \sigma^{p^n}$, $\Delta_n = \delta^{p^n}$.
\item If there exists $t\in S$ such that $\delta(s) = ts - \sigma(s)t$ for all $s\in S$ and $\sigma(t) = t$, set
$$T_n = (-1)^{p+1}t^{p^n},\qquad X_n = (x-t)^{p^n}+T_n,\qquad \Sigma_n = \sigma^{p^n},$$
and define $\Delta_n(s) = T_n s - \Sigma_n(s)T_n$ for all $s\in S$.
\end{enumerate}

These definitions coincide if we are in situations (1) and (2) simultaneously. In both cases, $(\Sigma_n, \Delta_n)$ is a skew derivation on $S$ induced by $X_n$ inside $S[x; \sigma, \delta]$. This means that $S[X_n; \Sigma_n, \Delta_n]$ is a skew polynomial subring of $S[x; \sigma, \delta]$, and in fact, we always have $X_{n+1} \in S[X_n]$, so we get a chain of subrings:
$$S[x;\sigma,\delta] = S[X_0; \Sigma_0, \Delta_0] \supseteq S[X_1; \Sigma_1, \Delta_1] \supseteq \dots,$$
where each $S[X_n; \Sigma_n, \Delta_n]$ is a free $S[X_{n+1}; \Sigma_{n+1}, \Delta_{n+1}]$-module of rank $p$. 

\textbf{Note:} In the world of skew polynomial rings, related results are known for a broader range of skew derivations (compare \cite[Lemma 3.11, Theorem 3.12]{goodearl-letzter}), but these will not be useful to us in our work on skew power series rings.

We prove in \S \ref{subsec: skew power series subrings} that, under mild conditions on $w$, the ring $S^+[[x;\sigma,\delta]]$ exists if and only if $S^+[[X_n;\Sigma_n,\Delta_n]]$ exists for some $n$ (Theorem \ref{thm: skew power series subrings exist}). We also give criteria under which $S^+[[x; \sigma, \delta]]$ is Noetherian (Corollary \ref{cor: Noetherian skew power series}).

\subsection{Standard and \(\sigma\)-standard filtered artinian rings}\label{subsec: standard and sigma-standard}

To explore the deeper properties of bounded skew power series rings, we want to work with a smaller class of complete filtered rings that are algebraically and topologically well-behaved. To be precise, let $Q$ be a simple artinian $\mathbb{Z}_p$-algebra. Then we say that a filtration $u$ on $Q$ is \emph{standard} if

\begin{itemize}
\item there exists an integer $n\geq 1$ and a division ring $F$ such that $Q = M_n(F)$,
\item $F$ is the division ring of fractions of an appropriately defined (noncommutative) complete \emph{discrete valuation ring} $(D,v)$ with the naturally induced valuation $v$,
\item if $x = \begin{pmatrix}
x_{11} & \dots & x_{1n}\\
\vdots & \ddots & \vdots\\
x_{n1} & \dots & x_{nn}
\end{pmatrix}\in Q$, then $u(x) = \min_{i,j} \{v(x_{ij})\}$.
\end{itemize}

We will say that $(Q,u)$ is a \emph{standard filtered artinian ring}. More generally, if we assume instead that $Q$ is semisimple artinian, and decomposes into a direct product $Q_1\times\dots\times Q_r$ of standard filtered artinian rings $(Q_i,u_i)$, then for any automorphism $\sigma$ of $Q$, we say that $Q$ is a \emph{$\sigma$-standard filtered artinian ring} if

\begin{itemize}
    \item $r=p^n$ and $\sigma(Q_i)=Q_{i+1}$ for all $1\leq i\leq p^n$ (subscripts modulo $p^n$),
    \item $u$ is the product filtration $u:=\inf\{u_i:i\leq p^n\}$,
    \item $u_i = u_{i+1}\circ \sigma$ for all $1\leq i\leq p^n-1$.
\end{itemize}

Note that we do \emph{not} assume $u_1 = u_1 \circ \sigma^{p^n}$.

The following result generalises the well-known crossed product decomposition for skew power series rings in the Iwasawa setting where $\delta=\sigma-\id$.

\begin{letterthm}[Theorem \ref{thm: crossed}]\label{letterthm: crossed product decomposition over Q}
Suppose that $(Q,u)$ is a $\sigma$-standard filtered artinian ring, and that there exists $t\in Q$ such that $\delta(q)=tq-\sigma(q)t$ for all $q\in Q$ and $\sigma(t) = t$, so that we may define $X_r, \Sigma_r, \Delta_r, T_r$ for all $r\in\mathbb{N}$ as in \S \ref{subsec: main results}. Suppose further that 

\begin{itemize}
\item $(\sigma,\delta)$ is quasi-compatible with $u$, and

\item there exists some $n \geq 0$ such that $(\Sigma_n, \Delta_n)$ is compatible with $u$, and

\item if $\chr(Q) = 0$, then $u(t) \geq 0$.
\end{itemize}

Then, for any $m\geq 0$, we have a crossed product
$$Q^+[[x; \sigma, \delta]]_{(x-t)} = Q^+[[X_m; \Sigma_m, \Delta_m]]_{(X_m - T_m)} * (\mathbb{Z}/p^m\mathbb{Z}),$$
where elements of $\mathbb{Z}/p^m\mathbb{Z}$ act by powers of $\sigma$. If $t$ is invertible and $u(t) = u(t^{-1}) = 0$, so that $x-t$ is already a unit in $Q^+[[x; \sigma, \delta]]$, then this gives a crossed product
\[
\pushQED{\qed}
Q^+[[x; \sigma, \delta]] = Q^+[[X_m; \Sigma_m, \Delta_m]] * (\mathbb{Z}/p^m\mathbb{Z}).\qedhere
\popQED
\]
\end{letterthm}

The above crossed product decomposition will be essential to understanding the ring structure of $Q^+[[x;\sigma,\delta]]$ for many reasons: for instance, in attempting to prove that $Q^+[[x; \sigma, \delta]]$ is a prime ring, it allows us to reduce to the case where $Q$ is standard (Theorem \ref{thm: can assume M is sigma-invariant}).

\begin{letterthm}[Theorem \ref{thm: simple extension} and Theorem \ref{thm: simple when no power of sigma is inner}]\label{letterthm: simple when no power of sigma is inner}

Let $(Q,u)$ be a standard filtered artinian $\mathbb{Z}_p$-algebra carrying a commuting skew derivation $(\sigma,\delta)$ which is quasi-compatible with $u$. Suppose there exists $t\in Q$ such that $\delta(q) = tq - \sigma(q)t$ for all $q\in Q$, that $(\Sigma_m,\Delta_m)$ is compatible with $u$ for sufficiently large $m$, and no positive power of $\sigma$ is inner as an automorphism of $Q$. Then:

\begin{itemize}

\item If $\chr(Q)=p$ then the localisation $Q^+[[x;\sigma,\delta]]_{(x-t)}$ is a simple ring.

\item If $\chr(Q)=0$ and $u(t)\geq 0$ then $Q^+[[x;\sigma,\delta]]_{(x-t)}\cong \left(Q^+[[X_m; \Sigma_m, \Delta_m]]_{(X_m - T_m)}\right)^k$ for some $k,m\in\mathbb{N}$ such that $Q^+[[X_m; \Sigma_m, \Delta_m]]_{(X_m - T_m)}$ is a simple ring.\qed
\end{itemize}
\end{letterthm}

As in Theorem \ref{letterthm: crossed product decomposition over Q} above, if additionally $t$ is invertible and $u(t) = u(t^{-1}) = 0$, then $Q^+[[x;\sigma,\delta]]$ is a simple ring. In either case, $Q^+[[x; \sigma, \delta]]$ is prime.

It is interesting to compare this to similar results for skew \emph{polynomial} rings. We have a reasonable understanding of the conditions under which the skew polynomial ring $Q[x;\sigma,\delta]$ is simple: this can be read off from \cite[Corollary 2.3]{LerMat91} and \cite[Theorems 3.6 and 3.7]{LerMat92}. However, when $\delta$ is an \emph{inner} $\sigma$-derivation, say $\delta(q) = tq - \sigma(q)t$ for all $q\in Q$, then $Q[x;\sigma,\delta]$ is never simple: setting $y = x-t$ gives $Q[x; \sigma, \delta] = Q[y; \sigma]$, and the ideal generated by $y$ is proper. Hence Theorem \ref{letterthm: simple when no power of sigma is inner} stands in stark contrast to the skew polynomial case. 

\subsection{Filtered localisation}\label{subsec: localisation}

Now, let us suppose that $(R,w_0)$ is a prime, Noetherian $\mathbb{Z}_p$-algebra satisfying our hypothesis \eqref{filt}.

An important classical tool in ring theory is passing from $R$ to its (semi)simple artinian ring of quotients $Q(R)$. Any skew derivation $(\sigma,\delta)$ of $R$ extends uniquely to $Q(R)$ \cite[Lemma 1.3]{goodearl-skew-poly-and-quantized}, and when studying skew polynomial rings, we have the relation $Q(R)[x;\sigma,\delta]=Q(R)\otimes_R R[x;\sigma,\delta]$, which allows us to simplify many arguments: see e.g.\ \cite[2.3(iv), \S 3]{goodearl-letzter}.

However, there are several barriers to extending this kind of argument to the world of skew power series rings. Most importantly, our rings carry the additional data of a complete ring filtration which are essential to defining skew power series rings in this level of generality. If $(R, w_0)$ is a complete filtered prime ring, using the techniques of \cite{li-ore-sets}, it is often possible to obtain a filtration $w$ on $Q(R)$ such that $(R, w_0)\to (Q(R), w)$ is continuous. However, this will usually not be complete; worse, in \cite[\S 3]{ardakovInv}, Ardakov noted that the completion $\widehat{Q(R)}^w$ (of $Q(R)$ with respect to $w$) will usually not be semisimple, even when $R$ is a commutative integral domain.

Under suitable assumptions on $R$ and $Q(R)$, for example when $R = \O$ is an appropriate complete maximal order in $Q(R)$ and $Q(R)$ is a standard filtered artinian ring, we get that $Q(R)^+[[x;\sigma,\delta]]=Q(R)\otimes_R R[[x;\sigma,\delta]]$, so we can safely extend to $Q(R)$ as in the skew polynomial case. But in general, our initial ring $R$ will be much smaller, and this will not work.

To overcome this issue, we need to replace $\widehat{Q}:=\widehat{Q(R)}^w$ with a ($\sigma$-)standard filtered artinian ring $Q$, a semisimple artinian ring which also arises as a completion of $Q(R)$ and which carries an extension of $(\sigma,\delta)$ under certain mild conditions. In \cite[\S 3]{ardakovInv}, Ardakov outlines a technique for constructing standard filtrations $u$ on the simple quotients $Q=\widehat{Q}/M$ (where $M$ is a maximal ideal).

In \S \ref{subsec: sigma-orbit of maximal ideals and maximal orders}, we will work slightly more generally, and define $Q = \widehat{Q}/N$, where $N$ is a finite intersection of maximal ideals, making $Q$ a \emph{semisimple} artinian ring. To ensure that $\sigma$ preserves $N$, it will be convenient to take $N$ to be the intersection of a $\sigma$-orbit of maximal ideals -- in fact, it will transpire that we can take this $\sigma$-orbit to have size $p^n$ for some integer $n$, and $Q$ to be a $\sigma$-standard filtered artinian ring.

We will regularly pass from $R$ to $Q(R)$ and its completions, via several different rings and filtrations, which we explain in detail as Procedure \ref{proc: build a standard filtration} below, but which we briefly outline here as follows:
\begin{equation*}
\xymatrix{
(R,w_0)\ar@{^(->}[r]& (Q(R),w)\ar@{^(->}[r]& (\widehat{Q},w)\ar@{->>}[r]&(Q,u).
}
\end{equation*}
Here $u$ is a $\sigma$-standard filtration on $Q$. Note that since $Q(R)$ is simple, the composition $Q(R)\to Q$ is an injection, so we can regard $Q$ as the completion of $Q(R)$ with respect to $u$. However, even though the skew derivation $(\sigma, \delta)$ on $R$ will extend uniquely to a skew derivation on $Q(R)$ and on $\widehat{Q}$, it will \emph{not} necessarily descend to a skew derivation on $Q$, as we cannot ensure that $\delta$ will preserve $N=\ker(\widehat{Q}\to Q)$.

As in the statement of Theorem \ref{letterthm: simple when no power of sigma is inner}, we focus on two crucial situations throughout the paper:

\begin{enumerate}[label=(\arabic*)]
\item $\chr(R) = p$.
\item There exists $t\in \widehat{Q}$ such that $\delta(q) = tq - \sigma(q)t$ for all $q\in \widehat{Q}$ (i.e.\ $\delta$ is an \emph{inner} $\sigma$-derivation). Moreover, if $\chr(R)=0$ then $w(t) \geq 0$.
\end{enumerate}

\begin{rks*}
$ $

\begin{enumerate}[label=(\roman*)]
\item 
We have stated the assumptions on characteristic in terms of $R$ for ease of understanding, but clearly we have $\chr(R) = \chr(Q(R)) = \chr(\widehat{Q}) = \chr(Q)$.

We have stated the assumptions on $t$ in situation (2) in terms of $\widehat{Q}$ for full generality. However, in practice we will usually be concerned with the case when $\delta$ is an inner $\sigma$-derivation on $R$ defined with respect to some $t\in R$. Indeed, if $\delta(r) = tr - \sigma(r)t$ for all $r\in R$, then this is also true for all $r\in Q(R)$, and hence for all $r\in \widehat{Q}$, by the uniqueness statements of Lemma \ref{lem: compatibility is not enough}. Moreover, if $t\in R$, then $w_0(t) \geq 0 \implies w(t)\geq 0$, by Corollary \ref{cor: degree zero pieces map to degree zero pieces}.

\item Note that we may always take $t = -1$ in situation (2), as $w(-1)\geq 0$ for any filtration $w$ and $\sigma(-1)=-1$ for any automorphism $\sigma$. Taking $t = -1$, we recover the case where $\delta=\sigma-\id$, which is essential in the study of Iwasawa algebras, so we call this the \emph{Iwasawa case}.
\end{enumerate}
\end{rks*}

\begin{letterthm}[Theorem \ref{thm: bounded in the semisimple case}]\label{letterthm: restricted skew power series ring over Q}

If $(R,w_0)$ satisfies \eqref{filt} and $(\sigma,\delta)$ is compatible with $w_0$, then $(\sigma,\delta)$ extends uniquely to a quasi-compatible commuting skew derivation on $(\widehat{Q},w)$. Moreover, if $(Q,u)$ is a $\sigma$-standard filtered artinian quotient of $\widehat{Q}$ and $\delta$ preserves $\ker(\widehat{Q}\to Q)$, then $(\sigma,\delta)$ induces a canonical commuting skew derivation on $Q$, which restricts to $(\sigma,\delta)$ on $Q(R)$. In this case, and assuming that we are in situations (1) or (2) above:
\begin{enumerate}[label=(\roman*)]
\item $(\sigma,\delta)$ is quasi-compatible with $u$, and the bounded skew power series ring $Q^+[[x;\sigma,\delta]]$ is a well-defined, Noetherian ring.
\item If $Q^+[[x;\sigma,\delta]]$ is prime, then $R^+[[x;\sigma,\delta]]$ is prime.\qed
\end{enumerate}

\end{letterthm}

This result demonstrates how $\sigma$-standard filtered artinian completions can play a similar role in the skew power series world to that played by the Goldie ring of quotients $Q(R)$ in the skew polynomial world. Using this result, and our results on skew power series rings over $\sigma$-standard filtered artinian rings from \S \ref{subsec: standard and sigma-standard}, we can prove our final main theorem, which we can regard as the most important result of the paper.

\begin{letterthm}[Theorem \ref{thm: infinite order implies prime}]\label{letterthm: infinite order implies prime}

Suppose $(R,w_0)$ satisfies \eqref{filt} and $(\sigma,\delta)$ is compatible with $w_0$, and suppose that we are simultaneously in situations (1) and (2) above. Let $(Q,u)$ be any $\sigma$-standard filtered artinian completion of $Q(R)$. If no positive power of $\sigma$ is an inner automorphism of $Q$, then $R^+[[x;\sigma,\delta]]$ is a prime ring.\qed

\end{letterthm}

Note that, as we are in situation (2), $\delta$ always preserves $\ker(\widehat{Q}\to Q)$ by Lemma \ref{lem: when delta preserves N}, so $(\sigma,\delta)$ induces a canonical commuting skew derivation on $Q$, and the above theorem makes sense.

\section{Preliminaries}

\subsection{Filtrations and filtration topologies}\label{subsec: filtrations}

Our standard reference for ring filtrations is \cite{LVO}, though we define our notation carefully here to set the conventions for this paper: in particular, the reader should note that filtrations as defined in \cite{LVO} are \emph{ascending} (in the sense that $F_aR \subseteq F_{b}R$ when $a\leq b$), whereas we adopt the opposite convention. Throughout this paper, we use the word \emph{filtration} to denote a separated, descending filtration. A full definition is given below.

Let $\Gamma$ be a discrete subgroup of $\mathbb{R}$. Throughout most of this paper, we will take $\Gamma = \mathbb{Z}$, but it will occasionally be useful to allow for other subgroups too.

\begin{defn}\label{defn: filtrations}
$ $

\begin{itemize}
\item Let $A$ be an abelian group under addition. A function $u: A\to \Gamma\cup\{\infty\}$ is said to be a \emph{$\Gamma$-filtration} on $A$ if, for all $a,b\in A$,
\begin{enumerate}[label=(\roman*),noitemsep]
\item $u(a+b) \geq \min\{u(a),u(b)\}$, and
\item $u(a) = \infty$ if and only if $a = 0$.
\end{enumerate}
We will sometimes say that $(A,u)$ is a \emph{$\Gamma$-filtered abelian group}. Note that condition (ii) states that the filtration is \emph{separated}, which is not a requirement for some authors.
\item \cite[Chapter I, Definition 2.1]{LVO} Let $R$ be a ring. A function $u: R\to \Gamma\cup\{\infty\}$ is said to be a \emph{(ring) $\Gamma$-filtration} on $R$, and $(R,u)$ is said to be a \emph{$\Gamma$-filtered ring}, if $(R,u)$ is a $\Gamma$-filtered abelian group which additionally satisfies $u(1) = 0$ and $u(rs) \geq u(r) + u(s)$ for all $r,s\in R$.
\item \cite[Chapter I, Definition 2.2]{LVO}
Let $(R,u)$ be a filtered ring, and $M$ a left $R$-module. A function $v: M\to \Gamma\cup\{\infty\}$ is said to be a \emph{(module) $\Gamma$-filtration} on $M$, and $(M,v)$ is said to be a \emph{$\Gamma$-filtered $(R,u)$-module}, if $(M,v)$ is a $\Gamma$-filtered abelian group which additionally satisfies $v(rm) \geq u(r) + v(m)$ for all $m\in M$ and $r\in R$.
\end{itemize}
In all cases, we may leave the filtration tacitly understood and simply say ``$M$ is a $\Gamma$-filtered $R$-module", and so on. Moreover, when $\Gamma = \mathbb{Z}$, we will say simply \emph{filtered} (resp. \emph{filtration}) instead of $\mathbb{Z}$-filtered (resp. $\mathbb{Z}$-filtration).
\end{defn}

\begin{rk}\label{rk: level sets for ring filtration}
\cite{LVO} prefers the following equivalent way of defining a filtration. Let $R$ be a ring, and suppose we have a sequence $FR = \{F_nR\}_{n\in\mathbb{Z}}$ of additive subgroups of $R$ satisfying $1\in F_0 R$ and such that, for all $m,n\in\mathbb{Z}$, we have both $F_nR \supseteq F_{n+1}R$, $(F_mR)(F_nR) \subseteq F_{m+n}R$, and $\bigcap_{n\in\mathbb{Z}} F_nR = \{0\}$. Such a sequence defines the filtration $u: x\mapsto \max\{n\in\mathbb{Z} : x\in F_n R\}$ for all $x\neq 0$ (and $u(0) = \infty$). Conversely, a filtration $u$ gives rise to such a sequence $FR$ by setting $F_nR = u^{-1}([n, \infty])$.

When $u$ and $FR$ are associated in this way, we will call $FR$ the sequence of \emph{level sets} associated to $u$.

We could also have defined a filtration on an abelian group $A$ as a sequence $FA = \{F_n A\}_{n\in\mathbb{Z}}$ of subgroups satisfying $F_nA \supseteq F_{n+1}A$ and $\bigcap_{n\in\mathbb{Z}} F_nA = \{0\}$, and a filtration on an $R$-module $M$ as a sequence $FM = \{F_n M\}_{n\in\mathbb{Z}}$ of subgroups satisfying $F_nM \supseteq F_{n+1}M$, $\bigcap_{n\in\mathbb{Z}} F_nM = \{0\}$, and $(F_mR)(F_nM) \subseteq F_{m+n}M$.
\end{rk}

\begin{defn}\label{defn: filtration topology}
Let $A$ be a $\Gamma$-filtered abelian group with level sets $\{F_nA\}_{n\in\mathbb{Z}}$. This defines the \emph{filtration topology} on $A$: a subset $U\subseteq A$ is defined to be open if and only if, for all $x\in U$, there exists $n\in\Gamma$ such that $x + F_nA\subseteq U$.

This means that, for all $x\in A$, the sets $\{x + F_n A\}_{n\in\Gamma}$ form a neighbourhood base for $x$. As $\Gamma$ is countable, in particular this makes $A$ a first-countable topological space.
\end{defn}

As ($\Gamma$-)filtered rings and ($\Gamma$-)filtered modules are also ($\Gamma$-)filtered abelian groups under addition, we will define the filtration topology on them in the same way, and the following definitions and results will also hold for them.

\begin{defn}\label{defn: cauchy, convergent, complete}
Let $(A,u)$ be a $\Gamma$-filtered abelian group. We make the following usual definitions.

\begin{itemize}
\item The sequence $(a_i)_{i\in\mathbb{N}}$ of elements of $A$ is \emph{Cauchy} if, for all $K\in\Gamma$, there exists $N\in\mathbb{N}$ such that $i,j \geq N \implies u(a_i - a_j) \geq K$.
\item The sequence $(a_i)_{i\in\mathbb{N}}$ \emph{converges} to $a\in A$ if, for all $K\in\Gamma$, there exists $N\in\mathbb{N}$ such that $i\geq N \implies u(a_i - a) \geq K$.
\item The infinite sum $\sum_{i=0}^\infty a_i$ \emph{converges} if its sequence of partial sums $\left(\sum_{i=0}^n a_i\right)_{n\in\mathbb{N}}$ converges.
\item $(A,u)$ is \emph{complete} if all Cauchy sequences converge in $A$.
\end{itemize}

Basic properties of convergence and complete filtered abelian groups can be found in \cite[Chapter I, \S 3.3]{LVO}. A treatment of similar material, phrased in terms of \emph{normed rings}, can be found in \cite[\S\S 6.1--6.2]{DDMS}.
\end{defn}

\begin{rk}\label{rk: filtration topology is a ring or module topology}
Openness is preserved under translation: that is, given any $a\in A$, the subset $U\subset A$ is open if and only if $U+a := \{u+a:u\in U\}\subseteq A$ is open.

In fact, the filtration topology on a filtered ring is a \emph{ring topology} in the sense of \cite[Definition 1.1]{warner}, and the filtration topology on a filtered module is a \emph{module topology} in the sense of \cite[Definition 2.1]{warner}: see e.g.\ \cite[Chapter I, Property 3.1(e)]{LVO}.
\end{rk}

Next, we turn to studying morphisms. (All of the below could be done for $\Gamma$-filtrations, but as we will not need this in this paper, we restrict simply to filtrations for ease of notation.)

We recall the following basic fact.

\begin{lem}\label{lem: continuous homomorphisms}
Let $(A,u)$ and $(B,v)$ be filtered abelian groups, with level sets $FA$ and $FB$ respectively. If $f: A\to B$ is a group homomorphism, then $f$ is continuous if and only if, for all $q\in\mathbb{Z}$, there exists $p\in\mathbb{Z}$ such that $F_pA \subseteq f^{-1}(F_q B)$.
\end{lem}

\begin{proof}
One direction follows immediately from the definitions: since $F_qB$ is open in $B$ by definition, we must have that $f^{-1}(F_qB)$ is open in $A$ for any continuous function $f$, and $0\in f^{-1}(F_qB)$ as $f$ is a group homomorphism. In particular, there exists $p\in\mathbb{Z}$ such that $0 + F_pA \subseteq f^{-1}(F_qB)$.

To prove the other direction, we start with an open set $V\subseteq B$. If $U := f^{-1}(V)\neq\varnothing$, take arbitrary $a\in U$. Then we will have $f(a) \in V$, so $V' = V - f(a)$ is an open set of $B$ (by Remark \ref{rk: filtration topology is a ring or module topology}) containing $0$. Hence, by definition, there exists $q\in \mathbb{Z}$ such that $F_qB \subseteq V'$, and our assumption now implies that there exists $p\in\mathbb{Z}$ such that $F_pA \subseteq f^{-1}(F_qB) \subseteq f^{-1}(V')$. Hence $a + F_pA \subseteq a + f^{-1}(V')$, and since $f(a + f^{-1}(V')) = V$ we must have $a + F_pA\subseteq U$ as required. As $a\in U$ was arbitrary, $U$ is open.
\end{proof}

\begin{defn}\label{defn: filtered and strict morphisms}
Let $(A,u)$, $(B,v)$ be filtered abelian groups and $f: A\to B$ a group homomorphism. Following \cite[Chapter I, \S 2.5]{LVO}:

\begin{enumerate}[label=(\roman*)]
\item $f$ is \emph{filtered} if there exists $d\in\mathbb{Z}$ such that $f(F_nA) \subseteq F_{n+d}B$ for all $n\in \mathbb{Z}$. If $d$ is the \emph{largest} integer such that $f(F_nA) \subseteq F_{n+d}B$ for all $n\in \mathbb{Z}$, we say that $d$ is the \emph{degree} of $f$, and write $d = \deg(f)$. 
\item $f$ is \emph{strictly filtered} if $f(F_nA) = f(A) \cap F_nB$.
\item If $A = B$ and $u = v$, and $f: (A,u) \to (A,u)$ is filtered, we will say that $f$ is a \emph{filtered endomorphism} of $(A,u)$. In this case we may say that the degree of $f$ is taken \emph{with respect to $u$}, and denote it $\deg_u(f)$ or similar. This will avoid confusion when there are several filtrations on $A$ under consideration.
\end{enumerate}
\end{defn}

These properties are related by

\vspace{12pt}
\centerline{
strictly filtered$\implies$filtered of degree 0$\implies$filtered$\implies$continuous,
}

where the final implication follows from Lemma \ref{lem: continuous homomorphisms}. Note that, if $f$ is a strictly filtered \emph{automorphism}, then $f^{-1}$ is also strictly filtered.

\textbf{Warning.} In addition to our choice to use descending filtrations, our terminology differs from that of \cite{LVO} in two ways:
\begin{itemize}
\item In \cite[Chapter I, \S 2.5]{LVO}, the term ``filtered" has two different meanings: they define \emph{filtered (ring) homomorphisms} (which allow $\deg(f)\in\mathbb{Z}$ to be arbitrary) and \emph{filtered morphisms} (which place a restriction on the allowed values of $\deg(f)$). Our usage is consistent with the former: for us, $\deg(f)\in\mathbb{Z}$ can be arbitrary.
\item In \cite[Chapter I, \S 2.5]{LVO}, a filtered (homo)morphism can have many different degrees, whereas for us the degree of a filtered map is unique.
\end{itemize}

\begin{rks}\label{rks: filtered and strict linear maps}
$ $

\begin{enumerate}[label=(\roman*)]
\item If $A, B, C$ are filtered abelian groups and $f: A\to B$ and $g: B \to C$ are filtered linear maps, then $gf: A\to C$ is filtered, and $\deg(gf) \geq \deg(g) + \deg(f)$.

If $f_1, f_2: A\to A'$ are filtered linear maps, then $f_1\pm f_2: A\to A'$ are filtered, and $\deg(f_1 \pm f_2) \geq \min\{\deg(f_1), \deg(f_2)\}$.
\item Let $(R,u)$ be a filtered ring. Then inner automorphisms are always filtered: indeed, if we take $a\in R^\times$ and define $\eta(r) = ara^{-1}$ for all $r\in R$, then $u(\eta(r)) \geq u(r) + u(a) + u(a^{-1})$, so $\deg_u(\eta) \geq u(a) + u(a^{-1})$.

Similarly, suppose $\sigma$ is a filtered automorphism of $R$, $t$ is a fixed element of $R$, and $\delta$ is the linear endomorphism of $R$ defined by $\delta(r) = tr - \sigma(r)t$ for all $r\in R$. (Later, such a map $\delta$ will be called an \emph{inner $\sigma$-derivation}.) Then $\delta$ is always filtered: for any $r\in R$, we have $u(\delta(r)) \geq \min\{u(t) + u(r), u(\sigma(r)) + u(t)\}$, so $\deg_u(\delta) \geq u(t) + \min\{0, \deg_u(\sigma)\}$.
\item Suppose $f: (A,u)\to (B,v)$ is strictly filtered. If $y\in B$ is in the image of $f$ with $v(y) = n$, then there exists $x\in A$ with $f(x) = y$ such that $u(x) = n$. Put another way, if $y\in B$ is in the image of $f$, then $v(y) = \max\{u(x): x\in f^{-1}(y)\}$.

When $f$ is injective, this says that $v(f(x)) = u(x)$ for all $x\in A$. However, in general, we can only conclude that $v(f(x)) = \max\{u(x+k) : k\in\ker(f)\} \geq u(x)$ for general $x\in A$: that is, $v|_{f(A)}$ is the quotient filtration on $f(A)$ induced by $u$.

If $f$ is surjective, and $\varphi_A, \varphi_B$ are filtered endomorphisms of $A$ and $B$ respectively satisfying $\varphi_B f = f\varphi_A$, then we can see from the above that $\deg_v(\varphi_B) \geq \deg_u(\varphi_A)$.
\end{enumerate}
\end{rks}

As mentioned in the introduction, we will often have to deal with simple artinian rings $Q$. By the classical Artin-Wedderburn theory, there exists a set of matrix units in $Q$ giving an isomorphism $Q \cong M_s(F)$, for a determined integer $s\geq 1$ and division ring $F$. Throughout this paper, $M_s(F)$ will denote the $s\times s$ matrix ring over $F$ together with a fixed choice of matrix units called the \emph{standard} matrix units \cite[Example 2.2.11(iii)]{MR}: we will often fix a set of matrix units and write $Q = M_s(F)$ where this will not cause any difficulties.

\begin{defn}\label{defn: matrix filtration}
Let $(R, u)$ be a filtered ring and $s\geq 1$ an integer. Then the filtration $M_s(u)$ on $M_s(R)$ is defined to be the map
$$\sum_{1\leq i,j\leq s} a_{ij} e_{ij} \mapsto \min_{1\leq i,j\leq s} \{u(a_{ij})\},$$
where $\{e_{ij}\}_{1\leq i,j\leq s}$ are the standard matrix units of $M_s(R)$ \cite[Example 2.2.11(iii)]{MR}.
\end{defn}

\subsection{Discrete valuation rings and standard filtrations}\label{subsec: DVRs}

Beginning with a filtered ring $(R, w)$ satisfying \eqref{filt}, we hope to construct a simple artinian filtered ring $(Q, u)$ which is particularly well-behaved. In this subsection, we make this precise.

\begin{defn}\label{defn: DVR}
For the purposes of this paper, and following \cite[\S 3.6]{ardakovInv}, we will say that a \emph{discrete valuation ring} is a Noetherian domain $D$ which is not a division ring, with the property that, for every non-zero $x\in Q(D)$ (its Goldie division ring of quotients), we have either $x\in D$ or $x^{-1}\in D$.
\end{defn}

\begin{props}\label{props: DVRs}
Suppose $D$ is a discrete valuation ring. Most of the following facts were proved in \cite[\S 2.5]{jones-woods-2}.
\begin{enumerate}[label=(\roman*)]
\item $D$ is a (scalar) local ring. This implies that $D^\times = D\setminus J(D)$, and so in particular, if $x\in Q(D)$ but $x\not\in D$, then $x^{-1}\in J(D)$.
\item The Jacobson radical is given by $J(D) = \pi D$ for some normal element $\pi$, and every non-zero left or right ideal has the form $\pi^n D=J(D)^n$ for some $n\geq 0$. It follows that $\bigcap_{n\in\mathbb{N}} \pi^n D = \{0\}$. Indeed, we can take $\pi$ to be any element of $J(D)\backslash J(D)^2$, and we call such an element $\pi$ a \emph{uniformiser} for $D$. Then $J(D)$ is easily seen to be invertible, with inverse $J(D)^{-n} = \pi^{-n} D \subseteq Q(D)$ for all $n$.
\item For every $x\in Q(D)$, $x\in J(D)^n$ for some $n\in\mathbb{Z}$, i.e.\ $Q(D)=\bigcup_{n\in\mathbb{Z}} \pi^n D$.
\end{enumerate}
\end{props}

By default, we will assume that all discrete valuation rings $D$ are endowed with their associated $J(D)$-adic filtration $v$. (This is a \emph{valuation} in the sense that $v(ab) = v(a) + v(b)$ for all $a,b\in D$.)

\begin{defn}\label{defn: standard filtrations}
$ $

\begin{enumerate}[label=(\roman*)]

\item Suppose $D$ is a discrete valuation ring (complete or otherwise), and $F=Q(D)$ is its division ring of quotients. Then the \emph{induced} filtration $v$ on $F$ is defined as follows: for all $0\neq x\in F$, we have $v(x) = n$ if and only if $n\in\mathbb{Z}$ is maximal such that $x\in J(D)^n$. Of course, this restricts to the $J(D)$-adic filtration defined on $D$, and since $\pi$ is normal and $D$ is a domain, $\gr_v(F)$ is isomorphic to a skew Laurent polynomial ring $(D/J(D))[\overline{\pi},\overline{\pi}^{-1}; \alpha]$.

On the other hand, if $F$ is any division ring, we will say that a filtration $v$ on $F$ is \emph{standard} if it is induced from some complete discrete valuation ring $D\subseteq F$ satisfying $Q(D) = F$.
\item If $Q$ is a simple artinian ring, we will say that a filtration $u$ on $Q$ is \emph{standard} if there exists an isomorphism $\iota: Q\to M_s(F)$ such that $u = M_s(v) \circ \iota$ for some standard filtration $v$ on $F$.

Where possible, we will omit mention of the isomorphism $\iota$ and simply write $Q = M_s(F)$ and $u = M_s(v)$, though it will sometimes be necessary to retain it when there are multiple standard filtrations under consideration.
\end{enumerate}
\end{defn}

\begin{rks}\label{rks: standard filtrations are complete and have unique discrete valuation rings}
$ $

\begin{enumerate}
\item Note that we have defined standard filtrations to be complete.
\item If $F$ is a division ring equipped with a standard filtration $f$, then the complete discrete valuation ring $D$ that induces $f$ is \emph{unique}, and can be recovered as the $f$-positive subring of $F$, i.e.\ $f^{-1}([0,\infty])$.

If $Q = M_s(F)$ is a simple artinian ring equipped with a standard filtration $u = M_s(v)$ as above, then its $u$-positive part is $\O = M_s(D)$. This is a maximal order in $Q$.
\end{enumerate}
\end{rks}

If $R$ is a ring admitting two filtrations, $w_1$ and $w_2$, then we say that $w_1$ and $w_2$ are \emph{topologically equivalent} if the identity map $\id:R\to R$ is continuous both as a map $(R, w_1)\to (R, w_2)$ and as a map $(R, w_2)\to (R, w_1)$. This is equivalent to the definition given in \cite[Chapter I, \S 3.2]{LVO} by Lemma \ref{lem: continuous homomorphisms}.

\begin{lem}\label{lem: D is uniquely determined by topology}
Let $F$ be a division ring admitting two standard filtrations $f_1$ and $f_2$. Then $f_1$ and $f_2$ are topologically equivalent if and only if $f_1 = f_2$.
\end{lem}

\begin{proof}
Write $D_1$ and $D_2$ for the $f_1$- and $f_2$-positive subrings of $F$ respectively. By Remark \ref{rks: standard filtrations are complete and have unique discrete valuation rings}.2, $f_1 = f_2$ if and only if $D_1 = D_2$. Suppose $D_1 \neq D_2$: then (without loss of generality) there exists some element $0\neq x\in D_1 \setminus D_2$. By Property \ref{props: DVRs}(i), it follows that $x^{-1} \in J(D_2) \setminus J(D_1)$: then the sequence $(x^{-n})_{n\in \mathbb{N}}$ converges to $0$ under $f_2$ but diverges under $f_1$, so we cannot have $f_1$ topologically equivalent to $f_2$.
\end{proof}

\begin{rk}
Note that this is not true for general simple artinian rings $Q$, which can admit two distinct but topologically equivalent standard filtrations: see Example \ref{ex: sigma does not preserve maximal orders}. However, we show here that the discrete valuation ring determining any given standard filtration on $Q$ is unique up to isomorphism.

If $Q$ is a simple artinian ring, then $Q\cong M_s(F)$, and the division ring $F$ is uniquely determined up to isomorphism as the endomorphism ring of a minimal right ideal, as in the usual Wedderburn theory. If $Q$ is additionally equipped with a standard filtration $u$, then there exists an isomorphism $\iota: Q\to M_s(F)$ such that $u = M_s(v)\circ \iota$ for a standard filtration $v$ on $F$, which determines a discrete valuation ring $D$.

Now suppose additionally that we have an isomorphism $\iota': Q\to M_s(F)$, and that $u = M_s(v')\circ \iota'$ for some standard filtration $v'$ on $F$ determining a discrete valuation ring $D'$. Then $M_s(v') \circ \iota'\iota^{-1}$ is equal to $M_s(v)$. Write $\iota'\iota^{-1} = M_s(\tau)\circ \eta$, where $\tau$ is an automorphism of $F$ and $\eta$ is an inner automorphism of $M_s(F)$ by \cite[Lemma 2.2, Theorem 2.4]{cauchon-robson}: this means that $M_s(v'\circ \tau) \circ \eta = M_s(v)$, and so $M_s(v'\circ \tau)$ and $M_s(v)$ are topologically equivalent by Remark \ref{rks: filtered and strict linear maps}(ii), which finally implies that $v'\circ \tau$ and $v$ must be topologically equivalent standard filtrations on $F$. But now they are equal by Lemma \ref{lem: D is uniquely determined by topology}, and so $\tau$ restricts to an isomorphism $\tau: D\to D'$.
\end{rk}

\subsection{Skew derivations and compatibility}\label{subsec: compatibility}

\begin{defn}\label{defn: skew derivation}
Let $R$ be a ring. We will say that the pair $(\sigma, \delta)$ is a \emph{skew derivation} on $R$ if $\sigma\in\Aut(R)$ and $\delta$ is a \emph{(left) $\sigma$-derivation}, i.e.\ a linear endomorphism of $R$ such that $\delta(rs) = \delta(r) s + \sigma(r)\delta(s)$ for all $r,s\in R$. (Note that some authors allow $\sigma$ to be a more general endomorphism of $R$.)

We say that the skew derivation $(\sigma,\delta)$ is \emph{commuting} if $\sigma\circ\delta=\delta\circ\sigma$, which we will usually assume.
\end{defn}

Given a skew derivation $(\sigma, \delta)$ on $R$, we can define the \emph{skew polynomial ring} $R[x; \sigma, \delta]$: this is equal to $R[x]$ as a left $R$-module, and its multiplication is uniquely determined by the multiplication in $R$ and the relations $xr = \sigma(r)x + \delta(r)$ for all $r\in R$. This construction encompasses many important and classical ring constructions, such as Weyl algebras, certain quantum groups \cite{BroGoo02}, enveloping algebras of soluble complex Lie algebras, and (after suitable localisations) group algebras of poly-(infinite cyclic) groups. See e.g.\ \cite[\S 1.2]{MR} or \cite[\S 2]{GooWar04} for some of the basic properties of this construction.

However, in general, it is not possible to define skew power series rings without further constraints on $(\sigma, \delta)$, due to potential convergence issues. We begin with the following definition:

\begin{defn}\label{defn: compatibility} \cite[Definition 2.8]{jones-woods-2}
Let $(R,w)$ be a filtered ring. Then the skew derivation $(\sigma, \delta)$ is \emph{(strongly) compatible} with $w$ if $w(\sigma(r) - r) > w(r)$ and $w(\delta(r)) > w(r)$ for all $0\neq r\in R$. 

That is, $\sigma-\id$ and $\delta$ are filtered linear endomorphisms satisfying $\deg_w(\sigma-\id) > 0$ and $\deg_w(\delta) > 0$. Notice that $w(\sigma(r) - r) > w(r)$ implies $w(\sigma(r)) = w(r)$, so that $\sigma$ is \emph{strictly} filtered, from which we can conclude also that $\deg_w(\sigma^{-1}-\id) > 0$.

The condition $w(\sigma(r) - r) > w(r)$ is sometimes weakened to the condition $w(\sigma(r)) = w(r)$: we call this \emph{weak compatibility} \cite[Definition 2.6]{jones-woods-2}, but we will not use this notion in this paper, so whenever we refer to a skew derivation being \emph{compatible}, we mean that it is strongly compatible.
\end{defn}

\begin{rk*}
As far as we are aware, all definitions of skew power series rings (in this context) in the literature so far rely on some notion of compatibility with a filtration: this definition of (strong) compatibility comes from \cite[Definition 1.8]{jones-woods-1}, but see \cite[Definition 2.6]{jones-woods-2}, \cite[Setup 3.1(5)]{letzter-noeth-skew}, \cite[\S 1]{SchVen06} for various related notions. However, none of these definitions will be enough for our eventual purposes, so we will extend this definition in \S \ref{subsec: restricted skew power series} below.
\end{rk*}

For now, we give the full definition of skew power series rings over positively filtered rings mentioned in the introduction.

\begin{defn}\label{defn: skew power series ring, compatible case}
Let $(R,w)$ be a complete, \emph{positively} filtered ring, i.e.\ suppose that $w$ takes values in $\mathbb{N}\cup\{\infty\}$. Let $(\sigma, \delta)$ be a compatible skew derivation. Then we define the \emph{skew power series ring}
$$R[[x;\sigma,\delta]] := \left\{ \sum_{n\geq 0} r_n x^n : r_n\in R\right\}.$$
If we choose any rational $\varepsilon > 0$ and set $\Gamma = \mathbb{Z} + \varepsilon \mathbb{Z}$, then we can make $R[[x; \sigma, \delta]]$ into a complete $\Gamma$-filtered left $R$-module with complete $\Gamma$-filtration $v_\varepsilon: R[[x; \sigma, \delta]] \to \Gamma\cup\{\infty\}$ defined by
$$v_\varepsilon\left(\sum_{n\geq 0} r_n x^n\right) = \inf_{n\geq 0} \{w(r_n) + \varepsilon n\},$$
and with continuous multiplication extending the multiplication on $R$, uniquely defined by the family of relations $xr = \sigma(r)x + \delta(r)$ for all $r\in R$. (In fact, if $0 < \varepsilon \leq 1$, then $v_\varepsilon$ is a \emph{ring} $\Gamma$-filtration in this context, but we do not rely on this fact.)
\end{defn}

\begin{rk*}
It is non-trivial to prove that this does in fact define a continuous multiplication on $R[[x; \sigma, \delta]]$. However, as this notion is now well-established in the literature, and in any case it will follow from results in \S \ref{subsec: restricted skew power series} below, we omit the proof at this stage. Proofs of the existence of skew power series rings in very similar contexts can be found in \cite[\S 3.4]{letzter-noeth-skew} or \cite[Lemma 2.1]{venjakob}.
\end{rk*}

The following theorem highlights the important role of \emph{standard} filtrations in our work.

\begin{thm}\label{thm: simple artinian with compatible skew derivation}
Suppose $Q$ is a simple artinian ring with a standard filtration $u$, and let $\O = u^{-1}([0,\infty])$ be the associated maximal order. Suppose that $(\sigma, \delta)$ is a skew derivation on $Q$ compatible with $u$. Then the skew power series ring $\O[[x; \sigma, \delta]]$ exists and is prime.
\end{thm}

\begin{proof}
It follows immediately from the definitions that $(\sigma, \delta)$ restricts to a skew derivation on $\O$, and is still compatible with $u|_\O$, so the skew power series ring $\O[[x; \sigma, \delta]]$ exists by the remark above.

Let $I$ and $J$ be non-zero ideals of $\O[[x; \sigma, \delta]]$. Since $\O[x; \sigma, \delta]$ is a prime ring by \cite[Theorem 1.2.9(iii)]{MR}, it will suffice to prove that $I\cap \O[x; \sigma, \delta]$ and $J\cap \O[x; \sigma, \delta]$ are non-zero, as this will ensure that their product is non-zero and contained in $IJ$. But this is \cite[Theorem C]{jones-woods-2}.
\end{proof}

We now look specifically at \emph{inner} $\sigma$-derivations, which will play a large role in this paper.

\begin{defn}\label{defn: inner sigma-derivation}
Let $R$ be a ring and $(\sigma, \delta)$ a skew derivation on $R$. We say that $\delta$ is \emph{inner} if there exists an element $t\in R$ such that $\delta(r)=tr-\sigma(r)t$ for all $r\in R$.
\end{defn}

We will be primarily interested in rings $Q$ admitting standard filtrations (Definition \ref{defn: standard filtrations}(ii)), and by definition these are simple artinian. However, from \S \ref{subsec: sigma-orbit of maximal ideals and maximal orders} onwards, we will sometimes need to pass to more general semisimple artinian rings, and in these cases the following lemma simplifies the situation considerably:

\begin{lem}\cite[Lemma 1.4]{cauchon-robson}\label{lem: if Q is not simple, delta is inner}
If $(\sigma, \delta)$ is a skew derivation on $Q$, where $Q$ is a semisimple artinian ring that is not simple, and $\sigma$ permutes the simple factors of $Q$ transitively, then $\delta$ is inner.\qed
\end{lem}

\subsection{Filtered \texorpdfstring{\(\mathbb{Z}_p\)}{Zp}-algebras}

\begin{defn}\label{defn: filtered Z_p-algebra}
We will say that $(R,w)$ is a \emph{filtered $\mathbb{Z}_p$-algebra} if it is a filtered ring and a $\mathbb{Z}_p$-algebra, and the algebra map $(\mathbb{Z}_p, v_p)\to (R,w)$ is filtered of non-negative degree.
\end{defn}

\begin{lem}\label{lem: p-adic facts}
Suppose that $(R,w)$ is a filtered $\mathbb{Z}_p$-algebra. Then

\begin{enumerate}[label=(\roman*)]
\item $\displaystyle w(n) \geq v_p(n)$ for all $n\in\mathbb{Z}_p$.
\item $\displaystyle w\left(\binom{p^n}{i}\right) \geq n - v_p(i)$ for all $n\in\mathbb{N}$ and all $0\leq i\leq p^n$.
\item $\displaystyle w\left(\binom{p^n}{i}\right) \geq n - \lfloor \log_p(N) \rfloor$ for all $n\in\mathbb{N}$ and all $0\leq i\leq N\leq p^n$.
\end{enumerate}
\end{lem}

\begin{proof}
$ $

\begin{enumerate}[label=(\roman*)]
\item Clearly $\displaystyle w(p^k) \geq kw(p)\geq k$ for all $k\in\mathbb{N}$. Also, by definition we have $w(1) = 0$, and so for all $i\geq 1$ we have 
$\displaystyle w(i) = w(\underbrace{1+\dots+1}_{i \text{ times}}) \geq w(1) = 0$.

So now we will prove that $w(1) = w(2) = \dots = w(p-1) = 0$. Suppose not: then there is some $2\leq m\leq p-1$ minimal such that $w(m) \geq 1$. Choose the integer $\ell\geq 1$ such that $\ell m < p < (\ell+1)m$, and rearrange to see that $1\leq p - \ell m < m$. But then $w(p - \ell m) \geq \min\{w(p), w(\ell) + w(m)\} \geq 1$, contradicting the minimality of $m$.

This is enough, as every element $n\in \mathbb{Z}_p$ can be written as $n = c_0 + c_1 p + c_2 p^2 + \dots$ for $0\leq c_i \leq p-1$, and so $w(n) \geq \min\{i : c_i \neq 0\} = v_p(n)$.
\item We can calculate $v_p\left(\binom{p^n}{i}\right) = v_p((p^n)!) - v_p(i!) - v_p((p^n-i)!)$. But $v_p(\ell) = v_p(p^n-\ell)$ for all $1\leq \ell\leq p^n-1$, so
$$v_p((p^n-i)!) = v_p(i) + v_p(i+1) + \dots + v_p(p^n-1) = v_p((p^n-1)!) - v_p((i-1)!).$$
It follows that
$$v_p\left(\binom{p^n}{i}\right) = v_p((p^n)!) - v_p((p^n-1)!) - v_p(i!)  + v_p((i-1)!) = n - v_p(i),$$
and now the result follows from (i).
\item Follows from (ii) using the fact that $v_p(i) \leq \lfloor \log_p(i) \rfloor$ for all $i\geq 0$.\qedhere
\end{enumerate}
\end{proof}

\subsection{Crossed products}\label{subsec: defn of crossed product}

\begin{defn}\label{defn: crossed product}
Let $R$ be a ring and $G$ a group. The ring $S$ is said to be a \emph{crossed product} of $R$ with $G$ if
\begin{enumerate}[label=(\roman*)]
\item $R$ is a subring of $S$,
\item there is an injective map (of sets) $G\to S$, with image $\overline{G} = \{\overline{x} : x\in G\}$ (the \emph{basis} of the crossed product), such that $S = \bigoplus_{x\in G} R\overline{x}$ as left $R$-modules,
\item there exist a map of sets $\tau: G\times G\to R^\times$ (the \emph{twisting} map) and a map of sets $\alpha: G\to \Aut(R)$ (the \emph{action} map) which define the multiplication in $S$ so that, for all $x,y\in G$ and all $r\in R$, we have
$$\overline{xy} = \tau(x,y)\, \overline{x}\, \overline{y} \qquad \qquad \text{and} \qquad\qquad \overline{x}r = r^{\alpha(x)}\overline{x}.$$
\end{enumerate}
\end{defn}

We will write such a crossed product as $R * G$, omitting mention of $\tau$ and $\alpha$, as we hope that they will always be clear from context.

Our standard reference for crossed products is Passman \cite{passmanICP}, though note that the convention in \cite{passmanICP} is that $S$ is a \emph{right} $R$-module, with corresponding modifications in the definitions of $\tau$ and $\alpha$.

\section{Rings and subrings of bounded skew power series}\label{sec: filtered localisation}

Let $(R, w_0)$ be a filtered ring. In \S \ref{subsec: compatibility}, we defined the skew power series ring $R[[x; \sigma, \delta]]$ in the case where $(\sigma, \delta)$ is compatible with $w_0$ and $w_0$ is a positive filtration. Under the further assumption that $(R, w_0)$ satisfies \eqref{filt}, we briefly outlined in \S \ref{subsec: localisation} the procedure to pass from $(R, w_0)$ to the far nicer filtered ring $(Q, u)$. We will elaborate on this procedure in \S  \ref{sec: sigma-orbits of maximal ideals}, but unfortunately, when we do this, the skew derivation on $Q$ induced by $(\sigma, \delta)$ (when it exists) may no longer be compatible with $u$. 

In this section, we will weaken the notion of compatibility to ensure that some kind of skew power series ring will still exist. To avoid any confusion, we will replace $(R,w_0)$ and $(Q,u)$ with the more general complete, filtered ring $(S,w)$ throughout.

\subsection{Quasi-compatible skew derivations}

\begin{defn}\label{defn: strongly bounded}
Let $(S, w)$ be a filtered ring. Suppose $(\sigma, \delta)$ is a commuting skew derivation on $S$.

\begin{enumerate}[label=(\roman*)]
\item Suppose $X$ is a set of linear endomorphisms of $S$ which are filtered (Definition \ref{defn: filtered and strict morphisms}(i)) with respect to $w$, so that $\deg_w(\varphi)$ exists for each $\varphi\in X$. We will say that $X$ is \emph{strongly bounded} (with respect to $w$) if there exists a simultaneous lower bound for all $\deg_w(\varphi)$ as $\varphi$ ranges over $X$: that is, if there exists $B\in\mathbb{Z}$ such that $\min_{\varphi\in X} \{\deg_w(\varphi)\} \geq B$.
\item We will say that $(\sigma, \delta)$ is \emph{filtered} (with respect to $w$) if the maps $\sigma$, $\sigma^{-1}$ and $\delta$ are filtered endomorphisms of $S$.
\item We will say that $(\sigma, \delta)$ is \emph{strongly bounded} (with respect to $w$) if it is filtered and the set $\sigma^{\mathbb{Z}} \delta^{\mathbb{N}} = \{\sigma^i \delta^j : i\in\mathbb{Z}, j\in\mathbb{N}\}$ is strongly bounded.
\item We will say that $(\sigma,\delta)$ is \emph{quasi-compatible} (with $w$) if it is strongly bounded, and $\deg_w(\delta^N)>0$ for some $N\in\mathbb{N}$.
\end{enumerate}
\end{defn}

This is a generalisation of the notion of \emph{compatible} (Definition \ref{defn: compatibility}): it roughly says that the filtered maps $\sigma-\id$ and $\delta$ are not \emph{too} far away from having positive degree. This is what will allow us to define the bounded skew power series ring $S^+[[x; \sigma, \delta]]$.

\begin{rk*}
Let $(S, w)$ be a filtered ring, and $(\sigma, \delta)$ a commuting skew derivation on $S$ which is quasi-compatible with $w$, so that we can set $B = \min \deg_w(\sigma^{\mathbb{Z}} \delta^{\mathbb{N}})$ and take some $N\in\mathbb{N}$ such that $\deg_w(\delta^N) \geq 1$. Let $p = eN + f$, where $e, f\in\mathbb{N}$, and let $q\in\mathbb{Z}$: then we have
$$\deg_w(\sigma^q\delta^p) = \deg_w(\sigma^q\delta^f\delta^{eN}) \geq \deg_w(\delta^{eN}) + B \geq e + B,$$
where the first inequality comes from Remark \ref{rks: filtered and strict linear maps}(i). In particular, for all $p\in\mathbb{N}$, we have
\begin{align}\label{eqn: bound with nilpotency of delta}
\deg_w(\sigma^q\delta^p) \geq \left\lfloor \dfrac{p}{N}\right\rfloor + B.
\end{align}
\end{rk*}

The following is an easy sufficient criterion for a finitely generated monoid of filtered endomorphisms to be strongly bounded. Later we will apply it to the special cases $(f_1, f_2) = (\sigma, \sigma^{-1})$ where $\sigma$ is an automorphism, and $(f_1, f_2, f_3) = (\sigma, \sigma^{-1}, \delta)$ where $(\sigma, \delta)$ is a skew derivation.

\begin{lem}\label{lem: criterion for strong boundedness}
Let $(S,w)$ be a filtered ring, and let $f_1, \dots, f_r$ be pairwise commuting filtered endomorphisms of $(S,w)$. Suppose there exist positive integers $m_1, \dots, m_r$ such that $\deg_w(f_j^{m_j}) \geq 0$ for all $1\leq j\leq r$. Then $X = \{f_1^{i_1} \dots f_r^{i_r} : i_1, \dots, i_r\in\mathbb{N}\}$ is strongly bounded with respect to $w$.
\end{lem}

\begin{proof}
If $f$ and $g$ are filtered, and $\deg_w(f^m) \geq 0$, then $\deg_w(f^{d+m}g) \geq \deg_w(f^d g)$ for all $d\in\mathbb{N}$. It follows that
$$\min \left\{\deg_w(f_1^{i_1} \dots f_r^{i_r}) : i_j \in\mathbb{N}\right\} \geq \min \left\{\deg_w(f_1^{i_1} \dots f_r^{i_r}) : 0\leq i_j < m_j\right\}.$$
The right-hand side is the minimum degree of a \emph{finite} set of filtered endomorphisms, so there must be a lower bound.
\end{proof}

\begin{cor}\label{cor: criterion for quasi-compatibility}

If $(\sigma,\delta)$ is a filtered, commuting skew derivation of $S$ such that $\deg_w(\sigma^k) \geq 0$, $\deg_w(\sigma^{-k})\geq 0$ and $\deg(\delta^m)>0$ for some $k,m\in\mathbb{N}$, then $(\sigma,\delta)$ is quasi-compatible.

In particular, if $\sigma^k$ is strictly filtered for some $k$, and $\deg_w(\delta^m)>0$ for some $m$, then $(\sigma,\delta)$ is quasi-compatible.

\end{cor}

\begin{proof}

This is immediate from Lemma \ref{lem: criterion for strong boundedness}, taking $f_1=\sigma,f_2=\sigma^{-1},f_3=\delta$.\end{proof}

\subsection{Bounded skew power series rings}\label{subsec: restricted skew power series}

We now define the most appropriate notion of a skew power series ring for our purposes, ultimately proving Theorem \ref{letterthm: restricted skew power series rings exist}.

\begin{defn}\label{defn: R+[[x]] as module}
Let $(S,w)$ be a filtered ring. As a left $S$-module in the obvious way, we will define
$$S^+[[x]] = \left\{\underset{n\in\mathbb{N}}{\sum}{s_nx^n}:s_n\in S,\text{ and there exists }M\in\mathbb{Z}\text{ such that }w(s_n)\geq M\text{ for all }n\in\mathbb{N}\right\}.$$
This is a submodule of the full module of formal power series $S[[x]]$.

For any rational $\varepsilon > 0$, set $\Gamma = \mathbb{Z} + \mathbb{Z}\varepsilon$, and define the function $v_\varepsilon: S[[x]]\to\Gamma\cup\{\pm \infty\}$ by
\begin{align}
v_\varepsilon\left(\sum_{n\in\mathbb{N}} s_n x^n\right) = \min_{n\in\mathbb{N}}\{w(s_n) + \varepsilon n\}.
\end{align}
If the $w(s_n)$ are bounded below, we may say that the sequence $(s_n)$ is itself \emph{bounded} (with respect to $w$). In this case, $v_\varepsilon(\sum s_n x^n) \in \Gamma\cup\{\infty\}$. This makes $S^+[[x]]$ into a $\Gamma$-filtered left $S$-module.
\end{defn}

\begin{rk*}
The submodule $S[x]$ of polynomial elements is \emph{dense} in $(S^+[[x]], v_\varepsilon)$ for any $\varepsilon$. Indeed, any element $f = \sum_{n\in\mathbb{N}} s_n x^n\in S^+[[x]]$ can be approximated arbitrarily well by polynomials $f_N = \sum_{n=0}^N s_n x^n\in S[x]$ for sufficiently large $N$.
\end{rk*}

Given any skew derivation $(\sigma, \delta)$, recall that there is a unique multiplication on the left $S$-module $S[x]$ which extends the multiplication on $S$ and satisfies $xs = \sigma(s)x + \delta(x)$ for all $s\in S$. This makes the module $S[x]$ into the skew polynomial ring $S[x;\sigma,\delta]$ as described in Definition \ref{defn: skew derivation}.

Assume now that $\sigma\delta = \delta\sigma$. Then, inside $S[x; \sigma, \delta]$, we have $\displaystyle x^i s = \sum_{e=0}^i \binom{i}{e} \sigma^e \delta^{i-e}(s) x^e$ for any $s\in S$ \cite[Lemma 1.10]{woods-SPS-dim}, and so, for arbitrary coefficients $a_i, b_j\in S$,
\begin{align*}
\left( \sum_{i=0}^m a_i x^i\right)\left( \sum_{j=0}^n b_jx^j\right) = \sum_{i=0}^m \sum_{j=0}^n a_i\left(\sum_{e=0}^i \binom{i}{e} \sigma^e \delta^{i-e}(b_j) x^e\right)x^j,
\end{align*}
and by swapping the sums and setting $k = j+e$, we can calculate that
\begin{align}
\left( \sum_{i=0}^m a_i x^i\right)\left( \sum_{j=0}^n b_jx^j\right) = \sum_{k=0}^{m+n} \left(\sum_{\substack{e, i:\\ 0\leq e\leq i\leq m,\\ 0\leq k-e\leq n}} \binom{i}{e} a_i \sigma^e \delta^{i-e}(b_{k-e}) \right)x^k.\label{eqn: product in skew polynomial ring}
\end{align}
(Compare \cite[equation (4)]{SchVen06}.) For notational convenience, we set $$\mathcal{S}_k(m,n) = \{(e, i)\in\mathbb{N}^2 : e \leq i \leq m, k-e \leq n\},$$ and when possible we will write the rather unwieldy limits of the inner sum as $(e,i) \in \mathcal{S}_k(m,n)$.

Our aim is to show that, under appropriate conditions on $v_\varepsilon$, we can extend this to a multiplication on $S^+[[x]]$ by continuity: that is, we need to find a continuous map $m: S^+[[x]] \times S^+[[x]] \to S^+[[x]]$ with the property that, if $f,g\in S[x]$ are arbitrary and $fg\in S[x]$ is their product considered as elements of $S[x; \sigma, \delta]$, then $m(f,g) = fg$.

So, \textbf{for the remainder of \S \ref{subsec: restricted skew power series}}:
\begin{itemize}[noitemsep]
\item Let $(S, w)$ be a complete filtered ring.
\item Suppose $(\sigma,\delta)$ is a commuting skew derivation on $S$ which is quasi-compatible with $w$, say $B = \min\deg_w(\sigma^{\mathbb{Z}} \delta^\mathbb{N})$ and $\deg_w(\delta^N) \geq 1$.
\item Fix two arbitrary elements $f,g\in S^+[[x]]$ and two integers $L,M$ such that
\begin{align}
f = \sum_{i\geq 0} r_i x^i \quad&\text{ and }\quad g = \sum_{j\geq 0} s_j x^j,\notag\\
\min_{i\geq 0} \{w(r_i)\} \geq L\quad &\text{ and }\quad  \min_{j\geq 0}\{w(s_j)\} \geq M,\label{eqn: w(r_i) and w(s_i) bounded below}
\end{align}
where $r_i, s_j\in S$.
\end{itemize}

Begin by defining $\widetilde{f}^{(p)}$ and $\widetilde{g}^{(p)}$ to be the partial sums of $f$ and $g$ to degree $p$ for each $p\geq 1$: that is, $\widetilde{f}^{(p)} = \sum_{i=0}^{p} r_i x^i \text{ and } \widetilde{g}^{(p)} = \sum_{j=0}^{p} s_j x^j$.
Returning to \eqref{eqn: product in skew polynomial ring}, we can calculate their product $\widetilde{f}^{(p)}\widetilde{g}^{(p)} = \sum_{k=0}^{2p} c^{(p)}_k x^k$, where
$$c^{(p)}_k := \sum_{(e,i) \in \mathcal{S}_k(p,p)} \binom{i}{e} r_i \sigma^e \delta^{i-e}(s_{k-e}).$$

(For convenience, set $c_k^{(p)} = 0$ if $k > 2p$.)
\begin{lem}\label{lem: coefficients of products of partial sums converge}
For each fixed $k \in\mathbb{N}$, the coefficients $c^{(p)}_k$ converge in $(S,w)$ as $p\to\infty$.
\end{lem}

\begin{proof}
Fix $k \in\mathbb{N}$. As $S$ is assumed complete, we will show that the sequence $(c^{(p)}_k)_{p\geq 1}$ is Cauchy: that is, fix some arbitrary $K>0$, and we will show that there exists $p\geq 1$ such that, for all $P\geq p$, we have $w(c^{(P)}_k - c^{(p)}_k) > K$.

We begin by calculating $c^{(P)}_k - c^{(p)}_k$ for arbitrary $P$ and $p$:
\begin{align*}
c^{(P)}_k - c^{(p)}_k &= \sum_{(e,i)} \binom{i}{e} r_i \sigma^e \delta^{i-e}(s_{k-e}),
\end{align*}
where the sum ranges over all $(e,i) \in \mathcal{S}_k(P,P) \setminus \mathcal{S}_k(p,p)$. It follows that
\begin{align*}
\displaystyle w\left(c^{(P)}_k - c^{(p)}_k\right) &\geq \min_{(e,i)} \left\{ w(r_i\sigma^e\delta^{i-e}(s_{k-e}))\right\}\\
&\geq \min_{(e,i)} \left\{ w(r_i) + w(\sigma^e\delta^{i-e}(s_{k-e}))\right\}\\
&\geq \min_{(e,i)} \left\{ w(r_i) + w(s_{k-e}) + \left\lfloor \dfrac{i-e}{N} \right\rfloor + B \right\},
\end{align*}
using \eqref{eqn: bound with nilpotency of delta}.  Then, by \eqref{eqn: w(r_i) and w(s_i) bounded below}, we get
\begin{align*}
\displaystyle w\left(c^{(P)}_k - c^{(p)}_k\right)\geq \min_{(e,i)} \left\{ L + M + B + \left\lfloor \dfrac{i-e}{N} \right\rfloor \right\}.
\end{align*}

We now choose $p$ to satisfy $p \geq k$ and $p > N(K-(L+M+B)) + k$: note that this latter condition is equivalent to the condition that $L+M+B + \left\lfloor \dfrac{p-k}{N} \right\rfloor > K$.

Fix arbitrary $P \geq p$ and $(e,i) \in \mathcal{S}_k(P,P) \setminus \mathcal{S}_k(p,p)$: this means that either $i > p$ or $k-e > p$ (or both). Firstly, as $p\geq k$ by our first assumption on $p$, the condition $k - e > p$ is impossible, and so we must have $i > p$. Secondly, as we always have $k - e \geq 0$, we will have $i-e \geq i-k > p-k$. Hence we must have
\begin{align*}
\displaystyle \min_{(e,i)} \left\{ L + M + B + \left\lfloor \dfrac{i-e}{N} \right\rfloor \right\} \geq L + M + B + \left\lfloor \dfrac{p-k}{N} \right\rfloor,
\end{align*}
which is greater than $K$ by our second assumption on $p$.
\end{proof}

Now write $c_k\in S$ for the limit of $(c^{(p)}_k)_{p\geq 1}$, and define $\ell = \sum_{k\geq 0} c_k x^k\in S[[x]]$.

\begin{lem}\label{lem: the limit exists in R^+[[x]]}
$ $

\begin{enumerate}[label=(\roman*)]
\item For all $p$ and $k$, we have $w(c^{(p)}_k) \geq L + M + B$.
\item For all $k$, we have $w(c_k) \geq L+M+B$.
\item $\ell\in S^+[[x]]$.
\end{enumerate}
\end{lem}

\begin{proof}
If $k > 2p$, then $c^{(p)}_k = 0$, so $w(c^{(p)}_k) = \infty$, so part (i) is trivially true. For $k \leq 2p$, part (i) follows from a very similar calculation to that in the proof of Lemma \ref{lem: coefficients of products of partial sums converge}:
\begin{align*}
w(c^{(p)}_k) &= w\left(\sum_{(e,i) \in \mathcal{S}_k(p,p)} \binom{i}{e} r_i \sigma^e \delta^{i-e}(s_{k-e})\right)\\
&\geq \min_{(e,i)} \{w(r_i) + w(\sigma^e \delta^{i-e} (s_{k-e}))\}\\
&\geq \min_{(e,i)} \{w(r_i) + w(s_{k-e})\} + B\\
&\geq L + M + B
\end{align*}
using \eqref{eqn: bound with nilpotency of delta} and \eqref{eqn: w(r_i) and w(s_i) bounded below}. Parts (ii) and (iii) now follow immediately from the definitions.
\end{proof}

This $\ell$ will be our value of $m(f,g)$. Note that it is independent of the constant $\varepsilon$ chosen in Definition \ref{defn: R+[[x]] as module}.

\begin{propn}\label{propn: first limit}
Let $\varepsilon > 0$ be arbitrary, and write $v = v_\varepsilon$. Then $\widetilde{f}^{(p)}\widetilde{g}^{(p)} \to \ell$ in $(S^+[[x]], v)$ as $p\to \infty$.
\end{propn}

\begin{proof}
Let $K > 0$. We must show that, for sufficiently large $p$, we have $v(\widetilde{f}^{(p)}\widetilde{g}^{(p)} - \ell) > K$. But
\begin{align*}
v(\widetilde{f}^{(p)}\widetilde{g}^{(p)} - \ell) &= v\left(\sum_{k=0}^{\infty} (c^{(p)}_k - c_k) x^k\right) = \min_{k\geq 0} \left\{ w(c^{(p)}_k - c_k) + \varepsilon k\right\},
\end{align*}
so we must show equivalently that, for sufficiently large $p$, $w(c^{(p)}_k - c_k) + \varepsilon k > K$ for all $k\geq 0$.

Now, on the one hand we know that $w(c^{(p)}_k - c_k) \geq L+M+B$ for \emph{all} $p$ and $k$ by Lemma \ref{lem: the limit exists in R^+[[x]]}(i), (ii). Hence, for all $k > A := \varepsilon^{-1}(K-(L+M+B))$, we have $w(c^{(p)}_k - c_k) + \varepsilon k > K$ for \emph{all} $p$. On the other hand, for each fixed $0\leq k\leq A$, it follows from Lemma \ref{lem: coefficients of products of partial sums converge} and the definition of $c_k$ that there exists $P_k$ such that $w(c^{(p)}_k - c_k) > K$ whenever $p > P_k$. Taking $P = \max\{P_0, \dots, P_A\}$, it is now clear that $v(\widetilde{f}^{(p)}\widetilde{g}^{(p)} - \ell) > K$ whenever $p > P$.
\end{proof}

Now let $(f^{(p)})_{p\geq 1}$, $(g^{(p)})_{p\geq 1}$ be \emph{arbitrary} sequences of polynomials, say $f^{(p)} = \sum_{i=0}^{m(p)} r^{(p)}_i x^i$ and $g^{(p)} = \sum_{j=0}^{n(p)} s^{(p)}_j x^j$, converging to $f$ and $g$ respectively under $v$. That is,
\begin{align}
\min_{i\geq 0}\left\{ w(r^{(p)}_i - r_i) + \varepsilon i\right\} \to \infty \text{ as } p\to \infty,\quad\;
\min_{j\geq 0}\left\{ w(s^{(p)}_j - s_j) + \varepsilon j\right\} \to \infty \text{ as } p\to \infty,\label{eqn: v-bounds 1}
\end{align}
where it is understood that $r^{(p)}_i = 0$ for $i > m(p)$, and similarly $s^{(p)}_j = 0$ for $j > n(p)$.

Once again from \eqref{eqn: product in skew polynomial ring}, we can calculate $f^{(p)} g^{(p)} = \sum_{k\geq 0} a^{(p)}_k x^k$, where
$$a^{(p)}_k := \sum_{(e,i) \in \mathcal{S}_k(m(p),n(p))} \binom{i}{e} r^{(p)}_i \sigma^e \delta^{i-e}(s^{(p)}_{k-e})$$
for $0 \leq k \leq m(p)+n(p)$ and $a^{(p)}_k = 0$ otherwise.

\begin{propn}\label{propn: existence of skew power series rings}
Let $0 < \varepsilon\leq 1/N$, and write $v = v_\varepsilon$. Then $f^{(p)}g^{(p)} \to \ell$ in $(S^+[[x]], v)$ as $p\to \infty$.
\end{propn}

\begin{proof}
Take $K > 0$. In light of Proposition \ref{propn: first limit}, we need only show that there exists some $P\geq 1$ such that, for all $p>P$, $v(f^{(p)}g^{(p)} - \widetilde{f}^{(p)}\widetilde{g}^{(p)}) > K$.

To do this, we will calculate
\begin{align*}
v(f^{(p)}g^{(p)} - \widetilde{f}^{(p)}\widetilde{g}^{(p)}) &= v\left(\sum_{k=0}^{\infty} (a^{(p)}_k - c^{(p)}_k) x^k\right) = \min_{k\geq 0} \left\{ w\left(a^{(p)}_k - c^{(p)}_k\right) + \varepsilon k\right\},
\end{align*}

and for each $k$,
\begin{align*}
a^{(p)}_k - c^{(p)}_k = \sum_{\mathcal{T}_1} \binom{i}{e} r^{(p)}_i \sigma^e \delta^{i-e}(s^{(p)}_{k-e}) - \sum_{\mathcal{T}_2} \binom{i}{e} r_i \sigma^e \delta^{i-e}(s_{k-e}),
\end{align*}
where $\mathcal{T}_1 = \mathcal{S}_k(m(p),n(p))$ and $\mathcal{T}_2 = \mathcal{S}_k(p,p)$. Split the right-hand side of this equation into $(e,i)$-terms, so that $\displaystyle a^{(p)}_k - c^{(p)}_k = \sum_{(e,i)\in \mathcal{T}} \binom{i}{e} t(k,e,i)$, where $\mathcal{T} = \mathcal{T}_1 \cup \mathcal{T}_2$: it will follow that
$$v(f^{(p)}g^{(p)} - \widetilde{f}^{(p)}\widetilde{g}^{(p)}) \geq \min_{k\geq 0} \min_{(e,i)\in \mathcal{T}} \left\{w(t(k,e,i)) + \varepsilon k\right\}.$$
Explicitly, for any $k\geq 0$ and $p\geq 1$, we define
$$
t(k,e,i) = t^{(p)}(k,e,i) =
\begin{cases}
\displaystyle r^{(p)}_i \sigma^e \delta^{i-e}(s^{(p)}_{k-e}) - r_i \sigma^e \delta^{i-e}(s_{k-e}), & (e,i) \in \mathcal{T}_1 \cap \mathcal{T}_2,\\
\displaystyle r^{(p)}_i \sigma^e \delta^{i-e}(s^{(p)}_{k-e}),& (e,i) \in \mathcal{T}_1 \setminus \mathcal{T}_2,\\
\displaystyle - r_i \sigma^e \delta^{i-e}(s_{k-e}),& (e,i) \in \mathcal{T}_2 \setminus \mathcal{T}_1.
\end{cases}
$$
For the remainder of the proof, we deal with these three cases separately.

We pause here to note a calculation that we will use several times. For arbitrary $r, s\in S$, it follows from \eqref{eqn: bound with nilpotency of delta} that
\begin{align}
w(r \sigma^e \delta^{i-e}(s)) + \varepsilon k &\geq w(r) + w(s) + \left\lfloor \dfrac{i-e}{N} \right\rfloor + B + \varepsilon k\notag\\
&\geq w(r) + w(s) + \dfrac{i-e}{N} + B - 1 + \varepsilon k\notag\\
&\geq w(r) + w(s) + \varepsilon(i-e+k) + B-1\notag\\
&= (w(r) + \varepsilon i) + (w(s) + \varepsilon(k-e)) + B-1.\label{eqn: using bounds}
\end{align}

We now bound $w(t(k,e,i)) + \varepsilon k$ in each case.

\textbf{Case 1:} $(e,i) \in \mathcal{T}_1 \cap \mathcal{T}_2$. In this case, it is easily checked that
$$t(k,e,i) = (r^{(p)}_i - r_i) \sigma^e \delta^{i-e}(s^{(p)}_{k-e} - s_{k-e}) + (r^{(p)}_i - r_i) \sigma^e \delta^{i-e}(s_{k-e}) + r_i \sigma^e \delta^{i-e}(s^{(p)}_{k-e} - s_{k-e}),$$
and so
\begin{align*}
w(t(k,e,i)) + \varepsilon k&\geq \min\Big\{ w \left((r^{(p)}_i - r_i) \sigma^e \delta^{i-e}(s^{(p)}_{k-e} - s_{k-e})\right) + \varepsilon k,\\
&\phantom{\geq \min\Big\{} w\left((r^{(p)}_i - r_i) \sigma^e \delta^{i-e}(s_{k-e})) \right) + \varepsilon k,\\
&\phantom{\geq \min\Big\{} w\left( r_i \sigma^e \delta^{i-e}(s^{(p)}_{k-e} - s_{k-e}) \right) + \varepsilon k \Big\}.
\end{align*}
By three applications of \eqref{eqn: using bounds}, we get
\refstepcounter{equation}\label{eqn: cases}
\begin{align}
\tag{\ref{eqn: cases}.1a}w(t(k,e,i)) + \varepsilon k&\geq \min\Big\{ \left(w (r^{(p)}_i - r_i) + \varepsilon i\right) + \left(w(s^{(p)}_{k-e} - s_{k-e}) + \varepsilon(k-e)\right) + B-1, \\
\tag{\ref{eqn: cases}.1b}&\phantom{\geq \min\Big\{} \left(w (r^{(p)}_i - r_i) + \varepsilon i\right) + \left(w(s_{k-e}) + \varepsilon(k-e)\right) + B-1, \\
\tag{\ref{eqn: cases}.1c}&\phantom{\geq \min\Big\{} \left(w(r_i) + \varepsilon i\right) + \left( w(s^{(p)}_{k-e} - s_{k-e})  + \varepsilon(k-e)\right) + B-1\Big\}.
\end{align}
Now, by definition of $f$ and $g$ \eqref{eqn: w(r_i) and w(s_i) bounded below}, we have $w(r_h) \geq L$ and $w(s_j) \geq M$ for \emph{all} $h\geq 0$ and $j\geq 0$. Also, by \eqref{eqn: v-bounds 1}, there exists $P\geq 1$ such that, for all $p > P$,
\begin{itemize}
\item $w (r^{(p)}_h - r_h) + \varepsilon h > \max\left\{ K-M-B+1, \dfrac{K-B+1}{2}\right\}$ for \emph{all} $h \geq 0$, and
\item $w(s^{(p)}_j - s_j) + \varepsilon j > \max\left\{ K-L-B+1, \dfrac{K-B+1}{2}\right\}$ for \emph{all} $j \geq 0$.
\end{itemize}
Fix any such $P$, and take $h = i$ and $j = k-e$. Using these bounds, it is easy to check that the three expressions (\ref{eqn: cases}.1a--c) are greater than $K$ as required for any $p > P$, $i\geq 0$ and $k-e\geq 0$. This holds in particular when $k\geq 0$ and $(e,i)\in \mathcal{T}_1 \cap \mathcal{T}_2$, as required for this case. So, whenever $p > P$,
$$\min_{k\geq 0} \min_{(e,i)\in \mathcal{T}_1 \cap \mathcal{T}_2} \left\{w(t(k,e,i)) + \varepsilon k\right\} > K.$$

\textbf{Case 2:} $(e,i) \in \mathcal{T}_1 \setminus \mathcal{T}_2$. In this case, similarly to case 1, we get
$$t(k,e,i)=(r_i^{(p)}-r_i)\sigma^e\delta^{i-e}(s_{k-e}^{(p)}-s_{k-e})+(r_i^{(p)}-r_i)\sigma^e\delta^{i-e}(s_{k-e})+r_i\sigma^e\delta^{i-e}(s_{k-e}^{(p)}-s_{k-e})+r_i\sigma^e\delta^{i-e}(s_{k-e}).$$
Again, with four applications of \eqref{eqn: using bounds}, we get
\begin{align}
\tag{\ref{eqn: cases}.2a}w(t(k,e,i)) + \varepsilon k&\geq \min\Big\{ \left(w (r^{(p)}_i - r_i) + \varepsilon i\right) + \left(w(s^{(p)}_{k-e} - s_{k-e}) + \varepsilon(k-e)\right) + B-1, \\
\tag{\ref{eqn: cases}.2b}&\phantom{\geq \min\Big\{} \left(w (r^{(p)}_i - r_i) + \varepsilon i\right) + \left(w(s_{k-e}) + \varepsilon(k-e)\right) + B-1, \\
\tag{\ref{eqn: cases}.2c}&\phantom{\geq \min\Big\{} \left(w(r_i) + \varepsilon i\right) + \left( w(s^{(p)}_{k-e} - s_{k-e})  + \varepsilon(k-e)\right) + B-1, \\
\tag{\ref{eqn: cases}.2d}&\phantom{\geq \min\Big\{} \left(w(r_i) + \varepsilon i\right) + \left( w(s_{k-e})  + \varepsilon(k-e)\right) + B-1\Big\}.
\end{align}
Expressions (\ref{eqn: cases}.2a--c) can be made greater than $K$ exactly as in case 1: the value of $P$ found in case 1 works here too.

To deal with the remaining expression, we recall that we are only interested in those $(e,i) \in \mathcal{T}_1\setminus \mathcal{T}_2$ in this case: that is, $0\leq e\leq i\leq m(p)$ and $0\leq k-e \leq n(p)$, but \emph{either} $i > p$ \emph{or} $k-e > p$ (or both). In particular, for any $(e,i) \in \mathcal{T}_1 \setminus \mathcal{T}_2$, we will always have $\varepsilon i + \varepsilon(k-e) > \varepsilon p$. Combining this with \eqref{eqn: w(r_i) and w(s_i) bounded below}, expression (\ref{eqn: cases}.2d) can be bounded below as follows:
\begin{align*}
\left(w(r_i) + \varepsilon i\right) + \left( w(s_{k-e})  + \varepsilon(k-e)\right) + B-1 > L + M + B - 1 + \varepsilon p.
\end{align*}
Now we need only choose our value of $P$ above so that $P \geq \varepsilon^{-1}(K - (L+M+B-1))$: this will ensure that expression (\ref{eqn: cases}.2d) is greater than $K$ for all $p > P$. In total, we have shown that, whenever $p > P$,
$$\min_{k\geq 0} \min_{(e,i)\in \mathcal{T}_1 \setminus \mathcal{T}_2} \left\{w(t(k,e,i)) + \varepsilon k\right\} > K.$$

\textbf{Case 3.} $(e,i) \in \mathcal{T}_2 \setminus \mathcal{T}_1$. In this case, $t(k,e,i)=- r_i \sigma^e \delta^{i-e}(s_{k-e})$, and we have simply
\begin{align}\label{eqn: case 3}
w(t(k,e,i)) + \varepsilon k \geq \left(w(r_i) + \varepsilon i\right) + \left(w(s_{k-e}) + \varepsilon(k-e)\right) + B-1.
\end{align}
We also know that $(e,i) \in \mathcal{T}_2\setminus \mathcal{T}_1$: that is, $0\leq e\leq i\leq p$ and $0\leq k-e \leq p$, but \emph{either} $i > m(p)$ \emph{or} $k-e > n(p)$.

Restrict first to those $(e,i)$ satisfying $i > m(p)$. Depending on the asymptotic behaviour of $m(p)$ as $p\to\infty$, we argue in two different ways:
\begin{itemize}
\item If $m(p)\to\infty$, then combining \eqref{eqn: case 3} with \eqref{eqn: w(r_i) and w(s_i) bounded below}, we get $$w(t(k,e,i)) + \varepsilon k \geq L + M + B-1 + \varepsilon m(p),$$ and by choosing $P$ sufficiently large, we can ensure that this expression is greater than $K$ for all $p > P$.
\item If $m(p)\not\to\infty$: let $d = \liminf_{p\to\infty} m(p)$, and fix $p_1 < p_2 < \dots$ such that $m(p_\ell) = d$ for all $\ell\geq 1$. Then $f^{(p_1)}, f^{(p_2)}, \dots$ is a sequence of polynomials of degree $d$ converging to $f$, and so $f$ must also be a polynomial of degree at most $d$. In particular, choosing $P$ sufficiently large, we can ensure that $m(p) \geq d$ for all $p > P$: this implies $i > d$ and so $r_i = 0$. In this case, $w(t(k,e,i)) + \varepsilon k = \infty$ for all $p > P$.
\end{itemize}
In both cases, we can choose a sufficiently large value of $P$ such that $w(t(k,e,i)) + \varepsilon k > K$ for all $p > P$ whenever $i > m(p)$. By a symmetric argument, we can also choose $P$ so that $w(t(k,e,i)) + \varepsilon k > K$ for all $p > P$ whenever $k-e > n(p)$.

Since every $(e,i)\in \mathcal{T}_2 \setminus \mathcal{T}_1$ satisfies either $i > m(p)$ or $k-e > n(p)$, we have shown that $\displaystyle \min_{k\geq 0} \min_{(e,i)\in \mathcal{T}_2 \setminus \mathcal{T}_1} \left\{w(t(k,e,i)) + \varepsilon k\right\} > K$ whenever $p > P$.
\end{proof}

We can now complete the proof of Theorem \ref{letterthm: restricted skew power series rings exist}, stated more fully as follows:

\begin{thm}\label{thm: restricted skew power series rings exist}
Suppose $(S,w)$ is a complete filtered ring, and $(\sigma,\delta)$ is a commuting skew derivation on $S$ which is quasi-compatible with $w$. Then: 
\begin{enumerate}[label=(\roman*)]
\item The \emph{bounded skew power series ring} $S^+[[x; \sigma, \delta]]$ (constructed with respect to $w$) is a well-defined ring containing $S[x; \sigma, \delta]$ as a subring. Moreover, $S^+[[x; \sigma, \delta]]$ can be made into a topological ring with topology given by a $\Gamma$-filtration of left $S$-modules $v = v_\varepsilon$ (for any sufficiently small positive rational $\varepsilon$) such that $v|_S = w$, and $S[x; \sigma, \delta]$ is a \emph{dense} subring of $(S^+[[x; \sigma, \delta]], v)$.
\item If $S$ is positively filtered, then $S^+[[x; \sigma, \delta]]$ is complete with respect to $v$, and is equal to $S[[x; \sigma, \delta]]$.
\item Suppose $A$ is a complete, positively filtered subring of $S$ preserved by $(\sigma, \delta)$. Let $T$ be a $\sigma$-invariant denominator set in $A$, and suppose that $S = T^{-1}A$. Suppose also that, given any bounded countable sequence $(s_n)_{n\in\mathbb{N}}$ of elements of $S$, there exists $t\in T$ such that $ts_n\in A$ for all $n\in\mathbb{N}$. Then $T^{-1}(A[[x; \sigma, \delta]]) = S\otimes_A A[[x; \sigma, \delta]]$ can be canonically identified with $S^+[[x; \sigma, \delta]]$.\qed
\end{enumerate}
\end{thm}

\begin{proof}

As $(\sigma, \delta)$ is quasi-compatible with $w$, there exists some $N > 0$ such that $\deg_w(\delta^N) \geq 1$. Let $v = v_\varepsilon$ for any $0 < \varepsilon \leq 1/N$.

For part (i): suppose we are given arbitrary $f,g\in S^+[[x]]$. Choose $(f^{(p)})_{p\geq 1}$ and $(g^{(p)})_{p\geq 1}$ to be arbitrary sequences in $S[x]$ converging to $f$ and $g$ respectively, and define $m(f,g)$ to be the limit of $f^{(p)} g^{(p)}$ (where this product is calculated in $S[x; \sigma, \delta]$) as $p\to\infty$: by Proposition \ref{propn: existence of skew power series rings}, this is well-defined, and we remarked after Lemma \ref{lem: the limit exists in R^+[[x]]} that it is independent of the choice of $\varepsilon$. Moreover, if $f$ and $g$ themselves are polynomials, we can take the constant sequences $f^{(p)} = f$, $g^{(p)} = g$ to see that $m(f,g) = fg$.

We need to show that $m$ is a ring multiplication map on $S^+[[x]]$. We begin by showing that $m$ is continuous on all of $(S^+[[x]], v)$. Let $K > 0$ be arbitrary.
\begin{itemize}
\item Choose sequences $(f^{(p)})_{p\geq 1}$ and $(g^{(p)})_{p\geq 1}$ (\emph{not} necessarily sequences of polynomials) in $S^+[[x]]$ converging to $f$ and $g$ respectively. Then there exists $P \geq 1$ such that, for all $p\geq P$,
$$v(f^{(p)} - f) > K, \qquad v(g^{(p)} - g) > K.$$
\item For each fixed $p$, choose sequences $(f^{(p,i)})_{i\geq 1}$ and $(g^{(p,i)})_{i\geq 1}$ in $S[x]$ which converge to $f^{(p)}$ and $g^{(p)}$ respectively. Then there exists $I_p \geq 1$ such that, for all $i \geq I_p$,
$$v(f^{(p,i)} - f^{(p)}) > K, \qquad v(g^{(p,i)} - g^{(p)}) > K.$$
\item For each fixed $p$, since $f^{(p,i)}$ and $g^{(p,i)}$ are polynomials, $m(f^{(p,i)},g^{(p,i)}) \to m(f^{(p)}, g^{(p)})$ as $i\to\infty$. Then there exists $J_p$ such that, for all $i\geq J_p$,
$$v(m(f^{(p,i)},g^{(p,i)}) - m(f^{(p)}, g^{(p)})) > K.$$
\end{itemize}
Choose a sequence of positive integers $(i_p)_{p\geq 1}$ such that $i_p \geq \max\{I_p, J_p\}$ for all $p$. We have shown that, for all $p \geq P$, we have
\begin{align}
v(f^{(p,i_p)} - f) > K, \qquad v(g^{(p,i_p)} - g) > K,\label{eqn: convergence 1}\\
v(m(f^{(p,i_p)},g^{(p,i_p)}) - m(f^{(p)}, g^{(p)})) > K.\label{eqn: convergence 2}
\end{align}
In particular, as $K$ was arbitrary, \eqref{eqn: convergence 1} implies that $(f^{(p,i_p)})_{p\geq 1}$ and $(g^{(p,i_p)})_{p\geq 1}$ are sequences \emph{of polynomials} converging to $f$ and $g$ respectively, and so we must have $m(f^{(p,i_p)},g^{(p,i_p)}) \to m(f,g)$ as $p\to\infty$ by Proposition \ref{propn: existence of skew power series rings}.

Fix $K>0$ again.
\begin{itemize}
\item As $m(f^{(p,i_p)},g^{(p,i_p)}) \to m(f,g)$, there exists $P' \geq 1$ such that, for all $p\geq P'$,
$$v(m(f^{(p,i_p)},g^{(p,i_p)}) - m(f,g)) > K.$$
\end{itemize}
However, now combining this with \eqref{eqn: convergence 2}, for all $p\geq \max\{P,P'\}$, we get
$$v(m(f^{(p)},g^{(p)}) - m(f,g)) > K.$$
Since $K$ is arbitrary, we get $m(f^{(p)},g^{(p)}) \to m(f,g)$ as $p\to\infty$, and so $m$ is sequentially continuous. But as $S^+[[x]]$ is first-countable (see the remark after Definition \ref{defn: filtration topology}), this means that $m$ is continuous.

To show that $(f, g) \mapsto fg := m(f,g)$ is a multiplication map on $S^+[[x; \sigma, \delta]]$, it remains to show that it is associative. However, we remarked already that $S[x]$ is a dense submodule of $S^+[[x]]$ in Definition \ref{defn: R+[[x]] as module}, and we defined $m$ above to agree with the (associative) multiplication in $S[x; \sigma, \delta]$, so this follows by continuity, and we have established (i).

The remaining parts now follow easily from the definitions:
\begin{enumerate}[label=(\roman*)]
\setcounter{enumi}{1}
\item Follows from the description in Definition \ref{defn: R+[[x]] as module} of $S^+[[x; \sigma, \delta]]$ as a set: if $S$ is positively filtered, then $w(s_n) \geq 0$ for all $n\in\mathbb{N}$.
\item (Compare \cite[Lemma 1.4]{goodearl-skew-poly-and-quantized}.) The natural inclusion map $A\hookrightarrow S$ extends to a strictly filtered inclusion map $A[[x; \sigma, \delta]]\hookrightarrow S^+[[x; \sigma, \delta]]$ sending $x$ to $x$. Every element of $T$ is invertible in $S$, and hence in $S^+[[x; \sigma, \delta]]$. Furthermore, given any $\sum_n s_n x^n\in S^+[[x; \sigma, \delta]]$, as the $w(s_n)$ are bounded below, there exists some $t\in T$ such that $\sum_n s_n x^n$ can be written as $t^{-1}f$ for some $f \in A[[x; \sigma, \delta]]$. Hence $S^+[[x; \sigma, \delta]] = T^{-1}(A[[x; \sigma, \delta]]) = S\otimes_A A[[x; \sigma, \delta]]$.\qedhere
\end{enumerate}
\end{proof}

We have shown above that the following definition makes sense, and we record it here:

\begin{defn}\label{defn: skew power series rings}
Suppose $(S,w)$ is a complete filtered ring, and $(\sigma, \delta)$ is a commuting skew derivation on $S$ which is quasi-compatible with $w$. Then we can define the \emph{bounded skew power series ring}
$$S^+[[x;\sigma, \delta]] = \left\{\underset{n\in\mathbb{N}}{\sum}{s_nx^n}:s_n\in S,\text{ there exists }M\in\mathbb{Z}\text{ such that }w(s_n)\geq M\text{ for all }n\in\mathbb{N}\right\},$$
with multiplication given by
\begin{align*}
\left( \sum_{i\geq 0} a_i x^i\right)\left( \sum_{j\geq 0} b_jx^j\right) = \sum_{k\geq 0} \left(\sum_{0\leq e\leq k} \sum_{i\geq e} \binom{i}{e} a_i \sigma^e \delta^{i-e}(b_{k-e}) \right)x^k.
\end{align*}
Moreover, for any sufficiently small rational $\varepsilon > 0$, and setting $\Gamma = \mathbb{Z} + \varepsilon\mathbb{Z}$: the $\Gamma$-filtration of $S$-modules $v_\varepsilon$ defined by
\begin{align*}
v_\varepsilon\left(\sum_{n\in\mathbb{N}} s_n x^n\right) = \min_{n\in\mathbb{N}}\{w(s_n) + \varepsilon n\}
\end{align*}
defines a ring topology on $S^+[[x; \sigma, \delta]]$, the subring $S[x;\sigma,\delta]$ is dense in $(S^+[[x; \sigma, \delta]], v_\varepsilon)$, and the multiplication in $S^+[[x; \sigma, \delta]]$ is continuous with respect to $v_\varepsilon$.
\end{defn}

\begin{rks}\label{rks: topologies}
$ $

\begin{enumerate}[label=(\roman*)]
\item Different values of $\varepsilon$ may give different topologies if $S$ is not positively filtered. However, the resulting (abstract) ring is always the same, depending only on the filtration $w$ on $S$. We will usually omit these details in what follows, as the value of $\varepsilon$ will usually be unimportant to us and the filtration $w$ will be easily understood from context.
\item The module filtration $v_\varepsilon$ does not always define a ring filtration. In fact, if $\sigma$ has negative degree, it never does: indeed, choose some $s\in S$ such that $w(\sigma(s)) < w(s)$. Then $v_{\varepsilon}(xs)=v_{\varepsilon}(\sigma(s)x+\delta(s))=\min\{w(\sigma(s))+\varepsilon,w(\delta(s))\}$ by definition, and so in particular $v_\varepsilon(xs) \leq w(\sigma(s)) + \varepsilon < w(s) + \varepsilon=v_{\varepsilon}(x)+v_{\varepsilon}(s)$.
\item We can always choose $\varepsilon = m/n$ for positive integers $m,n$, and then replace $v_\varepsilon$ by the topologically equivalent ($\mathbb{Z}$-)filtration $nv_\varepsilon$. In this way, we will no longer need to deal with general $\Gamma$-filtrations throughout the remainder of the paper.
\end{enumerate}
\end{rks}

\subsection{Skew power series subrings: in characteristic \(p\)}\label{subsec: char p}

We now want to examine certain large subrings of bounded skew power series rings, whose structure we will later see is essential to understanding the full skew power series ring.

\begin{defn}\label{defn: X_n etc in char p}
Suppose that $S$ is a ring of characteristic $p>0$, and let $(\sigma,\delta)$ be any commuting skew derivation on $S$. For any $n\in\mathbb{N}$, set
\begin{itemize}
\item $X_n:=x^{p^n}\in S[x;\sigma,\delta]$,
\item $\Sigma_n:=\sigma^{p^n}$,
\item $\Delta_n:=\delta^{p^n}$.
\end{itemize}
Note that $(\Sigma_n,\Delta_n)$ is a (commuting) skew derivation on $S$ by \cite[Remark 1.2(i)]{jones-woods-1}, and we have $X_n s = \Sigma_n(s) X_n + \Delta_n(s)$ inside $S[x; \sigma, \delta]$ for all $s\in S$. Hence we can define the \emph{subring} $S[X_n;\Sigma_n,\Delta_n] \subseteq S[x; \sigma, \delta]$.
\end{defn}

\begin{propn}\label{propn: subrings in characteristic p}
Suppose $(S, w)$ is a complete filtered ring of characteristic $p$, and let $(\sigma, \delta)$ be a commuting skew derivation on $S$. Fix any $n\in\mathbb{N}$. Then the following are equivalent:
\begin{enumerate}[label=(\alph*)]
\item $(\sigma,\delta)$ is quasi-compatible with $w$.
\item $(\Sigma_n,\Delta_n)$ is quasi-compatible with $w$ for \emph{all} $n\in\mathbb{N}$.
\item $(\Sigma_n,\Delta_n)$ is quasi-compatible with $w$ for \emph{some} $n\in\mathbb{N}$.
\end{enumerate}
In this case, for any $n\in\mathbb{N}$, $S^+[[x;\sigma,\delta]]$ is freely generated as a module over the subring $S^+[[X_n;\Sigma_n,\Delta_n]]$ with basis $\{x^i:i=0,\dots,p^n-1\}$.

\end{propn}

\begin{proof}
The implications (b) $\implies$ (a) $\implies$ (c) are obvious, because $(\sigma, \delta) = (\Sigma_0, \Delta_0)$. Since $\Sigma_{n+1} = \Sigma_n^p$ and $\Delta_{n+1} = \Delta_n^p$ for all $n\in\mathbb{N}$, to prove (c) $\implies$ (b) by induction, it will suffice to write $X = X_1$, $\Sigma = \Sigma_1$, $\Delta = \Delta_1$ and prove that $(\sigma, \delta)$ is quasi-compatible if and only if $(\Sigma,\Delta)$ is quasi-compatible.

If $(\sigma,\delta)$ is quasi-compatible, then there are $A\in\mathbb{Z}$ and $N\in\mathbb{N}$ such that $\deg_w(\sigma^i\delta^j)\geq A$ for all $i\in\mathbb{Z}$, $j\in\mathbb{N}$, and $\deg_w(\delta^N)>0$. Therefore, for arbitrary $i\in\mathbb{Z}$ and $j\in\mathbb{N}$, we have $\deg_w(\Sigma^i\Delta^j)=\deg_w(\sigma^{ip}\delta^{jp})\geq A$, and $\deg_w(\Delta^N)=\deg_w(\delta^{Np})\geq p\deg_w(\delta^N)>0$, so $(\Sigma,\Delta)$ is quasi-compatible.

Conversely, if $(\Sigma,\Delta)$ is quasi-compatible, then there are $A\in\mathbb{Z}$ and $N\in\mathbb{N}$ such that $\deg_w(\Sigma^i\Delta^j)\geq A$ for all $i\in\mathbb{Z}$, $j\in\mathbb{N}$, and $\deg_w(\Delta^N)>0$. Set $B:=\min\{\deg_w(\sigma^k\delta^\ell): 0\leq k,\ell\leq p-1\}$, and take $i\in\mathbb{Z}$, $j\in\mathbb{N}$ arbitrary: since we know that $\deg(\Sigma^{i'}\Delta^{j'})=\deg(\sigma^{i'p}\delta^{j'p})\geq A$ for all $i'\in\mathbb{Z},j'\in\mathbb{N}$, it follows that if we set $i=pi'+k$, $j=pj'+\ell$ for some $0\leq k,\ell\leq p-1$, then $\deg_w(\sigma^i\delta^j)=\deg_w(\sigma^k\delta^\ell\Sigma^{i'}\Delta^{j'})\geq A+B$. Hence $(\sigma,\delta)$ is strongly bounded. We also have $\deg(\delta^{Np})=\deg(\Delta^N)>0$, so $(\sigma, \delta)$ is quasi-compatible.

Now assume that $(\sigma, \delta)$ and $(\Sigma, \Delta)$ are quasi-compatible with $w$. Given any sufficiently small $\varepsilon, \varepsilon'$, write $v_{x,\varepsilon}$ and $v_{X,\varepsilon'}$ for the associated filtrations on $S^+[[x]]$ and $S^+[[X]]$ respectively. Consider the canonical (left $S$-module) inclusion map $\theta: (S^+[[X]], v_{X,\varepsilon'}) \to (S^+[[x]],v_{x,\varepsilon})$: choosing $\varepsilon = \varepsilon'/p$, we get $v_{X,\varepsilon'}(f) = v_{x,\varepsilon}(\theta(f))$ for all $f\in S^+[[X]]$, and so $\theta$ is strictly filtered and hence continuous.

To show that $\theta: S^+[[X; \Sigma, \Delta]] \to S^+[[x; \sigma, \delta]]$ is a ring homomorphism, we need to show that the diagram

\centerline{
\xymatrix@C=36pt{
S^+[[X; \Sigma, \Delta]]\times S^+[[X; \Sigma, \Delta]]\ar[r]^-{m_X}\ar[d]_{\theta\times \theta}& S^+[[X; \Sigma, \Delta]]\ar[d]^\theta\\
S^+[[x; \sigma, \delta]]\times S^+[[x; \sigma, \delta]]\ar[r]_-{m_x}& S^+[[x; \sigma, \delta]]
}
}

commutes, where $m_X$ and $m_x$ are the multiplication maps. But $\theta\circ m_X$ and $m_x\circ (\theta\times\theta)$ are continuous and agree on the dense subset $S^+[X; \Sigma, \Delta] \times S^+[X; \Sigma, \Delta]$, so they agree everywhere.

Finally, since any power series $\underset{i\in\mathbb{N}}{\sum}{s_ix^{i}}$ in $S^+[[x;\sigma,\delta]]$ can be uniquely written as $$\sum_{k=0}^{p-1} \left(\sum_{i'\in\mathbb{N}} {s_{i'p+k}X^{i'}}\right)x^k,$$ it follows that $S^+[[x;\sigma,\delta]]$ is freely generated over $S^+[[X;\Sigma,\Delta]]$ with basis $\{x^i : 0\leq i\leq p-1\}$ as required.\end{proof}

\subsection{Skew power series subrings: inner \(\sigma\)-derivations}\label{subsec: skew power series subrings}

Following on from the previous subsection, we now want to assume $(S,w)$ is a general complete filtered ring, not necessarily of characteristic $p$. In general, we do not have that $(\sigma^{p^n},\delta^{p^n})$ is a skew derivation for all $n$, but there are cases where we can establish an appropriate analogue version of Proposition \ref{propn: subrings in characteristic p}.

\textbf{We will assume throughout \S \ref{subsec: skew power series subrings}} that $(\sigma, \delta)$ is a commuting skew derivation of $S$ and that $\delta$ is an \emph{inner} $\sigma$-derivation of $S$ (Definition \ref{defn: inner sigma-derivation}): precisely, we will assume that we have $t\in S$ such that $\delta(s)=ts-\sigma(s)t$ for all $s\in S$.

The condition that $\sigma(t)=t$ will be very important to us throughout the article, and the following lemma establishes an important case when we can ensure it is satisfied.

\begin{lem}\label{lem: power 1}
Suppose that $S=Q_1\times\dots\times Q_r$ is a product of simple rings, that $\sigma$ is not inner, and that $\sigma(Q_i)=Q_{i+1}$ (subscripts modulo $r$) for each $i$. Then $\sigma(t)=t$.
\end{lem}

\begin{proof}
Since $\sigma$ and $\delta$ commute, we have for all $q\in S$ 
\[
\delta(\sigma(q))=t\sigma(q)-\sigma^2(q)t=\sigma(\delta(q))=\sigma(t)\sigma(q)-\sigma^2(q)\sigma(t).
\]

Thus $(t-\sigma(t))\sigma(q)=\sigma^2(q)(t-\sigma(t))$, which implies that $t-\sigma(t)$ is a normal element of $S=Q_1\times\dots\times Q_r$. So write $t-\sigma(t)=(a_1,a_2,\dots,a_r)$, then each $a_i$ is a normal element in the simple ring $Q_i$, hence it is either 0 or a unit.

If $a_i$ is a unit for every $i=1,\dots,r$ then $t-\sigma(t)$ is a unit in $S$. However, taking $x:=\sigma^{-1}(q)$, we get that $(t-\sigma(t))x=\sigma(x)(t-\sigma(t))$. So if $t-\sigma(t)$ is a unit, it follows that $\sigma$ is conjugation by $t-\sigma(t)$, which implies that $\sigma$ is inner, contradicting our assumption.

Therefore, $a_i=0$ for some $i$. But if $\sigma(t)\neq t$, then $a_j\neq 0$ for some $j$, so there must exist $i$ such that $a_i=0$ but $a_{i+1}\neq 0$, and hence $a_{i+1}$ is a unit in $Q_{i+1}$. Fix such an $i$, and let $e_i$ be the standard $i$'th basis vector, so that $(t-\sigma(t))e_i=0$.

But $\sigma(Q_i)=Q_{i+1}$, so $0\neq\sigma(e_i)\in Q_{i+1}$, and hence $\sigma(e_i)a_{i+1}\neq 0$ since $a_{i+1}$ is a unit. Thus $$0=(t-\sigma(t))e_i=\sigma(e_i)(t-\sigma(t))\neq 0\text{ -- contradiction.}$$
Therefore, $\sigma(t)=t$.
\end{proof}

The following are some basic consequences of the assumption $\sigma(t) = t$.

\begin{lem}\label{lem: when sigma fixes t}
Suppose that $\sigma(t) = t$, and suppose that $(\sigma, \delta)$ is quasi-compatible with $w$, so that the bounded skew power series ring $S^+[[x; \sigma, \delta]]$ exists. Then
\begin{enumerate}[label=(\roman*)]
\item $x$ commutes with $t$,
\item $x-t$ is a normal element.
\end{enumerate}
\end{lem}

\begin{proof}
$ $

\begin{enumerate}[label=(\roman*)]
\item An easy calculation, using the relation $xt = \sigma(t)x + \delta(t)$.
\item We can check that, for any $s\in S$, we have $(x-t)s = \sigma(s)(x-t)$. Hence, if $(s_i)_{i\geq 0}$ is a bounded sequence of elements of $S$, then
$$(x-t)\left( \sum_{i\geq 0} s_i x^i\right) = \left( \sum_{i\geq 0} \sigma(s_i) x^i\right)(x-t),$$
and the sequence $(\sigma(s_i))_{i\geq 0}$ is still bounded as $\sigma$ is filtered.\qedhere
\end{enumerate}
\end{proof}

\begin{lem}\label{lem: power 2}
If $\sigma(t) = t$, then
\begin{align}\label{eqn: formula for delta^n}
\displaystyle \delta^n(s)=\underset{0\leq k\leq n}{\sum}{\binom{n}{k}(-1)^kt^{n-k}\sigma^k(s)t^k}
\end{align}
for all $n\in \mathbb{N}$ and all $s\in S$. In particular, if $S$ has characteristic $p$, then $\delta^{p^m}(s)=t^{p^m}s - \sigma^{p^m}(s)t^{p^m}$ for all $s\in S$.\end{lem}

\begin{proof}
Equation \eqref{eqn: formula for delta^n} is clearly true for $n=1$, and
$$\delta^{n+1}(s) =\delta(\delta^n(s))=\delta\left(\underset{0\leq k\leq n}{\sum}{\binom{n}{k}(-1)^kt^{n-k}\sigma^k(s)t^k}\right)=\underset{0\leq k\leq n}{\sum}{\binom{n}{k}(-1)^k\delta(t^{n-k}\sigma^k(s)t^k)}.$$
But we may calculate
$$\delta(t^{n-k}\sigma^k(s)t^k) = \delta(t^{n-k})\sigma^k(s)t^k + \sigma(t^{n-k})\delta(\sigma^k(s))t^k + \sigma(t^{n-k}\sigma^k(s))\delta(t^k)$$
as follows: by our assumptions on $\sigma$ and $\delta$, we have $\delta(t^s) = 0$ and $\sigma(t^s) = t^s$ for all $s$, and we have $\delta(\sigma^k(s)) = \sigma^k(ts - \sigma(s)t) = t\sigma^k(s) - \sigma^{k+1}(s)t$. Hence
\begin{align*}
\delta(t^{n-k}\sigma^k(s)t^k) &=\underset{0\leq k\leq n}{\sum}{\binom{n}{k}(-1)^kt^{n+1-k}\sigma^k(s)t^k}-\underset{0\leq k\leq n}{\sum}{\binom{n}{k}(-1)^kt^{n-k}\sigma^{k+1}(s)t^{k+1}}\\
&=t^{n+1}s+\underset{1\leq k\leq n}{\sum}{\left(\binom{n}{k}+\binom{n}{k-1}\right)(-1)^kt^{n+1-k}\sigma^k(s)t^k}+(-1)^{n+1}\sigma^{n+1}(s)t^{n+1}\\
&=\underset{0\leq k\leq n+1}{\sum}{\binom{n+1}{k}(-1)^kt^{n+1-k}\sigma^k(s)t^k},
\end{align*}
establishing \eqref{eqn: formula for delta^n} by induction.

Recalling that $\binom{p^m}{k} = 0$ in $S$ whenever $S$ has characteristic $p$ and $1 \leq k \leq p^m-1 $ (e.g. by Lemma \ref{lem: p-adic facts}(ii) applied to $(\mathbb{Z}_p, v_p)$), this implies that $\delta^{p^m}(s) = t^{p^m}s - \sigma^{p^m}(s)t^{p^m}$ for all $m\geq 0$ and all $s\in S$.
\end{proof}

The following definition is an analogue of Definition \ref{defn: X_n etc in char p} when $\delta$ is inner and defined by $t\in S$.

\begin{defn}\label{defn: X_n etc in inner case}
In this case, for any $n\in\mathbb{N}$, we set
\begin{itemize}
\item $T_n = (-1)^{p+1}t^{p^n}\in S$,
\item $X_n = (x-t)^{p^n} + T_n\in S[x;\sigma,\delta]$,
\item $\Sigma_n:=\sigma^{p^n}$,
\item $\Delta_n(s) = T_ns - \Sigma_n(s) T_n$ for all $s\in S$, an inner $\Sigma_n$-derivation.
\end{itemize}
Again, $(\Sigma_n,\Delta_n)$ is a (commuting) skew derivation on $S$, and we have $X_n s = \Sigma_n(s) X_n + \Delta_n(s)$ inside $S[x; \sigma, \delta]$ for all $s\in S$ by the calculations above.

If $S$ has characteristic $p$ and $\sigma(t) = t$, then $(x - t)^{p^n} = x^{p^n} - t^{p^n}$, and so $X_n$, $\Sigma_n$ and $\Delta_n$ agree with Definition \ref{defn: X_n etc in char p}.
\end{defn}

\begin{rk}\label{rk: relationship between X_n}
The relationship between $(t, x, \sigma, \delta)$ and $(T_n, X_n, \Sigma_n, \Delta_n)$ is the same as the relationship between $(T_i, X_i, \Sigma_i, \Delta_i)$ and $(T_{i+n}, X_{i+n}, \Sigma_{i+n}, \Delta_{i+n})$ for any $i$ and $n$: that is,
\begin{itemize}
\item $\Sigma_{i+n} = \Sigma_i^{p^n}$,
\item $T_{i+n} = (-1)^{p+1}T_i^{p^n}$, and $\Delta_{i+n}$ is the corresponding inner $\Sigma_{i+n}$-derivation,
\item $X_{i+n} = (X_i - T_i)^{p^n} + T_{i+n} \in S[X_i; \Sigma_i, \Delta_i]$.
\end{itemize}
\end{rk}

Define $\rho, \lambda: S\to S$ to be right- and left-multiplication by $t$, i.e.\ $\rho(s) = st$ and $\lambda(s) = ts$. This means we can write $\delta = \lambda - \rho\sigma$ and $(-1)^{p+1}\Delta_n = \lambda^{p^n} - \rho^{p^n} \sigma^{p^n}$. Clearly $\rho$ and $\lambda$ commute, and under the assumption that $\sigma(t) = t$, it is easily checked that they both commute with $\sigma$. So, noting that $\lambda^{p^n} = (\delta + \rho\sigma)^{p^n}$ and expanding binomially,
\begin{equation}\label{eqn: Delta_n}
\Delta_n = (-1)^{p+1}\sum_{i=1}^{p^n} \binom{p^n}{i} \delta^i (\rho\sigma)^{p^n-i}.
\end{equation}

\begin{propn}\label{propn: delta bounded implies Delta_n bounded}
Suppose now that $(S, w)$ is a complete filtered $\mathbb{Z}_p$-algebra and that $\sigma(t) = t$, and assume that $(\sigma, \delta)$ is strongly bounded with lower bound $A = \min\deg_w(\sigma^{\mathbb{Z}} \delta^{\mathbb{N}})$. Moreover, if $p\neq 0$ in $S$, we will assume that $w(t) \geq 0$.
\begin{enumerate}[label=(\roman*)]
\item Take $d\geq 0$. If there exists $K_1$ such that $\deg(\delta^i) \geq d-A$ for all $i \geq K_1$, then there exists $K_2$ such that $\deg(\Delta_n) \geq d$ for all $n \geq K_2$.
\item Suppose we are given $d \geq A$, $n\in\mathbb{N}$ such that $\sigma^{p^r}$ is strictly filtered and $\deg(\Delta_r) \geq d$ for all $r\geq n$. Then there exists some $N$ such that $\deg(\delta^{p^N}) \geq d$.
\end{enumerate}
\end{propn}

\begin{proof}

First note that if $p=0$ in $S$ then both statements are obvious, because $\Delta_n=\delta^{p^n}$ for all $n$. So from now on, we can assume that $p\neq 0$, and hence $w(t)\geq 0$.
$ $

\begin{enumerate}[label=(\roman*)]
\item Note that $\deg(\delta^i (\rho\sigma)^{p^n-i}) \geq \deg(\delta^i) + \deg(\rho^{p^n-i}) + \deg(\sigma^{p^n-i}) \geq \deg(\delta^i) + A$ by our assumptions on $t$ and $\sigma$. So, for any $n\in\mathbb{N}$, by \eqref{eqn: Delta_n},
\begin{equation}\label{eqn: degree of Delta_n}
\displaystyle \deg(\Delta_n) \geq \min_{1\leq i\leq p^n} \left\{w\left(\binom{p^n}{i}\right) + \deg(\delta^i)\right\} + A.
\end{equation}
Let $K_2:=\lfloor \log_p(K_1)\rfloor + d - 2A$, and fix $n\geq K_2$. Using \eqref{eqn: degree of Delta_n}, we only need to prove that $\displaystyle w\left(\binom{p^n}{i}\right) + \deg(\delta^i) \geq d-A$ for all $1\leq i\leq p^n$.
\begin{itemize}
\item Since $w\left(\binom{p^n}{i}\right)\geq 0$, the inequality holds for $i \geq K_1$ by our assumption on $K_1$.
\item Suppose $1\leq i< K_1$. Then by Lemma \ref{lem: p-adic facts}(ii), $$w\left(\binom{p^n}{i}\right)\geq n - \lfloor \log_p(K_1)\rfloor\geq K_2-\lfloor \log_p(K_1)\rfloor=d-2A$$ and $\deg(\delta^i) \geq A$ by assumption. Hence we have $w(\binom{p^n}{i}) + \deg(\delta^i) \geq  d - 2A+A=d-A$ in this case as well.
\end{itemize}

\item For all $i\geq 1$, define $\varphi(i) = \deg(\delta^i \sigma^{-i})$. Notice that $\deg(\delta^i (\rho\sigma)^{p^k-i}) \geq \varphi(i)$ whenever $k\geq n$ by our assumptions on $t$ and $\sigma$. We will show that, for all $i\geq 1$,
\begin{align}\label{eqn: bound for phi_i}
\varphi(i) \geq \min\{d, v_p(i) - n + A + 1\}.
\end{align}
The desired result will then follow by taking $N \geq \max\{d + n - A - 1, n\}$, as then we will have $\deg(\delta^{p^N}) = \deg(\delta^{p^N} \sigma^{-p^N}) = \varphi(p^N) \geq d$.

We will prove \eqref{eqn: bound for phi_i} using a double induction. More specifically, we will show that, for all $i\geq 1$ and all $k\in\mathbb{N}$, we have
\begin{itemize}
\item $\varphi(i) \geq \min\{d, v_p(i) - n + A + 1\}$ for $0\leq v_p(i)\leq k$,
\item $\varphi(i) \geq \min\{d, k - n + A + 1\}$ for $v_p(i) \geq k+1$.
\end{itemize}
As we always have $\varphi(i) \geq A$, this is true for all $i\geq 1$ whenever $k \leq n-1$. So suppose it is true for all $i\geq 1$ and all $k = 0, 1, \dots, \ell-1$, where $\ell \geq n$: we will show it is true for all $i\geq 1$ and $k = \ell$. 

If $v_p(i)<\ell$ then the result follows from induction, so we may assume that $p^\ell\mid i$. Thus we need only show that $\varphi(mp^\ell) \geq \min\{d, \ell-n+A+1\}$ for all $m\geq 1$.

In order to show this, we will assume that we are given some positive integer $m$ such that $\varphi(i) \geq \min\{d, v_p(i) - n + A + 1\}$ for all $i \leq mp^\ell - 1$. (This is true when $m = 1$.) We will then show it is true for $i = mp^\ell$, and it will follow that it is true for all $i \leq (m+1)p^\ell - 1$ by induction. By raising \eqref{eqn: Delta_n} to the $m$th power, we can calculate $\Delta_\ell^m$, and we get
$$\Delta_\ell^m - (-1)^{(p+1)m}\delta^{mp^\ell} = (-1)^{(p+1)m}\sum_{i_1, \dots, i_m} \binom{p^\ell}{i_1}\dots \binom{p^\ell}{i_m} \delta^I (\rho\sigma)^{mp^\ell - I},$$
where we write $I = i_1 + \dots + i_m$, and the sum ranges over all $m$-tuples $(i_1, \dots, i_m)$ such that $1\leq i_e \leq p^\ell$ for all $1\leq e\leq m$ and $i_1 + \dots + i_m \leq mp^\ell - 1$. Then
\begin{align*}
\deg(\Delta_\ell^m \pm \delta^{mp^\ell}) &= \deg\left(\sum_{i_1, \dots, i_m} \binom{p^\ell}{i_1}\dots \binom{p^\ell}{i_m} \delta^I (\rho\sigma)^{mp^\ell - I}\right)\\
&\geq \min_{i_1, \dots, i_m} \left\{ \deg\left(\binom{p^\ell}{i_1}\dots \binom{p^\ell}{i_m} \delta^I (\rho\sigma)^{mp^\ell - I}\right)\right\}\\
&\geq \min_{i_1, \dots, i_m} \left\{ \sum_{e=1}^m w\left(\binom{p^\ell}{i_e}\right) + \deg( \delta^I (\rho\sigma)^{mp^\ell - I})\right\}\\
&\geq \min_{i_1, \dots, i_m} \left\{ \sum_{e=1}^m (\ell-v_p(i_e)) + \varphi(I)\right\}\\
&\geq \min_{i_1, \dots, i_m} \left\{ \sum_{e=1}^m (\ell-v_p(i_e)) + \min\{d, v_p(I) - n + A + 1\}\right\},
\end{align*}
where the penultimate inequality follows as $mp^\ell$ is a multiple of $p^r$, and the final inequality comes from the induction hypothesis. (The $\pm$ symbol here means $+$ if $p=2$ and $m$ is odd, and $-$ otherwise.)

Now, for any fixed tuple $(i_1, \dots, i_m)$, suppose that $1\leq f\leq m$ is such that $v_p(i_f)$ is minimal: then since $0\leq v_p(i_e) \leq \ell$ for all $e$, we have $\sum_{e=1}^m (\ell-v_p(i_e)) \geq \ell-v_p(i_f)$, and $v_p(I) \geq v_p(i_f)$, so
\begin{align*}
\sum_{e=1}^m (\ell-v_p(i_e)) + \min\{d, v_p(I) - n + A + 1\} &\geq \ell-v_p(i_f) + \min\{d, v_p(i_f) - n + A + 1\}\\
&\geq \min\{d, \ell - n + A + 1\},
\end{align*}
and so we can conclude that $\deg(\Delta_\ell^m \pm \delta^{mp^\ell}) \geq \min\{d, \ell-n+A+1\}$. But $\deg(\Delta_\ell^m) \geq md \geq d$ by the assumption that $\ell\geq n$, and so $\varphi(mp^\ell) = \deg(\delta^{mp^\ell}) \geq \min\{d, \ell-n+A+1\}$ as required. \qedhere
\end{enumerate}
\end{proof}

\begin{thm}\label{thm: skew power series subrings exist}
Suppose that $(S,w)$ is a complete filtered $\mathbb{Z}_p$-algebra and $(\sigma, \delta)$ is a commuting skew derivation on $S$ such that $\sigma^{p^r}$ is strictly filtered for some $r$. Suppose there exists $t\in S$ such that $\delta(s)=ts-\sigma(s)t$ for all $s\in S$, satisfying $\sigma(t) = t$, and we assume that $w(t)\geq 0$ when $p\neq 0$ in $S$. Then the following are equivalent:
\begin{enumerate}[label=(\alph*)]
\item $(\sigma, \delta)$ is quasi-compatible with $w$.
\item $(\Sigma_n, \Delta_n)$ is quasi-compatible with $w$ for \emph{all} $n\in\mathbb{N}$.
\item $(\Sigma_n, \Delta_n)$ is quasi-compatible with $w$ for \emph{some} $n\in\mathbb{N}$.
\end{enumerate}

In particular, $S^+[[x; \sigma, \delta]]$ exists if and only if any $S^+[[Y; \Sigma_n, \Delta_n]]$ exists. Moreover, for any sufficiently small rational $\varepsilon > 0$, the natural inclusion map $\theta:(S[Y;\Sigma_n,\Delta_n], v_{Y, \varepsilon})\to (S^+[[x;\sigma,\delta]], v_{x, \varepsilon})$ extending the identity map on $S$ and sending $Y$ to $X_n$ is a filtered ring homomorphism of non-negative degree.
\end{thm}

\begin{proof}

If $p=0$ in $S$ then this follows from Proposition \ref{propn: subrings in characteristic p}, so we may assume that $p\neq 0$ and $w(t)\geq 0$.

The implications (b) $\implies$ (a) $\implies$ (c) are obvious, since $(\sigma, \delta) = (\Sigma_0, \Delta_0)$.

To show (c) $\implies$ (b), we assume that $(\Sigma_n, \Delta_n)$ is quasi-compatible with $w$, and aim to show that $(\Sigma_\ell, \Delta_\ell)$ is quasi-compatible. First, note that $\Sigma_\ell^{p^r}$ is strictly filtered, so by Corollary \ref{cor: criterion for quasi-compatibility}, we need to show that there exists some $m$ such that $\deg(\Delta_\ell^m) \geq 1$.

Since $(\Sigma_n, \Delta_n)$ is quasi-compatible, there exists $A\leq 0$ such that $\min\deg(\Sigma_n^\mathbb{Z} \Delta_n^\mathbb{N}) = A$, and there exists $N$ such that $\deg(\Delta_n^N) \geq 1$. Hence $\deg(\Delta_n^{N(1-2A)})\geq 1 - 2A$, and it follows that $\deg(\Delta_n^M) \geq (1 - 2A) + A = 1-A$ for all $M\geq N(1-2A)$. In particular, by Proposition \ref{propn: delta bounded implies Delta_n bounded}(i) with $\Delta_n$ in place of $\delta$, there exists $K$ such that $\deg(\Delta_k) \geq 1$ for all $k\geq K$.

Choose $k \geq r$ (so that $\Sigma_\ell^{p^k}$ is strictly filtered) and $k \geq \ell$. Then, using Proposition \ref{propn: delta bounded implies Delta_n bounded}(ii) with $(\Sigma_\ell, \Delta_\ell)$ in place of $(\sigma, \delta)$, we get some $M$ such that $\deg(\Delta_\ell^{p^M})\geq 1$. Hence $(\Sigma_\ell, \Delta_\ell)$ is quasi-compatible.

Now assume that these three equivalent conditions hold. We can write $X_n:=(x-t)^{p^n}+(-1)^{p+1}t^{p^n}$ as $X_n=g(x,t)x$, where $g$ is a polynomial in $x$ and $t$, with coefficients in $\mathbb{Z}$. So since $w(t)\geq 0$, it follows that for every $m\in\mathbb{N}$, $$g(x,t)^m=\underset{k\geq 0}{\sum}{\gamma_{k,m}x^k}$$ where $w(\gamma_{k,m})\geq 0$ for each $k$. So we can write $\theta$ as
\begin{align*}
\theta:S[Y;\Sigma_n,\Delta_n]&\to S^+[[x;\sigma,\delta]],\\
\underset{m\geq 0}{\sum}{\lambda_mY^m}&\mapsto\underset{m\geq 0}{\sum}{\lambda_mX_n^m}=\underset{m\geq 0}{\sum}{\lambda_mg(x,t)^mx^m}=\underset{m\geq 0}{\sum}\underset{k\geq 0}{\sum}{\lambda_m\gamma_{k,m}x^{m+k}}.
\end{align*}
Clearly this is a ring homomorphism. Fix any appropriate $\varepsilon>0$: then, given arbitrary
$$f=\underset{m\geq 0}{\sum}{\lambda_mY^m}\in S[Y;\Sigma_n;\Delta_n]$$
with $v_{Y,\varepsilon}(f)=\inf\{w(\lambda_m)+\varepsilon m:m\geq 0\}=N$, we see that
\begin{align*}
v_{x,\varepsilon}(\theta(f)) &= \inf\left\{w\left(\underset{i+j=m}{\sum}{\lambda_i\gamma_{j,i}}\right)+\varepsilon m:m\geq 0\right\}\\
&\geq\inf\{\min\{w(\lambda_i)+w(\gamma_{j,i}):i+j=m\}+\varepsilon m:m\geq 0\}\\
&\geq\inf\{\min\{w(\lambda_i)+\varepsilon m:i\leq m\}:m\geq 0\}\\
&\geq\inf\{\min\{w(\lambda_i)+\varepsilon i:i\leq m\}:m\geq 0\}\\
&=\inf\{w(\lambda_m)+\varepsilon m:m\geq 0\}=N,
\end{align*}
hence $\theta$ is filtered of non-negative degree as required.\end{proof}

\subsection{Infinite matrices and filtered homomorphisms}\label{subsec: continuous homomorphisms}

Theorem \ref{thm: skew power series subrings exist} is almost an analogue of Proposition \ref{propn: subrings in characteristic p}, but it is missing the detail that we have not necessarily proved that $S^+[[X_n;\Sigma_n,\Delta_n]]$ is a subring. We now want to prove the full analogue.

Let $(S,w)$ be a complete filtered ring. We now consider \emph{infinite} matrices over $S$ (by which we mean matrices $A = (a_{i,j})$ whose rows and columns are indexed by $\mathbb{N}$) and \emph{infinite} row vectors (by which we mean row vectors $\mathbf{v} = (v_i)$ whose entries are indexed by $\mathbb{N}$). Recall that, for us, $0\in\mathbb{N}$, so the indexing will start at $0$.

\begin{defn}
$ $

\begin{enumerate}[label=(\roman*)]
\item An infinite matrix $A$ with entries in $S$ is \emph{bounded} (with respect to $w$) if there exists some $N\in\mathbb{Z}$ such that, for all $i,j\in\mathbb{N}$, $w(a_{i,j}) \geq N$. We will call $N$ a \emph{lower bound} for $A$. Similarly, an infinite row vector $\mathbf{v}$ with entries in $S$ is \emph{bounded} if there exists $N\in\mathbb{Z}$ such that, for all $i\in\mathbb{N}$, $w(v_i) \geq N$.

Below, bounded row vectors $\mathbf{s} = \begin{pmatrix}
s_0 & s_1 & s_2 & \dots
\end{pmatrix}$ will correspond to elements of the bounded power series module $s_0 + s_1x + s_2x^2 + \dots\in S^+[[x]]$.

\item An infinite matrix $A$ is \emph{row-null} if each row is a null sequence, i.e.\ if $\displaystyle \lim_{j\to\infty} w(a_{i,j}) = \infty$ for all $i$. (This is a generalisation of the notion of \emph{row-finiteness}.) Similarly, $A$ is \emph{column-null} if $\displaystyle \lim_{j\to\infty} w(a_{j,k}) = \infty$ for all $k$.

\item If $A$ and $B$ are infinite matrices, we will say that the product $AB$ \emph{is defined} if the infinite sum $x_{i,k} := \sum_{j\in\mathbb{N}} a_{i,j} b_{j,k}$ converges in $S$ for all $i$ and $k$, and we define it to be the infinite matrix $AB := (x_{i,k})_{i,k\in\mathbb{N}}$. Similarly, if $\mathbf{v}$ is an infinite row vector and $A$ is an infinite matrix, the product $\mathbf{v}A$ \emph{is defined} if $y_j := \sum_{i\in\mathbb{N}} v_i a_{i,j}$ converges in $S$ for all $j$, and we define it to be the infinite row vector $\mathbf{v}A = (y_j)_{j\in\mathbb{N}}$.
\end{enumerate}
\end{defn}

The following lemma is easy, and we omit its proof.

\begin{lem}\label{lem: sums of column-null matrices are column-null}
Let $A$ and $B$ be infinite matrices with entries in $S$, and suppose they are column-null. Then $A+B$ is column-null.\qed
\end{lem}

\begin{lem}\label{lem: products of matrices are defined and bounded}
Let $A$ and $B$ be infinite matrices with entries in $S$, and let $\mathbf{v}$ be an infinite row vector with entries in $S$.

\begin{enumerate}[label=(\roman*)]
\item Suppose $A$ and $B$ are bounded with lower bounds $M$ and $N$ respectively. If either $A$ is row-null or $B$ is column-null, then the product $AB$ is defined and is bounded with lower bound $M+N$.
\item Suppose $\mathbf{v}$ and $A$ are bounded with lower bounds $L$ and $M$ respectively. If $A$ is column-null, then the product $\mathbf{v}A$ is defined and is bounded with lower bound $L+M$.
\end{enumerate}
\end{lem}

\begin{proof}
Suppose we are in the situation of part (i), and $B$ is column-null. Fix arbitrary $i,k\in\mathbb{N}$. Then $w(a_{i,j}) \geq M$ for all $j\in\mathbb{N}$, and $w(b_{j,k}) \to \infty$ as $j\to\infty$, so the sequence $(a_{i,j} b_{j,k})_{j\in\mathbb{N}}$ converges to $0$. As $S$ is complete, the sum $\sum_{j\in\mathbb{N}} a_{i,j} b_{j,k}$ converges \cite[Chapter I, \S 3.3, Proposition 4]{LVO}. Moreover, since $w(a_{i,j} b_{j,k}) \geq w(a_{i,j}) + w(b_{j,k}) \geq M + N$ for all $j\in\mathbb{N}$, we must have $w\left(\sum_{j\in\mathbb{N}} a_{i,j} b_{j,k}\right) \geq M+N$.

The other cases of this lemma are proved very similarly, so we omit their proofs.
\end{proof}

Part (ii) of this lemma allows us to explicitly define filtered left $S$-module homomorphisms out of $S^+[[x]]$. If $A$ is bounded with lower bound $N$ and column-null, then right multiplication by $A$ sends bounded infinite row vectors $\mathbf{s}$ to bounded infinite row vectors $\mathbf{s}A$. That is, the map
\begin{align*}
\rho_A: S^+[[x]] &\to S[[y]]\\
\sum_{i\in\mathbb{N}} s_i x^i &\mapsto \sum_{j\in\mathbb{N}} \left(\sum_{i\in\mathbb{N}} s_i a_{i,j} \right) y^j
\end{align*}
is well-defined and has image in $S^+[[y]]$.

If $S^+[[x]]$ and $S^+[[y]]$ are given the filtrations $v_{x,\varepsilon}$ and $v_{y,\varepsilon}$ respectively, then $\deg(\rho_A) \geq N$.

\begin{lem}\label{lem: product of column-nulls is column-null}
Suppose $A$ and $B$ are bounded and column-null. Then $AB$ is column-null.
\end{lem}

\begin{proof}
Fix arbitrary $k\in\mathbb{N}$, and write $c_k = \begin{pmatrix}
x_0&x_1&x_2&\dots
\end{pmatrix}^T$ for column $k$ of $AB$, i.e.\ $\displaystyle x_i = \sum_{j\in\mathbb{N}} a_{i,j} b_{j,k}$. Write $L$ and $M$ for lower bounds of $A$ and $B$ respectively.

Choose arbitrary $n\in\mathbb{N}$. Then, since $B$ is column-null, there exists some $J\in\mathbb{N}$ such that $w(b_{j,k}) \geq n-L$ for all $j\geq J$. Moreover, since $A$ is column-null, we can find some $I\in\mathbb{N}$ such that $w(a_{i,j}) \geq n-M$ for all $i\geq I$ and all $j<J$. In either case, $w(a_{i,j}b_{j,k}) \geq n$ whenever $i\geq I$, and so $w(x_i) \geq n$ whenever $i\geq I$. As $n$ was arbitrary, we have shown that $w(x_i) \to \infty$ as $i\to\infty$. Hence $c_k$ is a null sequence as required.
\end{proof}

The proof of the following result can be compared to \cite[Corollary 6.11]{DDMS}.

\begin{lem}\label{lem: associativity of column-null matrices}
Suppose $\mathbf{v}$ is a bounded infinite row vector, and $A$ and $B$ are bounded, column-null infinite matrices. Then $\mathbf{v}(AB) = (\mathbf{v}A)B$.
\end{lem}

\begin{proof}
By Lemmas \ref{lem: products of matrices are defined and bounded} and \ref{lem: product of column-nulls is column-null}, the products $AB$, $\mathbf{v}(AB)$, $\mathbf{v}A$ and $(\mathbf{v}A)B$ are defined, and $\mathbf{v}(AB)$ and $(\mathbf{v}A)B$ are infinite row vectors. Fix some arbitrary $k\in\mathbb{N}$: we will show that entry $k$ of $\mathbf{v}(AB)$ is equal to entry $k$ of $(\mathbf{v}A)B$, i.e.\ that
$$(\mathbf{v}(AB))_k = \sum_{i\in\mathbb{N}} \sum_{j\in\mathbb{N}} v_i a_{i,j} b_{j,k} = \sum_{j\in\mathbb{N}} \sum_{i\in\mathbb{N}} v_i a_{i,j} b_{j,k} = ((\mathbf{v}A)B)_k.$$

Write $K$, $L$ and $M$ for lower bounds of $\mathbf{v}$, $A$ and $B$ respectively.

First we will observe that, for each $N\in\mathbb{Z}$, there are only finitely many $(i,j)\in\mathbb{N}^2$ such that $w(v_i a_{i,j} b_{j,k}) \leq N$. Indeed, since $B$ is column-null, there exists $J\in\mathbb{N}$ such that $w(b_{j,k}) > N-K-L$ for all $j\geq J$, and since $A$ is column-null, there exists $I\in\mathbb{N}$ such that $w(a_{i,j}) > N-K-M$ for all $i\geq I$ and $j < J$. In either case, we get $w(v_i a_{i,j} b_{j,k}) > N$, and so the only $(i,j)\in\mathbb{N}^2$ satisfying $w(v_i a_{i,j} b_{j,k}) \leq N$ must satisfy $i<I$ and $j<J$. Write $T_N = \{(i,j)\in\mathbb{N}^2 : w(v_i a_{i,j} b_{j,k}) \leq N\}$ for this finite set.

Next, for all $N\in\mathbb{N}$, define $y_N := \sum_{(i,j)\in T_N} v_i a_{i,j} b_{j,k}$, a finite sum in $S$. The sequence $(y_N)_{N\in\mathbb{N}}$ is clearly Cauchy, and so converges in $S$ to some element $y$.

For all $i\in\mathbb{N}$, define $s_i = \sum_{j\in\mathbb{N}} v_i a_{i,j} b_{j,k}$; likewise, for all $j\in\mathbb{N}$, define $t_j = \sum_{i\in\mathbb{N}} v_i a_{i,j} b_{j,k}$. We are aiming to show that $\sum_{i\in\mathbb{N}} s_i = \sum_{j\in\mathbb{N}} t_j$. So fix some $N\in\mathbb{Z}$, and choose $I$ large enough so that $(i,j)\in T_N \implies i \leq I$: then
$$w\left(y_N - \sum_{i=0}^I s_i\right) > N \text{ and } w(y_N - y) > N,$$
from which it follows that $w(y - \sum_{i=0}^I s_i) > N$.

Hence, as $N$ was arbitrary, we have shown that $\sum_{i\in\mathbb{N}} s_i = y$. An almost identical argument shows $\sum_{j\in\mathbb{N}} t_j = y$.
\end{proof}

\begin{rk}
A very similar proof will show that, if $A$, $B$ and $C$ are bounded infinite matrices and $B$ and $C$ are column-null, then $(AB)C = A(BC)$. Likewise, if all three matrices are bounded and $A$ and $B$ are row-null, then $(AB)C = A(BC)$. Compare this to similar results for infinite matrices in the column-finite case: \cite[Theorem A]{Wan97}, \cite[Lemma 4]{ZhaDan06}, \cite[Theorem 21]{BosLop19}.

In particular, the set $\mathbf{M}(S)$ of bounded, column-null infinite matrices over $S$ forms a ring. So, if we write $S^\infty$ for the left $S$-module of bounded infinite row vectors, we get a ring homomorphism $\mathbf{M}(S) \to \mathrm{End}(S^\infty)$, where the matrix $A$ acts on $S^\infty$ by the map $(\mathbf{s} \mapsto \mathbf{s}A)$.
\end{rk}

Parts (i) and (ii) of the following lemma are well-known and used in combinatorics.

\begin{lem}\label{lem: infinite matrices}
Suppose $A = I-U$ is an infinite matrix, where $U$ is \emph{strictly lower-triangular}, i.e.\ $i\leq j\implies u_{i,j} = 0$. Then:
\begin{enumerate}[label=(\roman*)]
\item the sum $U + U^2 + U^3 + \dots$ converges entrywise, and the matrix $V = U + U^2 + U^3 + \dots$ is strictly lower-triangular,
\item $B = I + V$ satisfies $AB = BA = I$,
\item if $U$ is bounded with lower bound $1$ and is column-null, then $V$ is bounded with lower bound $1$ and is column-null.
\end{enumerate}
\end{lem}

\begin{proof}
As $U$ is row-finite, Lemma \ref{lem: products of matrices are defined and bounded}(i) ensures that the matrices $U^n$ are all defined, and an easy calculation shows that the first $n$ rows of $U^n$ are all zero: in particular, for each fixed $(i,j)\in\mathbb{N}^2$, only finitely many of the matrices $U$, $U^2$, $U^3$, $\dots$ have non-zero $(i,j)$-entry. Hence each entry of $V$ is in fact a finite sum. Since each $U^n$ is strictly lower-triangular, this establishes (i).

To show (ii), note that $A$ and $B$ are both row-finite, and so each entry of $AB$ and $BA$ requires only a finite calculation. The property $AB = BA = I$ can now be checked entrywise by truncating $A$ and $B$, i.e.\ by considering the square submatrices of $A$ and $B$ consisting of the first rows and columns, and then using the corresponding property of \emph{finite} lower-triangular matrices.

For part (iii): assume that $U$ has lower bound $1$. Then each $U^n$ has lower bound $n$ by Lemma \ref{lem: products of matrices are defined and bounded}(i), and is column-null by an inductive use of Lemma \ref{lem: product of column-nulls is column-null}. Now, to check that $V$ is column-null, fix some $N\in\mathbb{Z}$: we need to show that each column eventually has value greater than $N$. But this is true for the finite sum $U + U^2 + \dots + U^N$, by Lemma \ref{lem: sums of column-null matrices are column-null}, and adding the matrices $U^{N+1}, U^{N+2}, \dots$ will not affect this, as each of these matrices has lower bound $N+1$.
\end{proof}

\begin{lem}\label{lem: ring hom between bounded skew power series rings}
Suppose that $(S, w)$ is a complete filtered $\mathbb{Z}_p$-algebra and $(\sigma, \delta)$ is a commuting skew derivation on $S$, quasi-compatible with $w$. Suppose that there exists $t\in S$ such that $\delta(s) = ts - \sigma(s)t$ for all $s\in S$, satisfying $\sigma(t) = t$. Assume further that $\sigma^{p^r}$ is strictly filtered for some $r\in\mathbb{N}$, and if $p\neq 0$ in $S$ then $w(t) \geq 0$.

Then the continuous ring homomorphism $\theta: S[Y; \Sigma_n, \Delta_n]\to S^+[[x; \sigma, \delta]]$ of Theorem \ref{thm: skew power series subrings exist} extends uniquely to a continuous ring homomorphism $\theta: S^+[[Y; \Sigma_n, \Delta_n]]\to S^+[[x; \sigma, \delta]]$.
\end{lem}

\begin{proof}

If $p=0$ in $S$ then this is clear, because $\theta$ sends $Y$ to $x^{p^n}$, so it extends to a map from $S[[Y; \Sigma_n, \Delta_n]]$ whose image is the set of power series in $x^{p^n}$, as in Proposition \ref{propn: subrings in characteristic p}. So we may assume that $p\neq 0$ and $w(t)\geq 0$.

Since $\theta$ is a continuous ring homomorphism, it extends to a map between the completions $$\widehat{\theta}:\widehat{S[Y;\Sigma_n,\Delta_n]}\to\widehat{S^+[[x;\sigma,\delta]]},$$ which remains a continuous ring homomorphism \cite[Chapter I, Theorem 3.4.5]{LVO}.

But $S[Y;\Sigma_n,\Delta_n]$ is dense in $S^+[[Y;\Sigma_n,\Delta_n]]$ by Theorem \ref{thm: restricted skew power series rings exist}(i), so it follows that $S^+[[Y;\Sigma_n,\Delta_n]]$ is contained in $\widehat{S[Y;\Sigma_n,\Delta_n]}$. Thus it remains to prove that the image of $S^+[[Y;\Sigma_n,\Delta_n]]$ under $\widehat{\theta}$ is contained in $S^+[[x;\sigma,\delta]]$.

More precisely, given $f=\underset{k\in\mathbb{N}}{\sum}{s_kY^k}\in S^+[[Y;\Sigma_n,\Delta_n]]$, $f$ is the limit in $S^+[[Y;\Sigma_n,\Delta_n]]$ of the sequence of polynomials $f_N:=\underset{k\leq N}{\sum}{s_kY^k}\in S[Y;\Sigma_n,\Delta_n]$. So the sequence $\theta(f_N)\in S^+[[x;\sigma,\delta]]$ converges to $\widehat{\theta}(f)\in\widehat{S^+[[x;\sigma,\delta]]}$ by continuity, and thus it remains to prove that $\theta(f_N)$ converges to an element of $S^+[[x;\sigma,\delta]]$.

For each $k\in\mathbb{N}$, write $A_k(x):=X_n^k=((x-t)^{p^n}+(-1)^{p+1}t^{p^n})^k$, a monic polynomial in $x$ with entries in $S$. Let the $x^j$-coefficient of $A_k(x)$ be $a_{k,j}$, and set $A = (a_{k,j})$, an infinite matrix.

Since $A_k(x)$ has lowest non-zero entry of degree at least $k$, we see that $A$ is row and column-finite. Setting $S^{\infty}$ as the left $S$-module of bounded infinite row vectors, for any $\mathbf{s}\in S^{\infty}$, the infinite row vector $\mathbf{s}A$ is defined and bounded by Lemma \ref{lem: products of matrices are defined and bounded}(ii).

Therefore, let $\mathbf{s}:=(s_0,s_1,\dots)\in S^{\infty}$ be the vector of coefficients for $f$, and let $\mathbf{r}:=\mathbf{s}A$, say $r=(r_0,r_1,\dots)\in S^{\infty}$. We will prove that $\theta(f_N)$ converges to the bounded power series $g:=\underset{k\in\mathbb{N}}{\sum}{r_kx^k}\in S^+[[x;\sigma,\delta]]$.

Since $\theta(Y)=X_n$, it is clear that $\theta(f_N)=\underset{k\leq N}{\sum}{s_kX_n^k}=\underset{k\leq N}{\sum}{s_kA_k(x)}$. So let $A_N$ be the first $N$ rows of $A$, and let $\mathbf{s}_N:=(s_0,\dots,s_N)\in S^N$, and it follows that $\mathbf{s}_NA_N$ is the vector of coefficients of the polynomial $\theta(f_N)$.

But since $A$ is column finite, this means that for each $j\in\mathbb{N}$, and for all sufficiently high $N$, the $j$'th entry of $\mathbf{s}_NA_N$ is equal to the $j$'th entry of $\mathbf{s}A$, which is $r_j$. In other words, the $x^j$ coefficient of $\theta(f_N)$ is $r_j$ for sufficiently high $N$, and thus $\theta(f_N)$ converges to $g$ in $S^+[[x;\sigma,\delta]]$ as required.\end{proof}

\begin{thm}\label{thm: well-defined subring}
Let $(S,w)$ be a complete, filtered $\mathbb{Z}_p$-algebra, and let $(\sigma,\delta)$ be a quasi-compatible commuting skew derivation on $S$. Suppose we are in one of the following two situations:
\begin{enumerate}[label=(\arabic*)]
\item $S$ has characteristic $p$.
\item There exists $t\in S$ such that $\delta(s)=ts-\sigma(s)t$ for all $s\in S$, satisfying $w(t)\geq 0$ and $\sigma(t) = t$, and $\sigma^{p^r}$ is strictly filtered for some $r\in\mathbb{N}$.
\end{enumerate}

Then $S^+[[X_n;\Sigma_n,\Delta_n]]$ is a subring of $S^+[[x;\sigma,\delta]]$ for all $n$, and $S^+[[x;\sigma,\delta]]$ is freely generated as a module over $S^+[[X_n;\Sigma_n,\Delta_n]]$ with basis $\{x^i:i=0,\dots,p^n-1\}$.\end{thm}

\begin{proof}
In situation (1), this is just Proposition \ref{propn: subrings in characteristic p}, so we focus on situation (2). In situation (2), $S^+[[Y;\Sigma_n,\Delta_n]]$ exists by Theorem \ref{thm: skew power series subrings exist}, and there exists a continuous ring homomorphism $\theta:S^+[[Y;\Sigma_n,\Delta_n]]\to S^+[[x;\sigma,\delta]]$ sending $Y$ to $X_n$ by Lemma \ref{lem: ring hom between bounded skew power series rings}.

We aim to show that
$$S^+[[x; \sigma, \delta]] = \bigoplus_{i=0}^{p^n-1} S^+[[X_n; \Sigma_n, \Delta_n]] x^i.$$
More precisely, we can define the function $\displaystyle \Theta: \bigoplus_{i=0}^{p^n-1} S^+[[Y]]y^i \to S^+[[x]]$ by
$$\Theta(f^{(0)}(Y), f^{(1)}(Y) y, \dots, f^{(p^n-1)}(Y) y^{p^n-1}) = \theta(f^{(0)}(Y)) + \theta(f^{(1)}(Y)) x + \dots + \theta(f^{(p^n-1)}(Y)) x^{p^n-1},$$
where the $f^{(i)}$ are bounded power series with coefficients in $S$, and $Y,y$ are formal variables. It is clear that $\Theta$ is a $S^+[[Y;\Sigma_n,\Delta_n]]$-module homomorphism, as it is the extension of $\theta$ to the direct sum.

Write $S^\infty$ for the left $S$-module of bounded infinite row vectors with entries in $S$. For each $0\leq i\leq p^n-1$, write $\mathbf{f}^{(i)} = \begin{pmatrix}
f^{(i)}_0&f^{(i)}_1&f^{(i)}_2&\dots
\end{pmatrix}\in S^\infty$ for the vector of coefficients of $f^{(i)}$, so that
$$f^{(i)}(Y) = f^{(i)}_0 + f^{(i)}_1 Y + f^{(i)}_2 Y^2 + \dots$$
and
\begin{align*}
\Theta(f^{(0)}(Y), f^{(1)}(Y) t, \dots, f^{(p^n-1)}(Y) t^{p^n-1}) &= f^{(0)}_0 + f^{(1)}_0 x + \dots + f^{(p^n-1)}_0 x^{p^n-1}\\
&+ f^{(0)}_1 X_n + f^{(1)}_1 X_n x + \dots + f^{(p^n-1)}_1 X_n x^{p^n-1}\\
&+ f^{(0)}_2 X_n^2 + f^{(1)}_2 X_n^2 x + \dots + f^{(p^n-1)}_2 X_n^2 x^{p^n-1} + \dots
\end{align*}
We want to write the right-hand side as a bounded power series in $x$: that is, we aim to find a bounded vector of coefficients $\mathbf{s} = \begin{pmatrix}
s_0 & s_1 & s_2 & \dots
\end{pmatrix}$ such that the right-hand side above is equal to $s_0 + s_1 x + s_2 x^2 + \dots$

Denote by $\varphi: (S^\infty)^{p^n} \to S^\infty$ the ``interleaving" map
$$\varphi(\mathbf{f}) = \begin{pmatrix}
f^{(0)}_0&f^{(1)}_0&\dots&f^{(p^n-1)}_0&f^{(0)}_1&f^{(1)}_1&\dots&f^{(p^n-1)}_1&\dots
\end{pmatrix},$$
where $\mathbf{f} = (\mathbf{f}^{(0)}, \mathbf{f}^{(1)}, \dots, \mathbf{f}^{(p^n-1)})$. Defined in this way, entry $m$ of $\varphi(\mathbf{f})$ is $f^{(\ell)}_k$, where $m = kp^n + \ell$ for unique $k\in\mathbb{N}$ and $0\leq \ell\leq p^n-1$. Next, continuing to write $m = kp^n + \ell$, let $A_m(x) = X_n^k x^\ell = ((x-t)^{p^n} + (-1)^{p+1}t^{p^n})^k x^\ell$, a polynomial in $x$ with entries in $S$. Let the $x^j$-coefficient of $A_m(x)$ be $a_{m,j}$, and set $A = (a_{m,j})$, an infinite matrix. Then we should have $\mathbf{s} = \varphi(\mathbf{f})A$, as long as this product is defined. But since $A_m(x)$ has lowest non-zero entry of degree at least $k+\ell$, we have that $A$ is column-finite, and so the infinite row vector $\varphi(\mathbf{f})A$ is indeed defined and bounded by Lemma \ref{lem: products of matrices are defined and bounded}(ii).

We now want to show that $\Theta$ is a bijection. Note that $A_m(x)$ is a monic polynomial of degree $m$, and every coefficient of $A_m(x)$ except for the $x^m$-coefficient is divisible by $p$, we know that $A$ satisfies the hypotheses of Lemma \ref{lem: infinite matrices}(i)--(iii). That is, it can be written as $A = I-U$, where $U$ is strictly lower-triangular and bounded with lower bound 1, and so Lemma \ref{lem: infinite matrices} implies that $A$ has a two-sided inverse $B$ which is also lower-triangular, bounded and column-null. 

So, given $\mathbf{s}\in S^\infty$, we can uniquely recover some $\mathbf{t} = \mathbf{s}B\in S^\infty$ such that $\mathbf{s} = \mathbf{t}A$ by Lemma \ref{lem: associativity of column-null matrices}, and of course there exists a unique $\mathbf{f}\in (S^\infty)^{p^n}$ such that $\varphi(\mathbf{f}) = \mathbf{t}$. We have shown that the map $\mathbf{f} \mapsto \mathbf{s} = \varphi(\mathbf{f})A$ is bijective, and hence so is the map $\Theta$.

Finally, since $\Theta$ is an extension of $\theta$, it follows that $\theta$ is injective, so $S^+[[X_n;\Sigma_n,\Delta_n]]=\theta(S^+[[Y;\Sigma_n,\Delta_n]])$ is a well-defined skew power series subring of $S^+[[x;\sigma,\delta]]$.\end{proof}

\begin{cor}\label{cor: powers of x-t also form a basis}
Suppose we are in either case of Theorem \ref{thm: well-defined subring}. For each $0\leq i\leq p^n-1$, let $y_i\in S[x; \sigma, \delta]$ be a monic polynomial of degree $i$. Then $S^+[[x; \sigma, \delta]]$ is freely generated as a module over $S^+[[X_n; \Sigma_n, \Delta_n]]$ with basis $\{y_i : i=0, \dots, p^n-1\}$.
\end{cor}

\begin{proof}
Take an arbitrary element of the submodule spanned by the $y_i$:
$$f = \sum_{i=0}^{p^n-1} g_i y_i\in \sum_{i=0}^{p^n - 1} S^+[[X_n]] y_i \subseteq S^+[[x]].$$
Write $y_i = a_{i,0} + a_{i,1}x + \dots + a_{i,i-1}x^{i-1} + x^i$. Then we can rewrite $f = \sum_{k=0}^{p^n-1} h_k x^k$, where the $h_k$ and the $g_i$ are related by
$$h_k = \sum_{i=k}^{p^n-1} g_i a_{i,k}.$$
That is, if $\mathbf{h} = \begin{pmatrix}
h_0 & h_1 & \dots & h_{p^n-1}
\end{pmatrix}$ and $\mathbf{g} = \begin{pmatrix}
g_0 & g_1 & \dots & g_{p^n-1}
\end{pmatrix}$, then $\mathbf{h} = \mathbf{g} A$, where $A = (a_{i,j})$ (and we have set $a_{i,i} = 1$ for all $i$, and $a_{i,j} = 0$ for all $i < j$). Then $A$ is an upper-triangular $p^n\times p^n$ matrix with entries in $S$ and $1$s on the diagonal. Hence $A$ is invertible, and so $\{y_i: i=0, \dots, p^n-1\}$ forms a basis if and only if $\{x^i: i=0, \dots, p^n-1\}$ does, which is true by Theorem \ref{thm: well-defined subring}.
\end{proof}

In particular, under the assumptions of Theorem \ref{thm: well-defined subring}(2), we may take $y_i = (x-t)^i$. Consequently, when $x-t$ is a unit in $S^+[[x; \sigma, \delta]]$, we get a crossed product (see \S \ref{subsec: defn of crossed product})
$$S^+[[x; \sigma, \delta]] = \bigoplus_{i=0}^{p^n-1} S^+[[X_n; \Sigma_n, \Delta_n]](x-t)^i = S^+[[X_n; \Sigma_n, \Delta_n]] * (\mathbb{Z}/p^n\mathbb{Z}).$$
If we are additionally working in characteristic $p$, then this decomposition holds even without the assumption that $w(t)\geq 0$, by Theorem \ref{thm: well-defined subring}(1) and the remarks of Definition \ref{defn: X_n etc in inner case}. We will deal with the case when $x-t$ is \emph{not} a unit in \S \ref{sec: crossed products}.

Recall that, in this paper, ``compatible" means ``strongly compatible".

\begin{cor}\label{cor: Noetherian skew power series}
Suppose we are in either case of Theorem \ref{thm: well-defined subring}. If $(\Sigma_n,\Delta_n)$ is compatible with $w$ for some $n\in\mathbb{N}$, then $(\Sigma_m,\Delta_m)$ is compatible for all $m\geq n$.

Suppose further that $(\Sigma_n, \Delta_n)$ is compatible with $w$, and that $S$ contains a complete, positively filtered, $(\Sigma_n,\Delta_n)$-invariant subring $A$ and a $\Sigma_n$-invariant denominator set $T\subseteq A$ such that $S = T^{-1}A$. Assume that, given any bounded countable sequence $(s_n)_{n\in\mathbb{N}}$ of elements of $S$, there exists $t\in T$ such that $ts_n\in A$ for all $n\in\mathbb{N}$. In this case, if $\gr_w(A)$ is Noetherian, then $S^+[[x;\sigma,\delta]]$ is also Noetherian.
\end{cor}

\begin{proof}
By Theorem \ref{thm: well-defined subring}, $S^+[[x;\sigma,\delta]]$ is finitely generated over $S^+[[X_n;\Sigma_n,\Delta_n]]$, so $S^+[[x;\sigma,\delta]]$ is Noetherian if and only if $S^+[[X_n;\Sigma_n,\Delta_n]]$ is Noetherian.

So, in proving both statements, we may assume for convenience and without loss of generality that $n=0$: that is, $(\sigma,\delta)$ is compatible, i.e.\ $\deg_w(\sigma-\id)\geq 1$ and $\deg_w(\delta)\geq 1$. 

It follows that, for any $m\geq 1$,
$$\deg_w(\Sigma_m-\id)=\deg_w(\sigma^{p^m}-\id) = \deg_w\left(\underset{1\leq r\leq p^m}{\sum}{\binom{p^m}{r}(\sigma-\id)^r}\right)\geq 1,$$
where the second equality follows from writing $\sigma^{p^m} = ((\sigma-\id)+\id)^{p^m}$ and expanding using the binomial theorem.

We also have $\deg_w(\delta^r) \geq 1$ for all $r\geq 0$. In situation (1), when $S$ has characteristic $p$, this implies immediately that $\deg_w(\Delta_m)\geq 1$, as $\Delta_m = \delta^{p^m}$. In situation (2), it follows from Proposition \ref{propn: delta bounded implies Delta_n bounded}(i) and its proof (noting that $A = 0$) that $\deg_w(\Delta_m)\geq 1$.

Now choose any $0 < \varepsilon < 1$. Since $1 \leq \min\{\deg_w(\sigma-\id),\deg_w(\delta)\}$, the filtration of $S$-modules $v_{\varepsilon}$ on $S^+[[x;\sigma,\delta]]$ from Definition \ref{defn: skew power series rings} is a ring filtration (as it is a ring filtration on $S[x;\sigma,\delta]$ by the same arguments as \cite[Proposition 1.17]{jones-woods-1}, and $S[x;\sigma,\delta]$ is dense in $S^+[[x;\sigma,\delta]]$ by Theorem \ref{thm: restricted skew power series rings exist}(i)).

Restricting $v_{\varepsilon}$ to the subring $A[[x;\sigma,\delta]]$, we also see using \cite[Proposition 1.17]{jones-woods-1} that the associated graded ring $\gr_{v_{\varepsilon}}(A[x;\sigma,\delta])$ is isomorphic to $(\gr_w(A))[Z]$, and this must also be the associated graded ring of $A[[x;\sigma,\delta]]$, as $A[[x;\sigma,\delta]]$ is the completion of $A[x; \sigma, \delta]$ with respect to $v_{\varepsilon}$. Since $(\gr_w(A))[Z]$ is Noetherian, it follows from \cite[Ch II, Proposition 1.2.3]{LVO} that $A[[x;\sigma,\delta]]$ must also be Noetherian.

Finally $S^+[[x;\sigma,\delta]]=T^{-1}(A[[x;\sigma,\delta]])$ by Theorem \ref{thm: restricted skew power series rings exist}(iii), thus it is Noetherian \cite[Proposition 2.1.16]{MR}.\end{proof}

\section{Skew derivations on standard filtered artinian rings}\label{sec: sigma-orbits of maximal ideals}

Throughout \S \ref{sec: sigma-orbits of maximal ideals}, we will assume that $(R, w_0)$ is a filtered ring satisfying \eqref{filt}, and we will explore in more detail the construction of the ring $(Q,u)$ referred to in \S \ref{subsec: localisation}.

\subsection{The localisation theorem}
\label{subsec: localisation process}

We begin by stating the crucial filtered localisation theorem. A version of this theorem was first given in \cite[\S 3]{ardakovInv}, and expanded upon in \cite[Section 3]{jones-abelian-by-procyclic}.

\begin{thm}\label{thm: filtered localisation}
Suppose $(R, w_0)$ satisfies \eqref{filt}. Then there exists a filtration $u: Q(R) \to \mathbb{Z}\cup\{\infty\}$ such that: 
\begin{enumerate}[label=(\roman*),noitemsep]
\item the inclusion $(R,w_0) \to (Q(R),u)$ is continuous,
\item if $w_0(r)\geq 0$ then $u(r)\geq 0$ for any $r\in R$,
\item the completion $Q$ of $(Q(R),u)$ is simple artinian,
\item the unique extension of $u$ to $Q$ is standard,
\item $\O/J(\O)$ is a central simple algebra, where $\O = \{q\in Q : u(q) \geq 0\}$.
\end{enumerate}
\end{thm}

\begin{proof}
Parts (i)--(iv) are given by \cite[Theorem 3.3]{jones-abelian-by-procyclic}. Note that it was assumed there that $w_0$ takes values in $\mathbb{N}$ rather than $\mathbb{Z}$, but this makes no difference to the proof. Part (v) follows from \cite[Theorem 3.6 and Theorem 3.13]{ardakovInv}. Details of the construction of $u$ are given by Procedure \ref{proc: build a standard filtration} below.
\end{proof}

The localisation procedure making up the proof of Theorem \ref{thm: filtered localisation} begins with a filtered ring $(R,w_0)$ satisfying \eqref{filt}, and aims to build a filtration $u: Q(R) \to \mathbb{Z}\cup\{\infty\}$ such that $(Q,u)$ is standard, where $Q$ is the completion of $Q(R)$ with respect to $u$.

However, in this paper, we also start with the additional data of a skew derivation $(\sigma, \delta)$ on $R$ which interacts nicely with $w_0$, and we would like to understand how $(\sigma, \delta)$ interacts with $u$, so we outline this localisation procedure in detail below.

\begin{proc}\label{proc: build a standard filtration}
Beginning with a filtered ring $(R,w_0)$ satisfying \eqref{filt}, we proceed in five steps, labelled \textbf{(a)--(e)} below and throughout the paper.
\end{proc}

\begin{enumerate}[label=\textbf{(\alph*)}]
\item \textbf{Homogeneous localisation.} There exists a filtration \hbox{$w: Q(R)\to\mathbb{Z}\cup\{\infty\}$}, depending on a choice of minimal prime ideal $\mathfrak{q}$ of $A\subseteq Z(\gr_{w_0}(R))$. For further details on this construction and its properties, see \cite[\S 3]{jones-abelian-by-procyclic}.
\item \textbf{Microlocalisation.} \cite[\S 3.4]{ardakovInv} Complete $Q(R)$ with respect to $w$ to get the complete filtered artinian ring $(\widehat{Q}, w)$.
\item \textbf{Passing to an adic filtration.}
Let $U = \widehat{Q}_{\geq 0} := w^{-1}([0,\infty])$, and choose a regular normal element $z\in J(U)$ satisfying \cite[Properties 4.2]{jones-woods-1}. In fact, by the proof of \cite[Lemma 3.2]{jones-abelian-by-procyclic}, we may assume that $\gr_w(z)$ is central in $\gr_w (Q(R))$. Now define the ``$z$-adic" filtration $v_{z, U}$ on $\widehat{Q}$ whose $n$th level set is $z^n U = (zU)^n$.

\item \textbf{Quotienting by a maximal ideal.} Fix a maximal ideal $M\lhd \widehat{Q}$. Define $\overline{z}$ and $V$ to be the images of $z$ and $U$ inside the quotient $Q := \widehat{Q}/M$. Write $v_{\overline{z}, V}$ for the quotient filtration on $Q$, with $n$th level set $\overline{z}^n V = (z^n U + M)/M = (\overline{z}V)^n$.

\item \textbf{Passing to a maximal order and a standard filtration.} Choose a maximal order $\O \subseteq Q$ containing $V$ and equivalent to $V$ (which must exist by \cite[\S 3.11]{ardakovInv}).  Pass from $v_{\overline{z}, V}$ to the $J(\O)$-adic filtration $u$ on $Q$, i.e.\ the filtration with $n$th level set $J(\O)^n$. (For $n \geq 1$, $J(\O)^{-n}$ is defined to be $\{x\in Q : J(\O)^n x \subseteq \O\}$: as $J(\O)$ is left and right invertible \cite[proof of Proposition 3.9]{ardakovInv}, we have $J(\O)^a J(\O)^b = J(\O)^{a+b}$ for all $a, b\in \mathbb{Z}$.) This filtration $u$ is standard by \cite[Theorem 3.3]{jones-abelian-by-procyclic}.
\end{enumerate}

Further details on this process can be found in \cite[\S 3]{ardakovInv}, \cite[Construction 2.3]{jones-primitive-ideals} or \cite[\S\S 4.1--4.3]{jones-woods-1}. We will call a filtered ring $(Q,u)$ constructed in this way a \emph{standard completion} of $Q(R)$.

Let us analyse the individual stages of this procedure:
\begin{equation}\label{eqn: maps in localisation procedure}
\xymatrix@C+5pt{
(R,w_0)\ar@{^(->}[r]_-{\text{\textbf{(a)}}}& (Q(R),w)\ar@{^(->}[r]_-{\text{\textbf{(b)}}}& (\widehat{Q},w)\ar[r]_-{\text{\textbf{(c)}}}^-{\mathrm{id}}&(\widehat{Q}, v_{z,U})\ar@{->>}[r]_-{\text{\textbf{(d)}}}&(Q, v_{\overline{z},V})\ar[r]_-{\text{\textbf{(e)}}}^-{\mathrm{id}}&(Q,u)
}
\end{equation}

\begin{propn}\label{propn: continuity of localisation procedure maps}
$ $

\begin{enumerate}[label=\textbf{(\alph*)},noitemsep]
\item \raisebox{3pt}{\xymatrix{(R,w_0)\ar@{^(->}[r]& (Q(R),w)}} is filtered of degree 0.
\item \raisebox{3pt}{\xymatrix{(Q(R),w)\ar@{^(->}[r]& (\widehat{Q},w)}} is strictly filtered.
\item Write $F\widehat{Q}$ for the sequence of level sets of $(\widehat{Q},w)$: then there exists $e\geq 1$ such that $z^n U = F_{en}\widehat{Q}$ for all $n\in\mathbb{Z}$.
\item \raisebox{3pt}{\xymatrix{(\widehat{Q}, v_{z,U})\ar@{->>}[r]&(Q, v_{\overline{z},V})}} is strictly filtered.
\item There exist $m, n, r \geq 0$ such that $\overline{z}^{m\ell} V \subseteq J(\O)^\ell$ for all $\ell\in\mathbb{N}$, $\O\subseteq \overline{z}^{-r}V$, and $J(\O)^n \subseteq \overline{z}\O$.
\end{enumerate}
\end{propn}

\begin{proof}
$ $

\begin{enumerate}[label=\textbf{(\alph*)}]
\item This is \cite[Proposition 1.14]{jones-woods-1}.
\item This is a general fact about completions: see \cite[Chapter 1, Definition 3.4.1(2) and Proposition 3.5.1(b)]{LVO}.
\item Follows from \cite[\S 3.14, Proof of Theorem C(a)]{ardakovInv}, with $e = w(z)$.
\item $v_{\overline{z},V}$ is the quotient filtration of $v_{z,U}$ \cite[Chapter I, \S 2.5]{LVO}, so is always strictly filtered \cite[Chapter I, \S 3.2, after Definition 1]{LVO}.
\item It follows from \cite[\S 3.6, Proposition 3.7(a)]{ardakovInv} that there exists some $r\geq 0$ such that $V\subseteq \O\subseteq \overline{z}^{-r}V$, and it is proved implicitly in \cite[Lemma 3.7(d) and \S 3.14, proof of Theorem C]{ardakovInv} that there exist some $m,n\geq 0$ such that $\overline{z}^m \in J(\O)$ and $J(\O)^n \subseteq \overline{z}\O$: these are the values required.\qedhere
\end{enumerate}
\end{proof}

In particular, all of the maps in \eqref{eqn: maps in localisation procedure} are continuous, but are not necessarily filtered. It is also worth recording the following elementary corollaries:

\begin{cor}\label{cor: R and Q(R) embed in Q}
The composite map $Q(R)\to Q$ is continuous and injective, and its image is dense in $Q$ (with respect to either $v_{\overline{z},V}$ or $u$).
\end{cor}

\begin{proof}
The map is injective as $Q(R)$ is a simple ring. Since the maps \textbf{(b)--(e)} are continuous (by Proposition \ref{propn: continuity of localisation procedure maps}) and have dense image (\textbf{(b)} as it is a completion, and \textbf{(c)--(e)} as they are surjective), their composites are also continuous with dense image.
\end{proof}

\begin{cor}\label{cor: degree zero pieces map to degree zero pieces}
Each of the maps \eqref{eqn: maps in localisation procedure}\textbf{(a)--(e)} sends positive subrings to positive subrings: that is, $F_0 R \subseteq F_0 Q(R) \subseteq F_0 \widehat{Q} = U \twoheadrightarrow V \subseteq \O$.\qed
\end{cor}

\begin{cor}\label{cor: Zp-algebra}
If $(R,w_0)$ is a filtered $\mathbb{Z}_p$-algebra, then $(Q,u)$ is a filtered $\mathbb{Z}_p$-algebra.
\end{cor}

\begin{proof}
As the maps $\mathbb{Z}_p \to (R,w_0)\to (Q,u)$ are continuous, it follows that $p$ is topologically nilpotent in $Q$, and by Corollary \ref{cor: degree zero pieces map to degree zero pieces}, we have $p\in \O$. So, for some $n \geq 1$, we have $p^n \in J(\O)$. But $J(\O)$ is a prime ideal of $\O$ by Theorem \ref{thm: filtered localisation}(v), and $p$ is central, so $p\in J(\O)$, and hence the map $\mathbb{Z}_p\to (Q,u)$ is filtered of non-negative degree.
\end{proof}

\subsection{Skew derivations on simple artinian rings}

We ultimately want to extend the skew derivation $(\sigma,\delta)$ on $(R,w_0)$ to $(Q,u)$, but first, we will explore some broader properties of skew derivations on general simple artinian rings with standard filtrations.

Let $F$ be a division ring with standard filtration $v$ and valuation ring $D$, so that $Q = M_n(F)$ is a simple artinian ring with standard filtration $u = M_n(v)$ and positive part $\O = M_n(D)$. (In fact, some of the results below can be proven without assuming that $(F,v)$ or $(Q,u)$ are complete, but we will not need this extra generality.) We will use basic facts about these rings from \S \ref{subsec: DVRs} without special mention.

The following is a slight generalisation of \cite[Lemma 4.5]{jones-woods-1}, which we present here as we will use the argument several times in this paper.

\begin{lem}\label{lem: delta sends O to J(O)^d}
Suppose $(\sigma, \delta)$ is a skew derivation on $Q$ such that $\sigma(\O) = \O$, and assume there is an integer $d\in\mathbb{Z}$ such that $\delta(\O) \subseteq J(\O)^d$ and $\delta(J(\O)) \subseteq J(\O)^{d+1}$. Then $\delta(J(\O)^m) \subseteq J(\O)^{m+d}$ for all $m\in\mathbb{Z}$: that is, $\delta$ is filtered, and $\deg_u(\delta) \geq d$.
\end{lem}

\begin{proof}
$J(\O)$ is the unique maximal ideal of $\O$, so we must have $\sigma(J(\O)^m) = J(\O)^m$ for all $m\in\mathbb{Z}$.

For positive $m$, the result proceeds by induction. Suppose $\delta(J(\O)^{m-1}) \subseteq J(\O)^{m+d-1}$: then since $J(\O)^m$ is generated by elements of the form $ab$, where $a\in J(\O)^{m-1}$ and $b\in J(\O)$, we need only calculate $\delta(ab) = \delta(a)b + \sigma(a)\delta(b)$ and observe that both terms on the right-hand side are contained in $J(\O)^{m+d}$.

For negative $m$, fix $a\in J(\O)^m$, let $b\in J(\O)^{-m}$ be arbitrary, and calculate $\delta(a)b = \delta(ab) - \sigma(a)\delta(b)$. Now $ab\in \O$, so $\delta(ab)\in J(\O)^d$, and we know that $\delta(b) \in J(\O)^{-m+d}$, so it follows that $\delta(a)b \in J(\O)^d$. As $b$ was arbitrary, we get $\delta(a)J(\O)^{-m} \subseteq J(\O)^d$, and so multiplying by $J(\O)^m$ gives $\delta(a) \in J(\O)^{m+d}$.
\end{proof}

Since $Q$ is simple artinian, recall from \cite[Lemma 2.2, Theorem 2.4]{cauchon-robson} that any automorphism $\sigma$ of $Q=M_n(F)$ decomposes as $\sigma=\eta\circ M_n(\tau)$ for some inner automorphism $\eta$ of $Q$ and some automorphism $\tau$ of $F$. The following results can be seen as a generalisation of Lemma \ref{lem: delta sends O to J(O)^d}.

\begin{lem}\label{lem: conditions that imply filtered wrt standard filtration}
Let $\tau$ be an automorphism of $F$ and $\eta$ an inner automorphism of $Q$, and set $\sigma = \eta\circ M_n(\tau)$.

\begin{enumerate}[label=(\roman*)]
\item The following are equivalent: (a) $\tau$ is strictly filtered with respect to $v$, (b) $\tau$ and $\tau^{-1}$ are filtered with respect to $v$, (c) there exists some $t\in\mathbb{Z}$ such that $\tau(D), \tau^{-1}(D) \subseteq J(D)^t$.
\item $\sigma$ (resp. $\sigma^{-1}$) is filtered with respect to $u$ if and only if $\tau$ (resp. $\tau^{-1}$) is filtered with respect to $v$.
\item Suppose $\sigma$ and $\sigma^{-1}$ are filtered with respect to $u$, and $\delta$ is a $\sigma$-derivation of $Q$ satisfying $\sigma\delta = \delta\sigma$. If $\delta(\O)\subseteq J(\O)^t$ for some $t\in\mathbb{Z}$, then $\delta$ is filtered.
\end{enumerate}
\end{lem}

\begin{proof}
$ $

\begin{enumerate}[label=(\roman*)]
\item The implications (a)$\implies$(b)$\implies$(c) follow directly from the definitions, so we only prove (c)$\implies$(a).

Suppose $\tau(D), \tau^{-1}(D) \subseteq J(D)^t$, and assume first that $t \geq 0$. Then in particular $\tau(D) \subseteq D$ and $\tau^{-1}(D) \subseteq D$, and so $\tau(D) = D$. This implies that $\tau(J(D)^r) = J(D)^r$ for all $r\in\mathbb{Z}$ as required. So now assume that $t < 0$, so that we can take $x\in D$ satisfying $\tau(x) \in J(D)^a \setminus J(D)^{a+1}$ for some $a < 0$. But then, for all $m\in\mathbb{N}$, we have $x^m\in D$ and (recalling from Properties \ref{props: DVRs} that the $J(D)$-adic filtration is a valuation) $\tau(x^m) \in J(D)^{am} \setminus J(D)^{am+1}$. So for sufficiently large $m$ we will have $\tau(x^m) \not\in J(D)^t$, contradicting our assumption.

\item The inner automorphisms $\eta, \eta^{-1}$ are filtered by Remark \ref{rks: filtered and strict linear maps}(iii), so $\sigma$ is filtered with respect to $u$ if and only if $M_n(\tau) = \eta^{-1}\sigma$ is filtered with respect to $u$, which is true if and only if $\tau$ is filtered with respect to $v$. An almost identical argument works for $\sigma^{-1} = M_n(\tau^{-1}) \circ \eta$.

\item Since $\sigma$ and $\sigma^{-1}$ are filtered, it follows from (ii) that both $\tau$ and $\tau^{-1}$ are filtered, and hence from (i) that $\tau$ is strictly filtered.

Suppose that (for some $a\in Q^\times$) $\eta(q) = aqa^{-1}$ for all $q\in Q$: then $\delta':=a^{-1}\delta$ is easily checked to be an $M_n(\tau)$-derivation of $Q$, and satisfies $\delta'(\O) \subseteq J(\O)^s$ for $s = t + u(a^{-1})$. This time using \cite[Theorem 2.5]{cauchon-robson}, we see that $\delta'=\varepsilon+M_n(\theta)$, where (for some $x\in Q$) $\varepsilon(q)=xq-M_n(\tau)(q)x$ for all $q\in Q$, and $\theta$ is a $\tau$-derivation of $F$. As $\varepsilon$ is inner, it is filtered by Remark \ref{rks: filtered and strict linear maps}(ii).

It follows that $M_n(\theta)(\O) = (\delta' - \varepsilon)(\O) \subseteq J(\O)^k$ for some $k$, and hence $\theta(D) \subseteq J(D)^k$. Set $K:=k-1$, so that $\theta(D)\subseteq J(D)^K$ and $\theta(J(D))\subseteq J(D)^{K+1}$: now it follows from Lemma \ref{lem: delta sends O to J(O)^d} that $\theta(J(D)^r) \subseteq J(D)^{r+K}$ for all $r\in\mathbb{Z}$, and in particular $\theta$ is filtered. Hence $\delta' = \varepsilon + M_n(\theta)$ is filtered, and so $\delta = a\delta'$ is filtered.\qedhere
\end{enumerate}
\end{proof}

\begin{cor}\label{cor: level 1 filtered implies filtered}
Suppose $(\sigma, \delta)$ is a commuting skew derivation of $Q$, and suppose there exists $t\in\mathbb{Z}$ such that $\sigma(\O), \sigma^{-1}(\O), \delta(\O) \subseteq J(\O)^t$. Then $(\sigma,\delta)$ is filtered.
\end{cor}

\begin{proof}
Again writing $\sigma = \eta\circ M_n(\tau)$ as in Lemma \ref{lem: conditions that imply filtered wrt standard filtration}, we see that $\sigma(\O) \subseteq J(\O)^t$ implies $\tau(D) \subseteq J(D)^s$ for some $s$, by the same argument as in the proof of Lemma \ref{lem: conditions that imply filtered wrt standard filtration}(ii). A similar argument will show that $\tau^{-1}(D) \subseteq J(D)^{s'}$ for some $s'$. Hence by Lemma \ref{lem: conditions that imply filtered wrt standard filtration}(i), $\tau$ and $\tau^{-1}$ are filtered, and by Lemma \ref{lem: conditions that imply filtered wrt standard filtration}(ii), $\sigma$ and $\sigma^{-1}$ are filtered.  Now Lemma \ref{lem: conditions that imply filtered wrt standard filtration}(iii) implies that $\delta$ is filtered.
\end{proof}

\subsection{Extending the skew derivation}\label{subsec: extending the skew derivation}

Suppose that $(R, w)$ satisfies \eqref{filt}, and construct $\widehat{Q}$ and $(Q, u)$ as in \eqref{eqn: maps in localisation procedure}. In this subsection we show that, under a fairly restrictive condition on the maximal ideals of $\widehat{Q}$, the commuting skew derivation $(\sigma, \delta)$ on $(R, w_0)$ induces a canonical skew derivation $(\sigma^Q, \delta^Q)$ on our simple artinian filtered ring $(Q, u)$. Then, under good starting assumptions on $(\sigma, \delta)$, we show that the conditions for $Q^+[[x; \sigma^Q, \delta^Q]]$ to exist are met.

Now we readopt the notation of Procedure \ref{proc: build a standard filtration} used in \S \ref{subsec: localisation process}, including the subring $U = w^{-1}([0,\infty])$ of $\widehat{Q}$, the regular normal element $z\in J(U)$, and the maximal ideal $M$ used to define $Q := \widehat{Q}/M$,  $V := (U+M)/M$ and $\overline{z} = z+M$. Since the filtration $u$ on $Q$ is standard, by Definition \ref{defn: standard filtrations}, we will be able to write $u = M_n(v)$, where $\O = M_n(D)\subseteq Q = M_n(F)$ for a complete discrete valuation ring $(D,v)$ and its valued division ring of fractions $(F,v)$.

We need to explore how the skew derivation $(\sigma,\delta)$ on $R$ behaves under Procedure \ref{proc: build a standard filtration}. Recall the notion of \emph{(strong) compatibility} from Definition \ref{defn: compatibility}.

\begin{lem}\label{lem: compatibility is not enough}
Suppose $(\sigma,\delta)$ is a commuting skew derivation on $R$, compatible with $w_0$.

\begin{enumerate}[label=\textbf{(\alph*)}]
\item $(\sigma, \delta)$ extends uniquely to a commuting skew derivation $(\sigma^{Q(R)},\delta^{Q(R)})$ on $Q(R)$, compatible with $w$.
\item $(\sigma^{Q(R)},\delta^{Q(R)})$ extends uniquely to a commuting skew derivation $(\sigma^{\widehat{Q}}, \delta^{\widehat{Q}})$ on $\widehat{Q}$, compatible with $w$.
\item $(\sigma^{\widehat{Q}}, \delta^{\widehat{Q}})$ is continuous with respect to $v_{z,U}$.
\end{enumerate}
Suppose now that $(\sigma^{\widehat{Q}}, \delta^{\widehat{Q}})$ also preserves the maximal ideal $M$ of $\widehat{Q}$. Write $\pi$ for the quotient map $\widehat{Q} \to Q = \widehat{Q}/M$.
\begin{enumerate}[label=\textbf{(\alph*)}]
\setcounter{enumi}{3}
\item $(\sigma^{\widehat{Q}}, \delta^{\widehat{Q}})$ induces a unique commuting skew derivation $(\sigma^Q, \delta^Q)$ on $Q$ satisfying $\sigma^{Q} \circ \pi = \pi\circ \sigma^{\widehat{Q}}$ and $\delta^{Q} \circ \pi = \pi\circ \delta^{\widehat{Q}}$, and it is continuous with respect to $v_{\overline{z},V}$.
\item $(\sigma^Q, \delta^Q)$ is continuous with respect to $u$.
\end{enumerate}
\end{lem}

\begin{proof}

\emph{Any} skew derivation $(\sigma, \delta)$ on $R$ extends uniquely to a skew derivation $(\sigma^{Q(R)},\delta^{Q(R)})$ on $Q(R)$ \cite[Lemma 1.3]{goodearl-skew-poly-and-quantized}, and the identity $\sigma\delta=\delta\sigma$ lifts to $Q(R)$. As we are assuming $(\sigma,\delta)$ is compatible with $w_0$, we know that $(\sigma^{Q(R)}, \delta^{Q(R)})$ is compatible with $w$ \cite[Theorem 1.16]{jones-woods-1}, establishing \textbf{(a)}. (Note that this requires \emph{strong} compatibility, not the weaker notion used in \cite{jones-woods-2}.) Since compatibility implies continuity, the existence of $(\sigma^{\widehat{Q}}, \delta^{\widehat{Q}})$ follows from \cite[Chapter I, Theorem 3.4.5]{LVO}, and compatibility follows from the strictness of the filtered inclusion map $(Q(R),w)\hookrightarrow (\widehat{Q},w)$, which proves \textbf{(b)}. Parts \textbf{(c)--(e)} are now easy to check using the corresponding parts of Proposition \ref{propn: continuity of localisation procedure maps}.
\end{proof}

\begin{rk}\label{rk: restriction of extension of skew derivation}
The uniqueness in part \textbf{(a)} of Lemma \ref{lem: compatibility is not enough} implies that $\sigma^{Q(R)}|_R = \sigma$ and $\delta^{Q(R)}|_R = \delta$, and the uniqueness in part \textbf{(b)} implies that $\sigma^{\widehat{Q}}|_{Q(R)} = \sigma^{Q(R)}$ and $\delta^{\widehat{Q}}|_{Q(R)} = \delta^{Q(R)}$.

Moreover, recalling that $Q(R)\to Q$ is an injective map with dense image by Corollary \ref{cor: R and Q(R) embed in Q}, we can also view $Q(R)$ as embedded in $Q$, and view $Q$ as just a completion of $Q(R)$ with respect to either $v_{\overline{z},V}$ or $u$. Then Lemma \ref{lem: compatibility is not enough}\textbf{(d)} implies that $\sigma^Q|_{Q(R)} = \sigma^{Q(R)}$ and $\delta^Q|_{Q(R)} = \delta^{Q(R)}$, and by uniqueness, $(\sigma^Q, \delta^Q)$ is precisely the skew derivation we would obtain on $Q$ by extending $(\sigma^{Q(R)}, \delta^{Q(R)})$ along the completion map $Q(R) \to Q$ as in \cite[Chapter I, Theorem 3.4.5]{LVO}.

Hence we can regard $Q$ as a completion of $Q(R)$, and denote all of these skew derivations (where they exist) by $(\sigma, \delta)$, without confusion.

\textbf{For the remainder of \S \ref{sec: sigma-orbits of maximal ideals},} we will adopt this convention for ease of notation.
\end{rk}

As standard filtrations are so well-behaved, we can in fact strengthen Lemma \ref{lem: compatibility is not enough}\textbf{(e)}.

\begin{lem}\label{lem: compatible with w_0 implies filtered wrt u}
Suppose $(\sigma, \delta)$ is a commuting skew derivation on $R$ compatible with $w_0$. Assume the canonical extension of $(\sigma, \delta)$ to $\widehat{Q}$ preserves the maximal ideal $M\lhd \widehat{Q}$. Then the induced skew derivation $(\sigma,\delta)$ on $Q$ is filtered with respect to $u$.
\end{lem}

\begin{proof}
Set $\varphi$ to be $\sigma-\id$, $\sigma^{-1}-\id=-\sigma^{-1}(\sigma-\id)$ or $\delta$ (defined on the appropriate rings) throughout. Denote by $F\widehat{Q}$ the sequence of level sets of $(\widehat{Q},w)$: then $(\sigma, \delta)$ is compatible with $w$ by Lemma \ref{lem: compatibility is not enough}\textbf{(a)--(b)}, which implies that $\varphi(F_i \widehat{Q}) \subseteq F_i \widehat{Q}$ for all $i\in\mathbb{Z}$.

Using the corresponding results of Proposition \ref{propn: continuity of localisation procedure maps}:

\begin{itemize}[noitemsep]
\item[\textbf{(c)}] There exists $e\geq 1$ such that $z^i U = F_{ei}\widehat{Q}$ for all $i$, so $\varphi(z^i U) \subseteq z^i U$.
\item[\textbf{(d)}] By strictness, $\varphi(\overline{z}^iV)\subseteq\overline{z}^iV$.
\item[\textbf{(e)}] There exist $m,n,r\geq 0$ such that $\overline{z}^{m\ell}V\subseteq J(\O)^{\ell}$ for all $\ell\in\mathbb{N}$, $\O\subseteq\overline{z}^{-r}V$ and $J(\O)^{n}\subseteq\overline{z}\O$. In particular, $\O = J(\O)^0 \subseteq \overline{z}^{-r} V$, which is preserved by $\varphi$, and there exists $t\in\mathbb{Z}$ such that $\overline{z}^{-r}V \subseteq J(\O)^t$. In particular, we can choose $t$ such that $\sigma(\O)$, $\sigma^{-1}(\O)$, $\delta(\O)\subseteq J(\O)^t$.
\end{itemize}
Now Corollary \ref{cor: level 1 filtered implies filtered} implies the desired result.
\end{proof}

\subsection{Establishing quasi-compatibility}\label{subsec: establishing quasi-compatibility}

Throughout \S \ref{subsec: establishing quasi-compatibility}, we will continue to assume that $(R, w_0)$ satisfies \eqref{filt}, and construct the filtered rings of \eqref{eqn: maps in localisation procedure}.

The authors demonstrated under very similar circumstances in \cite[Theorem B(a)]{jones-woods-1} that, if $(\sigma,\delta)$ is compatible with $w_0$ and $R$ has characteristic $p$, the skew derivation $(\sigma^{p^m},\delta^{p^m})$ extends to $Q$ for sufficiently large $m$, and is compatible with $u$. We have a similar result in arbitrary characteristic when $\delta = \sigma - \id$ \cite[Theorem B(b)]{jones-woods-1}. 

Sadly, $(\sigma,\delta)$ itself will not necessarily remain compatible with each filtration in the process above. Indeed, we have already seen that $(\sigma, \delta)$ may no longer remain compatible with $v_{z,U}$ or the subsequent filtrations, from step \textbf{(c)} onwards. Due to Lemma \ref{lem: compatible with w_0 implies filtered wrt u}, we can ensure that $(\sigma, \delta)$ is \emph{filtered}, but this is not quite enough for our purposes: in order to use Theorem \ref{thm: restricted skew power series rings exist} to show that $Q^+[[x; \sigma, \delta]]$ exists, we also need $(\sigma, \delta)$ to be quasi-compatible with $u$ in the sense of Definition \ref{defn: strongly bounded}(iv). Thus we need more careful analysis.

In the following, write $\deg_U$ for the degree (of a filtered map) under $v_{z,U}$.

\begin{lem}\label{lem: degrees of powers of sigma-id and delta}
Suppose $\varphi$ is a filtered linear endomorphism of $(\widehat{Q},w)$ satisfying $\deg_w(\varphi) \geq 0$. Then $\varphi$ is filtered with respect to $v_{z,U}$, and $\deg_U(\varphi)\geq 0$. Also, there exists $e > 0$ such that, if $\deg_w(\varphi)\geq de$ for any $d\in\mathbb{Z}$, then $\deg_U(\varphi) \geq d$.
\end{lem}

\begin{proof}
Denoting by $F\widehat{Q}$ the sequence of level sets of $(\widehat{Q},w)$, we are assuming that $\varphi(F_i \widehat{Q}) \subseteq F_i \widehat{Q}$ for all $i\in\mathbb{Z}$. But Proposition \ref{propn: continuity of localisation procedure maps}\textbf{(c)} gives us an integer $e\geq 1$ that $\varphi(z^i U) = \varphi(F_{ei}\widehat{Q}) \subseteq F_{ei}\widehat{Q} = z^i U$. If instead we assume that $\deg_w(\varphi)\geq de$, so that $\varphi(F_i \widehat{Q}) \subseteq F_{i+de} \widehat{Q}$, then we get $$\varphi(z^i U) = \varphi(F_{ei}\widehat{Q}) \subseteq F_{e(i+d)} \widehat{Q} = z^{i+d}U,$$
and so $\deg_U(\varphi) \geq d$ as required.
\end{proof}

In particular:

\begin{cor}\label{cor: sigma-id and delta have non-negative degree}
Suppose $(\sigma, \delta)$ is a commuting skew derivation on $R$ compatible with $w_0$. Then $(\sigma, \delta)$ is filtered with respect to $v_{z,U}$, and:

\begin{enumerate}[label=(\roman*)]
\item $\deg_U(\sigma - \id) \geq 0$, $\deg_U(\sigma^{-1} - \id) \geq 0$ and $\deg_U(\delta)\geq 0$.
\item There exists $K > 0$ such that, for all $k \geq K$, $\deg_U((\sigma - \id)^k) \geq 1$ and $\deg_U(\delta^k)\geq 1$.
\item $\sigma(z^i U) = z^i U$ for all $i\in\mathbb{Z}$.
\item If $(R, w_0)$ is a filtered $\mathbb{Z}_p$-algebra, then for all $d > 0$ there exists $K$ such that, for all $k \geq K$, $\deg_U(\sigma^{p^k} - \id) \geq d$.
\end{enumerate}
\end{cor}

\begin{proof}
$(\sigma, \delta)$ is compatible with $w$ by Lemma \ref{lem: compatibility is not enough}\textbf{(a)--(b)}, and parts (i) and (ii) follow immediately from Lemma \ref{lem: degrees of powers of sigma-id and delta} applied to $\varphi = \sigma - \id$, $\sigma^{-1} - \id$ and $\delta$. As (i) implies that $\sigma(z^i U) \subseteq z^i U$ and $\sigma^{-1}(z^i U) \subseteq z^i U$, part (iii) follows.

To show part (iv): note that $(\widehat{Q},w)$ is a filtered $\mathbb{Z}_p$-algebra by Proposition \ref{propn: continuity of localisation procedure maps}\textbf{(a)--(b)}, and that $(\sigma, \delta)$ is compatible with $w$ on $\widehat{Q}$ by Lemma \ref{lem: compatibility is not enough}\textbf{(a)--(b)}. Hence $\deg_w(\sigma-\id) \geq 1$, and so $\deg_w((\sigma-\id)^k)$ can be made arbitrarily large by choosing sufficiently large $k$. Applying Proposition \ref{propn: delta bounded implies Delta_n bounded}(i) to $(\widehat{Q}, w)$ (in the special case where $t = -1$, so that $\delta = \sigma - \id$ and $\Delta_k = \sigma^{p^k} - \id$) now shows that $\deg_w(\sigma^{p^k} - \id)$ can be made arbitrarily large, and hence so can $\deg_U(\sigma^{p^k} - \id)$, by Lemma \ref{lem: degrees of powers of sigma-id and delta}.
\end{proof}

Let $(\sigma, \delta)$ be a commuting skew derivation on $R$ compatible with $w_0$, and extend it to $Q(R)$ and $\widehat{Q}$ as in \S \ref{subsec: extending the skew derivation}. As stated in the introduction and the statement of Theorem \ref{thm: well-defined subring}, the situations we are interested in are

\begin{enumerate}[label=(\arabic*)]
\item when $R$ has characteristic $p$,
\item when there exists $t\in \widehat{Q}$ such that $\delta(q)=tq-\sigma(q)t$ for all $q\in \widehat{Q}$, satisfying $w(t) \geq 0$ and $\sigma(t) = t$.
\end{enumerate}

\begin{rk}
Note that, since $(\sigma, \delta)$ is compatible with $w_0$, it follows from Lemma \ref{lem: compatibility is not enough}\textbf{(a)--(b)} that $\sigma$ is strictly filtered with respect to $w$. We will use this fact implicitly when applying results from \S\S \ref{subsec: skew power series subrings}--\ref{subsec: continuous homomorphisms} from now on.
\end{rk}

In both of these situations, we can define the family of skew derivations $(\Sigma_n,\Delta_n)$ on $\widehat{Q}$, as in Definitions \ref{defn: X_n etc in char p} and \ref{defn: X_n etc in inner case}.

\begin{lem}\label{lem: degrees of Delta_n w.r.t. U}
Suppose we are in situation (2), and $(R, w_0)$ is a filtered $\mathbb{Z}_p$-algebra. Then for all $d > 0$ there exists $K$ such that, for all $k\geq K$, $\deg_U(\Delta_k) \geq d$.
\end{lem}

\begin{proof}
Similarly to the proof of Corollary \ref{cor: sigma-id and delta have non-negative degree}(iv): we can make $\deg_w(\delta^k)$ arbitrarily large by choosing sufficiently large $k$, and so we can apply Proposition \ref{propn: delta bounded implies Delta_n bounded}(i) to $(\widehat{Q}, w)$ to see that $\deg_w(\Delta_k)$ can be made arbitrarily large, and then invoke Lemma \ref{lem: degrees of powers of sigma-id and delta}.
\end{proof}

\begin{propn}\label{propn: strong compatibility}
Suppose $(R, w_0)$ is a filtered $\mathbb{Z}_p$-algebra. If $(\sigma,\delta)$ is compatible with $w_0$, and we are in situations (1) or (2) above, then for sufficiently large $n$, $(\Sigma_n,\Delta_n)$ preserves $M$ and descends to a skew derivation of $Q$, and the induced skew derivation (which we also denote $(\Sigma_n,\Delta_n)$) is compatible with $u$.
\end{propn}

\begin{proof}
In situation (1), $R$ has characteristic $p$, so $(\Sigma_n,\Delta_n)=(\sigma^{p^n},\delta^{p^n})$, and both parts follow from \cite[Theorem B(a)]{jones-woods-1}.

In situation (2), $\delta$ is inner, and $\Sigma_k=\sigma^{p^k}$ for all $k$. It follows from \cite[Theorem 4.10]{jones-woods-1} (and we reprove below as Proposition \ref{propn: t is a power of p}) that $\Sigma_k$ preserves $M$ for sufficiently large $k$, and for any $q\in M$, $\Delta_k(q)=t^{p^k}q-\Sigma_k(q)t^{p^k}\in M$, so $\Delta_k$ also preserves $M$ for any such $k$. Hence $(\Sigma_k, \Delta_k)$ induces a skew derivation on $Q = \widehat{Q}/M$ for sufficiently large $k$. Assume in what follows that $k$ \emph{is} sufficiently large.

Now write $\deg_U$ and $\deg_V$ for degrees of filtered endomorphisms of $(\widehat{Q},v_{z,U})$ and $(Q,v_{\overline{z},V})$ respectively. It follows from Corollary \ref{cor: sigma-id and delta have non-negative degree}(iv) that $\deg_U(\Sigma_k-\id)$ can be made arbitrarily large by choosing sufficiently large $k$, and similarly from Lemma \ref{lem: degrees of Delta_n w.r.t. U} that $\deg_U(\Delta_k)$ can be made arbitrarily large. So write $\Phi_k$ for either $\Sigma_k - \id$ or $\Delta_k$: then $\deg_V(\Phi_k)$ can also be made arbitrarily large by Proposition \ref{propn: continuity of localisation procedure maps}\textbf{(d)}. Using Proposition \ref{propn: continuity of localisation procedure maps}\textbf{(e)}, we can choose $r,m\in\mathbb{N}$ such that $\mathcal{O}\subseteq \overline{z}^{-r}V$ and $\overline{z}^{2m}V\subseteq J(\O)^2$, and we can choose $k$ sufficiently large that $\deg_V(\Phi_k)\geq r+2m$. Then
\[
\Phi_k(\O)\subseteq \Phi_k(\overline{z}^{-r}V)\subseteq\overline{z}^{2m}V\subseteq J(\O)^2.
\]
Now applying Lemma \ref{lem: delta sends O to J(O)^d} to the skew derivation $(\Sigma_k, \Sigma_k-\id)$ for this $k$ will show that $\deg_u(\Sigma_k-\id)\geq 1$, and applying it to $(\Sigma_k, \Delta_k)$ will show that $\deg_u(\Delta_k) \geq 1$.
\end{proof}

\begin{thm}\label{thm: strongly bounded}
Assume that $(R, w_0)$ is a filtered $\mathbb{Z}_p$-algebra satisfying \eqref{filt}, and construct the filtered rings of \eqref{eqn: maps in localisation procedure}. Suppose $(\sigma, \delta)$ is a commuting skew derivation on $R$, compatible with $w_0$, whose natural extension to $\widehat{Q}$ preserves the maximal ideal $M\lhd \widehat{Q}$, and we have either:
\begin{enumerate}[label=(\arabic*)]
\item $R$ has characteristic $p$, or
\item There exists $t\in \widehat{Q}$ such that $\delta(q)=tq-\sigma(q)t$ for all $q\in \widehat{Q}$, satisfying $w(t) \geq 0$ and $\sigma(t) = t$.
\end{enumerate}

Then $(\sigma,\delta)$ is quasi-compatible with $u$ on $Q$, and the bounded skew power series ring $Q^+[[x;\sigma,\delta]]$ is Noetherian.
\end{thm}

\begin{proof}
We have seen that $(\sigma,\delta)$ is filtered with respect to $u$ by Lemma \ref{lem: compatible with w_0 implies filtered wrt u}.

In both situations (1) and (2), by Proposition \ref{propn: strong compatibility}, $(\Sigma_n,\Delta_n)$ is compatible with $u$ for sufficiently large $n$, and in particular $\Sigma_n = \sigma^{p^n}$ is strictly filtered with respect to $u$. In situation (1), we know that $\deg_w(\delta^{p^n}) \geq 1$, as $\delta^{p^n} = \Delta_n$, so it follows from Corollary \ref{cor: criterion for quasi-compatibility} that $(\sigma,\delta)$ is quasi-compatible with $u$. In situation (2), this follows from an application of Theorem \ref{thm: skew power series subrings exist}.

Now take $A=\O$, a $(\Sigma_n,\Delta_n)$-invariant subring of $Q$, and set $T$ to be the set of regular elements of $\O$. Since $u$ is standard by Procedure \ref{proc: build a standard filtration}\textbf{(e)}, recall that we may write $\O = M_\ell(D)$ for some complete discrete valuation ring $(D,v)$ so that $u = M_\ell(v)$ by \S \ref{subsec: DVRs}. Then, given any countable sequence $(s_n)_{n\in\mathbb{N}}$ of elements of $S$ which is bounded below, say $u(s_n) \geq B$ for all $n\in\mathbb{N}$, there exists $t\in T$ such that $ts_n\in \O$ for all $n\in\mathbb{N}$: we may simply take $t = \pi^B I_\ell$, where $\pi$ is a uniformiser of $D$ and $I_\ell$ is the identity matrix in $\O$. We can also calculate that $\gr_u(\O)\cong (\O/J(\O))[Z; \alpha]$ \cite[proof of Proposition 3.13]{ardakovInv}: this is a Noetherian ring since $\O/J(\O)$ is Noetherian by Theorem \ref{thm: filtered localisation}(v). Hence we can apply Corollary \ref{cor: Noetherian skew power series} to see that $Q^+[[x;\sigma,\delta]]$ is also Noetherian.
\end{proof}

\subsection{Orbits of maximal ideals and maximal orders}\label{subsec: sigma-orbit of maximal ideals and maximal orders}

Let $(R, w_0)$ be a filtered ring satisfying \eqref{filt}, and construct the filtered rings $(Q(R), w)$, $(\widehat{Q},w)$ and $(\widehat{Q}, v_{z,U})$ of \eqref{eqn: maps in localisation procedure} using Procedure \ref{proc: build a standard filtration}\textbf{(a)--(c)}. Assume that $(\sigma, \delta)$ is a commuting skew derivation on $R$ which is compatible with $w_0$.

So far, we have taken $Q$ to be the simple artinian ring $\widehat{Q}/M$, where $M$ is a maximal ideal of $\widehat{Q}$. But many of our results in \S\S \ref{subsec: extending the skew derivation}--\ref{subsec: establishing quasi-compatibility} depend on the rather unlikely assumption that $(\sigma,\delta)$ preserves this maximal ideal $M$, which we cannot rely on in general. But if we only require $Q$ to be \emph{semisimple}, we can overcome this problem.

To fix some assumptions and notation, first note that since $\widehat{Q}$ is artinian, it has only finitely many maximal ideals. We will begin by considering the action of $\sigma$ on these ideals.

\begin{enumerate}[label=(\roman*)]
\item Fix a maximal ideal $M = M_1\lhd \widehat{Q}$, with $\sigma$-orbit of size $t$, say $\{M_1, \dots, M_t\}$.
\item Renumber so that $\sigma(M_i) = M_{i+1}$ for all $1\leq i\leq t$ (with indices taken modulo $t$).
\item Set $N = M_1 \cap \dots \cap M_t$, and define $Q = \widehat{Q}/N$, a semisimple artinian ring. As the $M_1, \dots, M_t$ are pairwise coprime, we have the natural isomorphism
$\pi: Q \to \prod_{i=1}^t \widehat{Q}/M_i$ given by $x + N \mapsto (x + M_i)_{i=1}^t$.
\item For each $1\leq i\leq t$, set $Q_i = \pi^{-1}(\widehat{Q}/M_i)$, a simple artinian ring and an ideal in $Q$, so that $Q = Q_1 \times \dots \times Q_t$.
\item For each $1\leq i\leq t$, let $e_i = (0, \dots, 0, 1, 0, \dots, 0) \in Q_1 \times \dots \times Q_t = Q$ be the centrally primitive idempotent of $Q$ corresponding to the identity element of $Q_i$.
\end{enumerate}

We now extend the construction of \S \ref{subsec: localisation process} for the various maximal ideals. The appropriate analogue of Procedure \ref{proc: build a standard filtration}\textbf{(d)} is relatively straightforward:

\begin{enumerate}[label=\textbf{(\alph**)}]
\setcounter{enumi}{3}
\item Write $V = (U+N)/N \subseteq Q$ and $\overline{z} = z+N\in V$, with quotient filtration $v_{\overline{z},V}$ induced by $v_{z,U}$. Define $V_i = e_i V \subseteq Q_i$ and $z_i = e_i \overline{z} \in V_i$, with quotient filtration $v_{z_i, V_i}$.
\end{enumerate}

Here are some basic properties of this construction. For ease of notation, write $v_i := v_{z_i, V_i}$ for the associated $z_iV_i$-adic filtrations.

\begin{props}\label{props: zU-adic filtration, semisimple case}
In all of the following, indices are implicitly taken modulo $t$.

\begin{enumerate}[label=(\roman*)]
\item Continue to write $\sigma$ for the naturally induced automorphism of $Q$. This map induces a unique automorphism of $\prod_{i=1}^t \widehat{Q}/M_i$ (which we also denote $\sigma$), satisfying $\sigma\pi = \pi\sigma$. It permutes the factors $\widehat{Q}/M_i$ in the same way as $\sigma$ permutes the maximal ideals $M_i$, so it is given by $(x_i + M_i)_{i=1}^t \mapsto (\sigma(x_{i-1}) + M_i)_{i=1}^t$.
\item Since  $\sigma\pi^{-1} = \pi^{-1}\sigma$, we have $\sigma(Q_i) = Q_{i+1}$ and $\sigma(e_i) = e_{i+1}$ for all $1\leq i\leq t$. 
\item By (i), and since $\sigma(z^n U) = z^n U$ by Corollary \ref{cor: sigma-id and delta have non-negative degree}(iii), we have $\sigma(V_i) = V_{i+1}$ for all $1\leq i\leq t$.
\item $V = \prod_{i=1}^t V_i$ as rings. Also, for each $n\in \mathbb{Z}$, $\overline{z}^n V =\prod_{i=1}^t z_i^n V_i$ (as left and right modules).
\item Together, (iii) and (iv) imply that
$$\prod_{i=1}^t \sigma(z_{i-1}^n V_{i-1}) = \sigma(\overline{z}^n V) = \overline{z}^n V = \prod_{i=1}^t z_i^n V_i$$
for all $1\leq i\leq t$ and all $n\in \mathbb{Z}$. Multiplying by $e_j$, for any fixed $1\leq j\leq t$, shows that $\sigma(z_{j-1} V_{j-1})^n = (z_j V_j)^n$ for all $n\in\mathbb{Z}$: that is, the $\sigma(z_{j-1}V_{j-1})$-adic filtration and the $z_j V_j$-adic filtration are equal. In other words, $v_j \circ \sigma = v_{j-1}$.
\end{enumerate}
\end{props}

In the proof of the following proposition, we need the assumption that $(R,w_0)$ is a filtered $\mathbb{Z}_p$-algebra.

\begin{propn}\label{propn: t is a power of p}
$t = p^\ell$ for some integer $\ell \geq 0$.
\end{propn}

\begin{proof}
Write $t = p^\ell m$ for some $\ell \geq 0$ and $m\geq 1$ with $(m,p) = 1$.  Then, by Corollary \ref{cor: sigma-id and delta have non-negative degree}(iv) there exists $n \geq \ell$ such that $\deg_U(\sigma^{p^n} - \id) \geq 1$, and by the strictness of the filtered map $(\widehat{Q}, v_{z,U}) \to (Q, v_{\overline{z}, V})$, we also have $\deg_{V}(\sigma^{p^n} - \id) \geq 1$. In particular, 
$$v_{\overline{z},V}(\sigma^{p^n}(e_1) - e_1) = v_{\overline{z},V}(e_{p^n + 1} - e_1) \geq 1.$$
Now, $e_1 \in V_1 \setminus z_1 V_1$, and $e_{p^n + 1} \in V_{p^n+1} \setminus z_{p^n+1} V_{p^n+1}$. So if $p^n \not\equiv 0 \bmod t$, then $e_{p^n + 1} - e_1\in V\setminus \overline{z}V$ by Property \ref{props: zU-adic filtration, semisimple case}(iv), contradicting the above paragraph. So we must have $p^n \equiv 0 \bmod t$. This means that $p^\ell m \mid p^n$, and so $m \mid p^{n-\ell}$. But $(m,p) = 1$ by assumption, so $m = 1$.
\end{proof}

Of course, if $\ell=0$, then $t=1$ and $N=M_1=M$ is a $\sigma$-invariant maximal ideal, and our step \textbf{(d*)} coincides with Procedure \ref{proc: build a standard filtration}\textbf{(d)}. We have now reproved \cite[Theorem 4.10]{jones-woods-1}.

\textbf{For the remainder of \S \ref{subsec: sigma-orbit of maximal ideals and maximal orders},} we will assume that $(R,w_0)$ is a filtered $\mathbb{Z}_p$-algebra, and write $p^\ell$ for the size of the $\sigma$-orbit of $M_1$.

\begin{cor}\label{cor: minimal invariance}
$M_1,\dots,M_{p^\ell}$ are $\sigma^{p^\ell}$-invariant maximal ideals of $\widehat{Q}$.\qed
\end{cor}

Now let us consider the action of the $\sigma$-derivation $\delta$.

\begin{lem}\label{lem: when delta preserves N}
$ $

\begin{enumerate}[label=(\roman*)]
\item If $P$ is the prime radical of $\widehat{Q}$ and $\delta(P) \subseteq P$, then $\delta(N)\subseteq N$.
\item If $\delta$ is an inner $\sigma$-derivation, then $\delta(N)\subseteq N$.
\item If $\sigma\delta = \delta\sigma$, and $\chr(\widehat{Q}) = 0$, then $\delta(N) \subseteq N$.
\item If $\sigma\delta = \delta\sigma$, and $\chr(\widehat{Q}) = p > 0$, then there exists some $0\leq L\leq \ell$ such that $\delta^{p^L}(N_i) \subseteq N_i$ for each $i$, where $N_i = M_i \cap M_{p^L + i} \cap \dots \cap M_{p^\ell - p^L + i}$. Hence $\delta^{p^L}(N)\subseteq N$, and $\delta^{p^{\ell}}(N)=(\delta^{p^L})^{p^{\ell-L}}(N)\subseteq N$.
\end{enumerate}
\end{lem}

\begin{proof}
$ $

\begin{enumerate}[label=(\roman*)]
\item In this case, $(\sigma, \delta)$ descends naturally to a skew derivation on $\widehat{Q}/P$, which is a semisimple artinian ring, and so we can apply \cite[Lemma 1.2]{cauchon-robson} to the simple factors of $\widehat{Q}/P$ to see that $\delta(Q_i) \subseteq Q_i + \sigma(Q_i)$ for all $1\leq i\leq p^\ell$.
\item Clear, as $\sigma$ preserves $N$.
\item If $\chr(\widehat{Q}) = 0$, then $\chr(Q(R)) = 0$, and so $\mathbb{Q}\subseteq Q(R) \subseteq \widehat{Q}$. Hence $\widehat{Q}$ is an artinian $\mathbb{Q}$-algebra, and this result now simply follows from \cite[Theorem C(a)]{jones-woods-1}.
\item The only prime ideals of $\widehat{Q}$ containing $N$ are the $M_1, \dots, M_{p^\ell}$. Using the fact that $\sigma^{p^L}$-prime ideals are precisely intersections of $\sigma^{p^L}$-orbits of prime ideals \cite[Remarks 4* and 5*]{GolMic74}, this follows from \cite[Theorem C(b)]{jones-woods-1}.\qedhere
\end{enumerate}
\end{proof}

More examples of cases where $\delta(N) \subseteq N$ are given by \cite[Proposition 3.2 and Corollary 3.3]{jones-woods-1}. The important part of Lemma \ref{lem: when delta preserves N} in our case is part (ii), since we will assume frequently in \S \ref{sec: crossed products} that $\delta$ is an inner skew derivation.

Next, recall from Property \ref{props: zU-adic filtration, semisimple case}(iii) that $\sigma^{p^\ell}(V_1) = V_1$, and from Property \ref{props: zU-adic filtration, semisimple case}(v) that $v_1 \circ \sigma^{p^\ell} = v_1$. However, if we try to pass naively from $V_1$ to a maximal order $\O_1\subseteq Q_1$ as in Procedure \ref{proc: build a standard filtration}\textbf{(e)}, we may no longer have $\sigma^{p^\ell}(\O_1) = \O_1$, as shown by the following example which was kindly provided to us by Konstantin Ardakov:

\begin{ex}\label{ex: sigma does not preserve maximal orders}
Take $\widehat{Q} = M_2(\mathbb{F}_2((x)))$: this is a prime ring, so the unique maximal ideal of $\widehat{Q}$ is $M = 0$, and hence any $\sigma\in\Aut(\widehat{Q})$ will satisfy $\sigma(M) = M$, i.e.\ $p^\ell = 1$ in the above notation. The natural (entrywise $x$-adic) filtration on $\widehat{Q}$ is in fact the \emph{standard} filtration $u$ with respect to the maximal order $\O = M_2(F_2[[x]])$. 

Take $\sigma$ to be the automorphism of $\widehat{Q}$ (of order 2) given by conjugation by 
$$\begin{pmatrix} 1&1+x\\1&1 \end{pmatrix}\in \widehat{Q}.$$
Then $\sigma(\O) \neq \O$, so the standard filtrations $u$ and $u' := u \circ \sigma^{-1}$ are not equal.
\end{ex}

Returning to the general situation: we will choose a maximal order $\O_{1,1}$ equivalent to $V_1$, which by Example \ref{ex: sigma does not preserve maximal orders} we now know may not be preserved by $\sigma^{p^\ell}$. Then, for each $1\leq i\leq p^\ell$ and $j \in\mathbb{Z}$, we will set $\O_{i,j} = \sigma^{(i-1)+p^\ell(j-1)}(\O_{1,1})$, a maximal order in $Q_i$. Write $u_{i,j}$ for the associated (standard) $J(\O_{i,j})$-adic filtration, i.e.\ $u_{i,j} = u_{1,1} \circ \sigma^{-(i-1)-p^\ell(j-1)}$.

Diagrammatically:
\[
\begin{tikzcd}
&V_1\arrow[d, "\sigma"]\arrow[r, phantom, "\subseteq"]&{}&
\O_{1,1}\arrow[d, "\sigma"]&
\O_{1,2}\arrow[d, "\sigma"]&
\O_{1,3}\arrow[d, "\sigma"]\arrow[r, phantom, "\subseteq"]&
Q_1\arrow[d, "\sigma"]\\
&V_2\arrow[d, "\sigma"]\arrow[r, phantom, "\subseteq"]&{}&
\O_{2,1}\arrow[d, "\sigma"]
&\O_{2,2}\arrow[d, "\sigma"]
&\O_{2,3}\arrow[d, "\sigma"]\arrow[r, phantom, "\subseteq"]&
Q_2\arrow[d, "\sigma"]&{}\\
{}\arrow[ur, phantom, ""{coordinate, name=Z0}]&\vdots\arrow[d, "\sigma"]&\vdots\arrow[d, "\sigma"]\arrow[ur, phantom, ""{coordinate, name=Z1}]&
\vdots\arrow[d, "\sigma"]\arrow[ur, phantom, ""{coordinate, name=Z2}]&
\vdots\arrow[d, "\sigma"]\arrow[ur, phantom, ""{coordinate, name=Z3}]&
\vdots&
\vdots\arrow[ur, phantom, ""{coordinate, name=Z5}]\\
&V_{p^\ell}\arrow[r, phantom, "\subseteq"]\arrow[uuu, "\sigma", rounded corners=2.5ex, to path={-- ([yshift = -3ex]\tikztostart.south) 
 -| (Z0) |- ([yshift=3ex]\tikztotarget.north)[at start]\tikztonodes
-- (\tikztotarget)}]&
\O_{p^\ell,0}\arrow[uuur, "\sigma"', rounded corners=2.7ex, to path={-- ([yshift = -3ex]\tikztostart.south) 
 -| (Z1) |- ([yshift=3ex]\tikztotarget.north)[at start]\tikztonodes
-- (\tikztotarget)}]&
\O_{p^{\ell},1}\arrow[uuur, "\sigma"', rounded corners=2.7ex, to path={-- ([yshift = -3ex]\tikztostart.south) 
 -| (Z2) |- ([yshift=3ex]\tikztotarget.north)[at start]\tikztonodes
-- (\tikztotarget)}]&
\O_{p^{\ell},2}\arrow[uuur, "\sigma"', rounded corners=2.7ex, to path={-- ([yshift = -3ex]\tikztostart.south) 
 -| (Z3) |- ([yshift=3ex]\tikztotarget.north)[at start]\tikztonodes
-- (\tikztotarget)}]&
\arrow[r, phantom, "\subseteq"]&
Q_{p^\ell}\arrow[uuu, "\sigma"', rounded corners=2.5ex, to path={-- ([yshift = -3ex]\tikztostart.south) 
 -| (Z5) |- ([yshift=3ex]\tikztotarget.north)[at start]\tikztonodes
-- (\tikztotarget)}]
\end{tikzcd}
\]

So, for all $1\leq i\leq p^\ell$, we have a $\sigma^{p^\ell}$-orbit of maximal orders $\dots, \O_{i,-1}, \O_{i,0}, \O_{i,1}, \O_{i,2}, \dots$ in $Q_i$. Our first task is to prove that there are actually only finitely many maximal orders here.

Without loss of generality, we will take $i = j = 1$, and (for ease of notation) we will write $\tau = \sigma^{p^\ell}$ and $\O' = \O_{1,1}$.

\begin{lem}\label{lem: s is a power of p}
There exists $s\geq 1$ such that $\tau^s(\O') = \O'$. Moreover, the minimal such $s$ is a $p$'th power.
\end{lem}

\begin{proof}
By Proposition \ref{propn: continuity of localisation procedure maps}\textbf{(e)}, there exists $r\geq 0$ such that $V_1 \subseteq \O' \subseteq z_1^{-r} V_1$. But the map $\tau^{p^n} - \id$ has $v_{z,U}$-degree at least $r$ for sufficiently large $n$ (Corollary \ref{cor: sigma-id and delta have non-negative degree}(iv)). The strictness of the filtered quotient map $(\widehat{Q}, v_{z,U}) \to (Q_1, v_1)$ (Proposition \ref{propn: continuity of localisation procedure maps}\textbf{(d)}) implies that $\tau^{p^n}-\id$ will also have $v_1$-degree at least $r$ for any such $n$, and our choice of $r$ shows that $(\tau^{p^n} - \id)(\O') \subseteq (\tau^{p^n} - \id)(z_1^{-r}V_1) \subseteq V_1 \subseteq \O'$. Hence $\tau^{p^n}(\O') \subseteq \O'$.

Similarly, since $\tau^{-p^n}-\id=-\tau^{-p^n}(\tau^{p^n}-\id)$, it follows that $\tau^{-p^n}-\id$ has $v_1$-degree at least $r$, so we also have that $\tau^{-p^n}(\O')\subseteq \O'$. Hence $\tau^{p^n}(\O')=\O'$.

Finally, as the set $S = \{s\in \mathbb{Z} : \tau^s(\O')= \O'\}$ forms a subgroup of $\mathbb{Z}$ containing $p^n\mathbb{Z}$, its positive generator must be a power of $p$.
\end{proof}

Now we can state the right analogue of Procedure \ref{proc: build a standard filtration}\textbf{(e)}.

\begin{enumerate}[label=\textbf{(\alph**)}]
\setcounter{enumi}{4}
\item Choose a maximal order $\O_1 \subseteq Q_1$ equivalent to $V_1$, and set $\O_i = \sigma^{i-1}(\O_1) \subseteq Q_i$ for all $1\leq i\leq p^\ell$. Each $\O_i$ has the standard $J(\O_i)$-adic filtration $u_i$, and $u_{i+1} \circ \sigma = u_i$ for all $1\leq i\leq p^\ell - 1$.
\end{enumerate}

Again, if $\ell=0$ then \textbf{(e*)} coincides with Procedure \ref{proc: build a standard filtration}\textbf{(e)}.

We can (and will) define the filtration $u$ on $Q$ to be the product filtration $u = \min_i\{u_i\}$, or equivalently the $J(\O)$-adic filtration on $Q$ where $\O = \O_1 \times \dots \times \O_{p^\ell}$. We call $u$ a \emph{$\sigma$-standard filtration}, and $Q$ a \emph{$\sigma$-standard filtered artinian ring}.

In summary, we now have:
\begin{equation}\label{eqn: maps in localisation procedure 2}
\xymatrix@C+5pt{
(R,w_0)\ar@{^(->}[r]_-{\text{\textbf{(a)}}}& (Q(R),w)\ar@{^(->}[r]_-{\text{\textbf{(b)}}}& (\widehat{Q},w)\ar[r]_-{\text{\textbf{(c)}}}^-{\mathrm{id}}&(\widehat{Q}, v_{z,U})\ar@{->>}[r]_-{\text{\textbf{(d*)}}}&(Q, v_{\overline{z},V})\ar[r]_-{\text{\textbf{(e*)}}}^-{\mathrm{id}}&(Q,u),
}
\end{equation}
where $Q = \widehat{Q}/N$ is a \emph{semisimple} artinian ring with $p^\ell$ simple factors $Q_1, \dots, Q_{p^\ell}$. These simple factors are permuted transitively by $\sigma$, each $Q_i$ can be given a standard filtration $u_i$, and $u = u_1 \times \dots \times u_{p^\ell}$ is the product filtration with positive part $\O$. This sequence of filtered rings satisfies the obvious analogues of Proposition \ref{propn: continuity of localisation procedure maps} and Corollaries \ref{cor: R and Q(R) embed in Q}--\ref{cor: Zp-algebra}.

\begin{rk*}
Due to Example \ref{ex: sigma does not preserve maximal orders}, we may have $\sigma(\O)\neq \O$, and so there is no hope that $(\sigma, \delta)$ will remain compatible with $u$ in general. However, Lemma \ref{lem: s is a power of p} is strong enough to show that it will remain quasi-compatible.
\end{rk*}

The following result contains Theorem \ref{letterthm: restricted skew power series ring over Q}. Recall the two situations of interest:
\begin{enumerate}[label=(\arabic*)]
\item $R$ has characteristic $p$.
\item There exists $t\in \widehat{Q}$ such that $\delta(q)=tq-\sigma(q)t$ for all $q\in\widehat{Q}$, satisfying $w(t)\geq 0$ and $\sigma(t)=t$.
\end{enumerate}

\begin{thm}\label{thm: bounded in the semisimple case}
Let $(R, w_0)$ be a filtered $\mathbb{Z}_p$-algebra satisfying \eqref{filt}. Construct the filtered rings of \eqref{eqn: maps in localisation procedure 2}. Suppose $(\sigma, \delta)$ is a commuting skew derivation on $R$ which is compatible with $w_0$, and $\delta(N)\subseteq N$. Suppose further that we are in situation (1) or (2) above. Then:
\begin{enumerate}[label=(\roman*)]
\item $(\sigma, \delta)$ extends uniquely from $R$ to a canonical skew derivation $(\sigma^Q, \delta^Q)$ on $Q = \widehat{Q}/N$ which is quasi-compatible with $u$, and the associated bounded skew power series ring $Q^+[[x;\sigma^Q,\delta^Q]]$ is Noetherian.
\end{enumerate}
Now define $R^+[[x;\sigma,\delta]]$ with respect to the filtration $w_0$ on $R$. Then:
\begin{enumerate}[label=(\roman*)]
\setcounter{enumi}{1}
\item The embeddings $R\hookrightarrow Q$ and $Q(R) \hookrightarrow Q$ extend canonically to injective ring homomorphisms $R^+[[x; \sigma, \delta]] \hookrightarrow Q^+[[x; \sigma^Q, \delta^Q]]$ and $Q(R)[x; \sigma, \delta] \hookrightarrow Q^+[[x; \sigma^Q, \delta^Q]]$ sending $x$ to $x$.
\item If $Q^+[[x; \sigma^Q, \delta^Q]]$ is prime, then $R^+[[x; \sigma, \delta]]$ is prime.
\end{enumerate}
\end{thm}

\begin{proof}
View $(\sigma, \delta)$ as a compatible skew derivation on $(\widehat{Q},w)$ using Lemma \ref{lem: compatibility is not enough}\textbf{(a)--(b)}, and recall the associated skew derivations $(\Sigma_n,\Delta_n)$ on $\widehat{Q}$ of Definitions \ref{defn: X_n etc in char p} and \ref{defn: X_n etc in inner case}. Since $(\sigma,\delta)$ is compatible with $w$, it follows from Corollary \ref{cor: Noetherian skew power series} that $(\Sigma_n,\Delta_n)$ is compatible with $w$ for every $n$.

\begin{enumerate}[label=(\roman*)]
\item Define skew derivations $(\sigma^Q, \delta^Q) = (\widetilde{\sigma}, \widetilde{\delta})$ and $(\widetilde{\Sigma}_n, \widetilde{\Delta}_n)$ on $Q$ in the same way as in Lemma \ref{lem: compatibility is not enough}\textbf{(d)}: these restrict to $(\sigma, \delta)$ and $(\Sigma_n,\Delta_n)$ on $Q(R)$ by the same argument as in Remark \ref{rk: restriction of extension of skew derivation}, and are unique for the same reason.

Suppose $Q$ has $p^\ell$ simple factors $Q_i$ which are transitively permuted by $\widetilde{\sigma}$ in the usual way. Clearly $\widetilde{\Sigma}_\ell$ preserves each $Q_i$. We will show, in both cases of interest, that $(\widetilde{\Sigma}_{\ell}, \widetilde{\Delta}_{\ell})$ is quasi-compatible with $u$, and it will follow from Theorem \ref{thm: skew power series subrings exist} that $(\widetilde{\sigma},\widetilde{\delta})$ is quasi-compatible.

If $N=M$, so that $Q$ is simple, then the desired result is precisely Theorem \ref{thm: strongly bounded}. So we will assume from now on that $N\neq M$ and hence $Q$ is \emph{semisimple} but not simple. In particular, using Lemma \ref{lem: if Q is not simple, delta is inner}, we get that $\tilde{\delta}$ is an inner $\widetilde{\sigma}$-derivation in all cases. So, even in situation (1), we may fix $\overline{t}\in Q$ such that $\widetilde{\delta}(q)=\overline{t}q-\widetilde{\sigma}(q)\overline{t}$ for all $q\in Q$. (Note that we are not necessarily assuming that $u(t)\geq 0$, but we do still have that $\widetilde{\sigma}(\overline{t})=\overline{t}$ by Lemma \ref{lem: power 1}.) Then $\widetilde{\Delta}_\ell$ is an inner $\widetilde{\Sigma}_\ell$-derivation, defined by $\widetilde{\Delta}_{\ell}(q) = \overline{t}^{p^\ell}q - \widetilde{\Sigma}_\ell(q) \overline{t}^{p^\ell}$, and so $\widetilde{\Delta}_\ell$ also preserves each $Q_i$.

Now write $M_i$ for the maximal ideal of $\widehat{Q}$ containing $N$, defined such that $$M_i/N = Q_1\times\dots\times Q_{i-1}\times Q_{i+1}\times\dots Q_{p^\ell}.$$ Since $M_i/N$ is preserved by $(\widetilde{\Sigma}_\ell, \widetilde{\Delta}_\ell)$, it follows that $M_i$ is preserved by $(\Sigma_\ell,\Delta_\ell)$. Since $(\Sigma_\ell,\Delta_\ell)$ is compatible with $w_0$ on $R$, we may induce it to a skew derivation of $Q_i = \widehat{Q}/M_i$, and clearly this coincides with the restriction of $(\widetilde{\Sigma}_\ell, \widetilde{\Delta}_\ell)$ from $Q$ to $Q_i$. It then follows from Theorem \ref{thm: strongly bounded} that $(\widetilde{\Sigma}_\ell, \widetilde{\Delta}_\ell)$ is quasi-compatible with $u_i$, and the bounded skew power series ring $Q_i^+[[X_\ell; \widetilde{\Sigma}_\ell, \widetilde{\Delta}_\ell]]$ is Noetherian.

Since $(\widetilde{\Sigma}_\ell, \widetilde{\Delta}_\ell)$ is quasi-compatible with $u_i$, it follows from Definition \ref{defn: strongly bounded}(iii) that there exist $A_i\in\mathbb{Z}$ and $N_i\in\mathbb{N}$ such that $\deg_{u_i}(\widetilde{\Sigma}_\ell^k\widetilde{\Delta}_\ell^j)\geq A_i$ for all $k\in\mathbb{Z}$, $j\in\mathbb{N}$, and $\deg_{u_i}(\widetilde{\Delta}_\ell^{N_i})>0$. In other words, for all $q_i\in Q_i$, $u_i(\widetilde{\Sigma}_\ell^k\widetilde{\Delta}_\ell^j(q_i))\geq u_i(q)+A_i$ and $u_i(\widetilde{\Delta}_\ell^{N_i}(q_i))>u_i(q_i)$.

Let $A:=\min\{A_i:i=1,\dots,p^\ell\}$ and $N:=N_1\dots N_{p^\ell}$. Given $q=(q_1,\dots,q_\ell)\in Q$, $k\in\mathbb{Z}$, $j\in\mathbb{N}$: $$u\left((\widetilde{\Sigma}_\ell^k\widetilde{\Delta}_\ell^j)(q)\right)=\min\{u_i(\widetilde{\Sigma}_\ell^k\widetilde{\Delta}_\ell^j(q_i)):i=1,\dots p^\ell\}\geq A$$
and also $u(\widetilde{\Delta}_\ell^{N}(q))=\min\{u_i(\widetilde{\Delta}_\ell^{N}(q_i)):i=1,\dots,p^\ell\}>u(q)$, thus $(\widetilde{\Sigma}_\ell,\widetilde{\Delta}_\ell)$ is quasi-compatible with $u$ as required.

It follows from Proposition \ref{propn: subrings in characteristic p} and Theorem \ref{thm: skew power series subrings exist} that $(\widetilde{\sigma},\widetilde{\delta})$ is quasi-compatible with $u$, and hence $Q^+[[x;\widetilde{\sigma},\widetilde{\delta}]]$ is well-defined. Furthermore, $Q^+[[x;\widetilde{\sigma},\widetilde{\delta}]]$ is finitely generated over $Q^+[[X_\ell;\widetilde{\Sigma}_\ell,\widetilde{\Delta}_\ell]]$ by Proposition \ref{propn: subrings in characteristic p} and Theorem \ref{thm: well-defined subring}, and it is straightforward to see that we can realise $Q^+[[X_\ell;\widetilde{\Sigma}_\ell,\widetilde{\Delta}_\ell]]$ as a product of rings $$Q^+[[X_\ell;\widetilde{\Sigma}_\ell,\widetilde{\Delta}_\ell]]=\underset{1\leq i\leq p^\ell}{\prod}{Q_i^+[[X_\ell;\widetilde{\Sigma}_\ell,\widetilde{\Delta}_\ell]]}.$$ But each $Q_i^+[[X_\ell;\widetilde{\Sigma}_\ell,\widetilde{\Delta}_\ell]]$ in this product is Noetherian, so it follows that $Q^+[[x;\widetilde{\sigma},\widetilde{\delta}]]$ is Noetherian.

From now on, write $(\sigma, \delta)$ for the skew derivation on $Q$ instead of $(\widetilde{\sigma}, \widetilde{\delta})$.

\item Clearly there is an injective ring homomorphism $Q(R)[x;\sigma,\delta]\to Q^+[[x;\sigma,\delta]]$. Since $(R,w_0)\to (Q,u)$ is continuous, and hence sends bounded sequences to bounded sequences, there is a natural ring homomorphism $R^+[[x;\sigma,\delta]]\to Q^+[[x;\sigma,\delta]]$.

\item Using the homomorphisms of (iii), we will assume without loss of generality that $R^+[[x; \sigma, \delta]]$ and $Q(R)[x; \sigma, \delta]$ are subrings of $Q^+[[x; \sigma, \delta]]$. Recall that $Q^+[[x; \sigma, \delta]]$ can be given a ring topology, induced from the filtration $u$ on $Q$, by Theorem \ref{thm: restricted skew power series rings exist}(i).

Note first that $Q(R)$ is dense in $Q$, and $Q[x;\sigma, \delta]$ is dense in $Q^+[[x; \sigma, \delta]]$ by Theorem \ref{thm: restricted skew power series rings exist}(i), so we must have that $Q(R)[x; \sigma, \delta]$ is dense in $Q^+[[x; \sigma, \delta]]$.

If $I$ is a two-sided ideal of $R^+[[x; \sigma, \delta]]$, then $Q(R)I \subseteq Q^+[[x; \sigma, \delta]]$ is easily checked to be a $(Q(R)[x; \sigma, \delta], R^+[[x; \sigma, \delta]])$-submodule. Consider the topological closure $I' = \overline{Q(R)I}$ of $Q(R)I$ in $Q^+[[x; \sigma, \delta]]$. Then for any element $r\in Q^+[[x; \sigma, \delta]]$, $r$ is the limit point in $Q^+[[x; \sigma, \delta]]$ of a sequence $(r_n)$ of elements $r_n \in Q(R)[x;\sigma,\delta]$ by density. Moreover, any element $m\in\overline{Q(R)I}$ is the limit in $Q^+[[x; \sigma, \delta]]$ of a sequence $(m_n)$ of elements $m_n \in Q(R)I$, so $$rm=\left(\underset{n\rightarrow\infty}{\lim}r_n\right)\left(\underset{n\rightarrow\infty}{\lim}m_n\right)=\underset{n\rightarrow\infty}{\lim}r_nm_n\in\overline{Q(R)I}=I'$$ since each $r_nm_n\in Q(R)I$. Therefore, $I'$ is a $(Q^+[[x; \sigma, \delta]], R^+[[x; \sigma, \delta]])$-submodule of $Q^+[[x;\sigma,\delta]]$.

Furthermore, if $s\in R$ is a regular element, then $Q(R)Is \subseteq Q(R)I$, and hence $I's \subseteq I'$, and so
$$I' \subseteq I' s^{-1} \subseteq I' s^{-2} \subseteq \dots,$$
an ascending chain of left ideals in the Noetherian ring $Q^+[[x; \sigma, \delta]]$. So we must have that $I' s^{-n} = I' s^{-(n+1)}$ for some $n$, and multiplying on the right by $s^n$ shows that $I' = I' s^{-1}$. That is, $I'$ is right $Q(R)$-invariant, so $I'Q(R)\subseteq I'$.

But every element of $I'J' = I' \overline{Q(R) J}$ is a sum of elements of the form $fg$ where $f\in I'$ and $g\in\overline{Q(R)J}$. For each such $g$, we can again write $g$ as a limit $g=\underset{n\rightarrow\infty}{\lim}g_n$ inside $Q^+[[x; \sigma, \delta]]$, where each $g_n\in Q(R)J$. It follows that $fg=\underset{n\rightarrow\infty}{\lim}fg_n$. And since $fg_n\in I'Q(R)J\subseteq I'J$ this means that $I'J'\subseteq \overline{I'J}$.

But $I'J=\overline{Q(R)I}J$, and by the same argument, this is contained in $\overline{Q(R)IJ}=0$, and hence $I'J'=0$. So, if $Q^+[[x; \sigma, \delta]]$ is prime, then $\overline{Q(R)I} = 0$ or $\overline{Q(R)J} = 0$. But this implies that $I = 0$ or $J = 0$, and so $R^+[[x; \sigma, \delta]]$ is prime.\qedhere
\end{enumerate}
\end{proof}

\section{Simple skew power series rings}\label{sec: crossed products}

Now that we have established the relationship between the skew power series rings $R^+[[x;\sigma,\delta]]$ and $Q^+[[x;\sigma,\delta]]$, we want to use this to explore the ring-theoretic structure of these algebras and their localisations.

\subsection{The crossed product decomposition}\label{subsec: crossed}

Let $(Q,u)$ be a complete, filtered $\mathbb{Z}_p$-algebra admitting a quasi-compatible commuting skew derivation $(\sigma,\delta)$, so that $Q^+[[x; \sigma, \delta]]$ exists by Theorem \ref{thm: restricted skew power series rings exist}(i).

We make a number of assumptions throughout \S \ref{subsec: crossed}, which we now state and comment on.

\textbf{Assumption 1:} $Q=Q_1\times\dots\times Q_{p^n}$ for some $n \geq 0$, where each $Q_i$ is a simple artinian ring, and after reordering the factors, we get $\sigma(Q_i)=Q_{i+1}$ (where subscripts are taken modulo $p^n$) for each $i$.

This $n$ will be fixed throughout \S \ref{subsec: crossed}.

\textbf{Assumption 2:} each of the $Q_i$ admits a complete filtration $u_i$, and $u_{i+1} \circ \sigma = u_i$ for all $1\leq i\leq p^n - 1$. Moreover, $u$ is the product filtration:
$$u(q_1, \dots, q_{p^n}) = \min\{u_1(q_1), \dots, u_{p^n}(q_{p^n})\}.$$
In other words, for all $r\in\mathbb{Z}$, we have $F_r Q = F_r Q_1 \times \dots \times F_r Q_{p^n}$. Writing $(\gr_u(Q))_r := F_r Q/F_{r+1}Q$ and $(\gr_{u_i}(Q_i))_r := F_r Q_i/F_{r+1}Q_i$ shows that $(\gr_u(Q))_r = (\gr_{u_1}(Q_1))_r \times \dots \times (\gr_{u_{p^n}}(Q_{p^n}))_r.$

\textbf{Assumption 3:} $\delta$ is an inner $\sigma$-derivation, i.e. there exists $t\in Q$ such that $\delta(q)=tq-\sigma(q)t$ for all $q\in Q$. Define $X_r, \Sigma_r, \Delta_r, T_r$ for all $r\in\mathbb{N}$ as in Definition \ref{defn: X_n etc in inner case}, and assume that $\sigma(t) = t$, so that this matches up with Definition \ref{defn: X_n etc in char p} in characteristic $p$. Also assume that there exists some $m \geq 0$ such that $(\Sigma_m, \Delta_m)$ is compatible with $u$. Finally, if $\chr(Q) = 0$, then assume $u(t) \geq 0$.

(If $n>0$, then $\delta$ is automatically inner by Lemma \ref{lem: if Q is not simple, delta is inner}. Also, since $\sigma$ permutes the $Q_i$ non-trivially, $\sigma$ must be an \emph{outer} automorphism of $Q$ in this case, and hence $\sigma(t) = t$ automatically by Lemma \ref{lem: power 1}. Finally, if $(\Sigma_m, \Delta_m)$ is compatible with $u$, then in particular $\sigma^{p^m}$ is strictly filtered with respect to $u$.)

\textbf{Assumption 4:} $Q^+[[X_r;\Sigma_r,\Delta_r]]$ is a Noetherian ring for each $r\geq 0$.

\textbf{Assumption 5:} For each $i$, the associated graded ring $\gr_{u_i}(Q_i)$ contains a homogeneous unit of degree 1, and $(\gr_{u_i}(Q_i))_0$ is a simple ring.

\begin{lem}\label{lem: in semisimple case, sigma is filtered}
Adopt Assumptions 1--2. For each $1\leq i\leq p^n$ (with subscripts taken modulo $p^n$), the map $\sigma: (Q_i, u_i) \to (Q_{i+1}, u_{i+1})$ is an isomorphism of topological rings, and $\sigma$ and $\sigma^{-1}$ are filtered.
\end{lem}

\begin{proof}
When $1 \leq i \leq p^n-1$, this is true \emph{a fortiori}, as $\sigma$ is an isomorphism of filtered rings by Assumption 2. To prove the claim for $\sigma: (Q_{p^n}, u_{p^n}) \to (Q_1, u_1)$, it suffices to show that $\sigma^{p^n} : (Q_1, u_1) \to (Q_1, u_1)$ is an isomorphism of topological rings, and that $\sigma^{\pm p^n}$ are filtered as endomorphisms of $(Q_1, u_1)$. But $\sigma^{p^n}$ is an isomorphism of rings by definition of $\sigma$, and the assumption that $(\sigma, \delta)$ is quasi-compatible with $u$ (and in particular that $\sigma$ and $\sigma^{-1}$ are filtered with respect to $u$) implies the rest.
\end{proof}

The second part of the above lemma implies that, if $(q_r)_{r\in\mathbb{N}}$ is a bounded sequence of elements of $Q_i$, then $(\sigma(q_r))_{r\in\mathbb{N}}$ is a bounded sequence of elements of $Q_{i+1}$.

Recall that, if $x-t$ is a unit in $Q^+[[x; \sigma, \delta]]$, then the module decomposition of Corollary \ref{cor: powers of x-t also form a basis} is a crossed product decomposition 
\begin{equation*}
Q^+[[x; \sigma, \delta]] = \bigoplus_{i=0}^{p^r - 1} Q^+[[X_r;\Sigma_r,\Delta_r]] (x-t)^i = Q^+[[X_r;\Sigma_r,\Delta_r]] * (\mathbb{Z}/p^r\mathbb{Z})
\end{equation*}
for all $r\in\mathbb{N}$. If $x-t$ is \emph{not} a unit, we would like to localise to obtain a similar crossed product structure. We will use the following proposition to construct the localised ring, and to deduce information from it. We begin with a well-known lemma:

\begin{lem}\label{lem: normal2}
Suppose $S$ is a Noetherian ring and $u\neq 0$ is a regular, normal element in $S$. Then the localisation $S_u$ of $S$ at $\{u^r:r\in\mathbb{N}\}$ is well-defined, and the canonical localisation map $S\to S_u$ is injective. Moreover, $S$ is prime (resp. semiprime) if and only if $S_u$ is prime (resp. semiprime), and the minimal primes of $S_u$ are of the form $M_u$ for $u$ a minimal prime of $S$.
\end{lem}

\begin{proof}
The set $T:=\{u^r:r\in\mathbb{N}\}$ is clearly an Ore set in $S$ \cite[2.1.6]{MR} as $u$ is normal, and the assassinator of $T$ is zero as $u$ is regular \cite[2.1.10]{MR}, so the localisation $S_u$ is well-defined \cite[Theorem 2.1.12]{MR} and contains $S$ \cite[2.1.3(iii)]{MR}. The second claim follows from the bijective correspondence of \cite[Proposition 2.1.16(vii)]{MR}.
\end{proof}

We have already shown that $x-t$ is a normal element as Lemma \ref{lem: when sigma fixes t}(ii), so we now show that it is a regular element under certain conditions.

\begin{propn}\label{propn: x-t is regular and normal}
Adopt Assumptions 1--3 and 5. Then $x-t$ is a regular element of $Q^+[[x;\sigma,\delta]]$.
\end{propn}

\begin{proof}

Let us assume for contradiction that $\sum_k {q_k x^k}(x-t)=0$ for some $0\neq \sum_k q_k x^k \in Q^+[[x; \sigma, \delta]]$. Since $\sigma$ is an automorphism, $x$ is a regular element, so we can assume that $q_0\neq 0$. 

By Assumption 5, $\gr_{u_i}(Q_i)$ contains a homogeneous unit $X_i$ of degree 1, this has the form $a_i+F_2Q_i$ for some regular element $a_i\in Q_i$ with $u_i(a_i)=1$. Since the $u(q_k)$ are bounded, if we let $a:=(a_1,\dots,a_{p^n})$ then after multiplying $\sum_k {q_k x^k}$ on the left by a sufficiently high power of $a$, we can assume that all its coefficients lie in $F_0Q$. So from now on, we can assume that $q_k\in F_0Q$ for all $k$, and $q_0\neq 0$.

Since $t$ and $x$ commute, the equality $\sum_k {q_k x^k}(x-t)=0$ tells us that $-q_0t+\sum_{k\geq 1}{(q_{k-1}-q_kt)x^k}=0$. From this equation, we conclude:
$$
\begin{cases}
q_0t = 0,\\
q_k=q_{k+1}t \text{ for all } k\geq 0.
\end{cases}
$$
For all $k\in\mathbb{N}$ and all $1\leq i\leq p^n$, write $q_k = (q_{k,1}, \dots, q_{k,p^n})$, where $q_{k,i} \in Q_i$. Since $q_0\neq 0$ by assumption, we can fix some $1\leq j\leq p^n$ such that $q_{0,j} \neq 0$. For this $j$, we get
$$
\begin{cases}
q_{0,j} t_j = 0,\\
q_{0,j} = q_{1,j} t_j = q_{2,j} t_j^2 = \dots
\end{cases}$$

It follows from these equations that $q_{0,j} \in \bigcap_{r\in\mathbb{N}} F_0Q_j t_j^r$, so if $u_j(t_j^k)>0$ for some $k\in\mathbb{N}$ then $q_{0,j}\in\bigcap_{r\in\mathbb{N}} F_0Q_j t_j^r\subseteq \bigcap_{\ell\in\mathbb{N}} F_\ell Q_j=\{0\}$, and hence $q_{0,j}=0$, contradicting our assumption. Therefore, $u_j(t_j^r)\leq 0$ for all $r$, and since $q_{0,j}t_j=0$ it follows that $t_j$ is not a unit in $Q_j$.

Now, choose $m\in\mathbb{N}$ such that $m\geq n$ and $(\Sigma_m, \Delta_m)$ is compatible with $u$ as in Assumption 3, and hence compatible with $u_j$, as $(\Sigma_m, \Delta_m)$ preserves $Q_j$ by Assumption 1. Take arbitrary $q\in Q_j$. Then 
$$u_j(t_j^{p^m}q - \Sigma_m(q)t_j^{p^m}) > u_j(q)\geq u_j(q)+u_j(t_j^{p^m})$$
by definition of $\Delta_m$, because $u(t_j^{p^m})\leq 0$. We also have $$u_j((\Sigma_m(q) - q)t_j^{p^m})\geq u_j(\Sigma_m(q) - q)+u_j(t_j^{p^m}) > u_j(q)+u_j(t_j^{p^m}),$$ and it follows that 
\begin{align*}
u_j(t_j^{p^m}q-qt_j^{p^m})&=u_j(t_j^{p^m}q - \Sigma_m(q)t_j^{p^m}+(\Sigma_m(q)-q)t_j^{p^m})\\&\geq \min\{u_j(t_j^{p^m}q - \Sigma_m(q)t_j^{p^m}),u_j((\Sigma_m(q)-q)t_j^{p^m})\}> u_j(q)+u_j(t_j^{p^m})
\end{align*}

In other words, this proves that $K:= \gr(t_j^{p^m})$ is central in $\gr_{u_j}(Q_j)$, and it has degree $-d:=u(t_j^{p^m})\leq 0$. Since $X_j$ is a homogeneous unit of degree 1, it follows that $X_j^dK$ lies in $(\gr_{u_j}(Q_j))_0$, and since $K$ is central, this must be a normal element. But $(\gr_{u_j}(Q_j))_0$ is a simple ring by Assumption 5, so this implies that $X_j^dK$ is a unit, and thus $K$ is a unit in $\gr_{u_j}(Q_j)$.

But $K=t_j^{p^m}+F_{-d+1}Q_j$, so $K^{-1}=s+F_{d+1}Q_j$, and hence $y:=t_j^{p^m}s-1\in F_1Q_j$. So since $Q_j$ is complete, it follows that $t_j^{p^m}$ has inverse $s(1-y+y^2-y^3+\dots)$ in $Q_j$. Therefore, $t_j$ is a unit in $Q_j$, which is the desired contradiction.\end{proof}

Fix arbitrary $m\geq 0$. Under Assumption 3, by Theorem \ref{thm: well-defined subring}(2), we have an inclusion of rings
$$Q^+[[X_m; \Sigma_m, \Delta_m]] \subseteq Q^+[[x;\sigma,\delta]].$$
If we assume that $x-t$ is regular in $Q^+[[x;\sigma,\delta]]$ and adopt Assumption 4, then applying Lemma \ref{lem: normal2} gives that the localised rings $Q^+[[x; \sigma, \delta]]_{(x-t)}$ and $Q^+[[X_m; \Sigma_m, \Delta_m]]_{(X_m-T_m)}$ exist and are ring extensions of $Q^+[[x;\sigma,\delta]]$ and $Q^+[[X_m; \Sigma_m, \Delta_m]]$ respectively. Recalling that $X_m - T_m = (x-t)^{p^m}$, it is now clear that we get a corresponding inclusion
$$Q^+[[X_m; \Sigma_m, \Delta_m]]_{(X_m-T_m)} \subseteq Q^+[[x;\sigma,\delta]]_{(x-t)}.$$

\begin{thm}\label{thm: crossed}
Adopt Assumptions 1--5, and let $m\geq 0$ be arbitrary. Then $Q^+[[x;\sigma,\delta]]_{(x-t)}$ can be written as a crossed product
$$Q^+[[x;\sigma,\delta]]_{(x-t)} = Q^+[[X_m; \Sigma_m, \Delta_m]]_{(X_m-T_m)} * (\mathbb{Z}/p^m\mathbb{Z}).$$
\end{thm}

\begin{proof}
For ease of notation, we now write $X = X_m$, $\Sigma = \Sigma_m$, $\Delta = \Delta_m$ and $T = T_m = (-1)^{p+1} t^{p^m}$.

By Corollary \ref{cor: powers of x-t also form a basis}, we have $Q^+[[x; \sigma, \delta]] = \bigoplus_{i=0}^{p^m-1} Q^+[[X; \Sigma, \Delta]] g^i$, where $g = x-t$. Since inverting $x-t$ preserves the linear independence of the set $\{g^i:0\leq i\leq p^m-1\}$, and it is easily checked to still generate $Q^+[[x; \sigma, \delta]]_{(x-t)}$ as a left $Q^+[[X; \Sigma, \Delta]]_{(X-T)}$-module, we can conclude that
$$Q^+[[x; \sigma, \delta]]_{(x-t)} = \bigoplus_{i=0}^{p^m-1} Q^+[[X; \Sigma, \Delta]]_{(X-T)} g^i.$$
Crucially, $g$ is now a \emph{unit}, so this decomposition is a crossed product using Definition \ref{defn: crossed product}. To spell this out explicitly, write $S = Q^+[[x; \sigma, \delta]]_{(x-t)}$ then the twisting map $\tau$ and the action map $\alpha$ are respectively $\tau: (\mathbb{Z}/p^m\mathbb{Z})\times (\mathbb{Z}/p^m\mathbb{Z})\to S^\times$ and $\alpha: \mathbb{Z}/p^m\mathbb{Z}\to \Aut(S)$, defined by
$$
\tau([r],[s]) = \begin{cases}
1, & r+s<p^m,\\
X - T, & r+s\geq p^m,
\end{cases}
$$
and $\alpha([r]) = \sigma^r$, where $[r] = r + p^m\mathbb{Z}$ and $0\leq r,s\leq p^m-1$.
\end{proof}

Using this result, we may now finally remove the issue where our automorphism $\sigma$ does not fix a maximal ideal in $\widehat{Q}$. Recall that $n$ is defined so that $Q$ has $p^n$ simple factors.

\begin{cor}\label{cor: if Q_i[[x]] is prime then Q[[x]] is prime}
Adopt Assumptions 1--5.
\begin{enumerate}[label=(\roman*)]
\item $(\Sigma_n,\Delta_n)$ restricts to a quasi-compatible skew derivation $(\sigma_i,\delta_i)$ of $Q_i$ for each $1\leq i\leq p^n$, and $\sigma$ induces an isomorphism of rings $Q_i^+[[y_i;\sigma_i,\delta_i]]\cong Q_{i+1}^+[[y_{i+1};\sigma_{i+1},\delta_{i+1}]]$ sending $y_i$ to $y_{i+1}$ for all $i$ (with subscripts taken modulo $p^n$).
\item There is a canonical isomorphism of rings $\displaystyle Q^+[[X_n;\Sigma_n,\Delta_n]] \cong \prod_{i=1}^{p^n} Q_i^+[[y_i;\sigma_i,\delta_i]]$, extending the isomorphism $Q\cong \prod_{i=1}^{p^n} Q_i$ and sending $X_n$ to $(y_1, \dots, y_{p^n})$.
\item If $Q_i^+[[y_i;\sigma_i,\delta_i]]$ is prime for any $1\leq i\leq p^n$, then $Q^+[[x;\sigma,\delta]]$ is prime.
\end{enumerate}

\end{cor}

\begin{proof}
Write $X = X_n$, $\Sigma = \Sigma_n$, $\Delta = \Delta_n$ and $T = T_n = (-1)^{p+1} t^{p^n}$.

\begin{enumerate}[label=(\roman*)]
\item
Since $\sigma(Q_i)=Q_{i+1}$ for all $i$, we see that $\Sigma(Q_i) = Q_{i+p^n} = Q_i$, and it then follows that $\Delta(Q_i)\subseteq t^{p^n} Q_i + Q_i t^{p^n} \subseteq Q_i$. So write $\sigma_i = \Sigma|_{Q_i}$ and $\delta_i = \Delta|_{Q_i}$: then, since $(\Sigma,\Delta)$ is quasi-compatible with $u$ by Theorem \ref{thm: skew power series subrings exist}, the restriction $(\sigma_i, \delta_i)$ is clearly quasi-compatible with $u_i$.

Next, for any $i$, we are considering the map
\begin{align*}
Q_i^+[[y_i;\sigma_i,\delta_i]]&\to Q_{i+1}^+[[y_{i+1};\sigma_{i+1},\delta_{i+1}]]\\
\sum_{r\in\mathbb{N}} q_r y_i^r&\mapsto \sum_{r\in\mathbb{N}} \sigma(q_r)y_{i+1}^r,
\end{align*}
where again all subscripts are taken modulo $p^n$. Since $\sigma:Q_i\to Q_{i+1}$ is an isomorphism of topological rings and preserves boundedness of sequences by Lemma \ref{lem: in semisimple case, sigma is filtered}, this map is a ring isomorphism for all $i$.
\item Define 
\begin{align*}
\Theta:\prod_{1\leq i\leq p^n}{Q_i^+[[y_i;\sigma_i,\delta_i]]}&\to Q^+[[X;\Sigma,\Delta]]\\
\left(\sum_{r\in\mathbb{N}} q_{1,r}y_1^r, \dots, \sum_{r\in\mathbb{N}} q_{p^n,r}y_{p^n}^r\right)&\mapsto \sum_{r\in\mathbb{N}} (q_{1,r},\dots,q_{p^n,r})X^r.
\end{align*}

It is readily checked that this is an injective ring homomorphism, and since a sequence in $Q$ is bounded with respect to $u$ if and only if its $i$th component is bounded with respect to $u_i$ for all $i$, $\Theta$ must be a bijection, as we require.

\item Using Theorem \ref{thm: crossed}, we can realise the localisation $Q^+[[x;\sigma,\delta]]_{(x-t)}$ as a crossed product 
$$Q^+[[x;\sigma,\delta]]_{(x-t)} = \bigoplus_{j=0}^{p^n-1} Q^+[[X;\Sigma,\Delta]]_{(X-T)} g^j = Q^+[[X;\Sigma,\Delta]]_{(X-T)}\ast \mathbb{Z}/p\mathbb{Z},$$
where $g = x-t$ so that $g^{p^n}=X-T$, and conjugation by $g$ acts on $Q^+[[X;\Sigma,\Delta]]_{(X-T)}$ by $\sigma$. To show that this crossed product is prime, it suffices by \cite[Corollary 14.8]{passmanICP} to choose a minimal prime $I$ of $Q^+[[X;\Sigma,\Delta]]_{(X-T)}$, with stabiliser $H \subseteq \mathbb{Z}/p\mathbb{Z}$, and show that the induced crossed product
$$\displaystyle \left(\dfrac{Q^+[[X;\Sigma,\Delta]]_{(X-T)}}{I}\right)\ast H$$
is prime. But if $Q_i^+[[y_i;\sigma_i,\delta_i]]$ is prime for some $i$, then $Q_j^+[[y_j;\sigma_j,\delta_j]]\cong Q_i^+[[y_i;\sigma_i,\delta_i]]$ is prime for every $j$, and so
$$J = \prod_{i=2}^{p^n} Q_i^+[[y_i;\sigma_i,\delta_i]] \lhd Q^+[[X;\Sigma,\Delta]]$$
is a minimal prime ideal. Note that $J\cap Q_1^+[[y_1;\sigma_1,\delta_1]] = \{0\}$ (as subsets of $Q^+[[X; \Sigma, \Delta]]$), but that $\sigma^a(J) \cap Q_1^+[[y_1;\sigma_1,\delta_1]] \neq \{0\}$ for all $1\leq a\leq p^n-1$. Now, $x-t$ is regular by assumption and normal by Lemma \ref{lem: when sigma fixes t}(ii), so $X-T = (x-t)^{p^n}$ is also regular and normal, and it follows from Lemma \ref{lem: normal2} that $I = J_{(X-T)}$ is a minimal prime ideal of $Q^+[[X;\Sigma,\Delta]]_{(X-T)}$, and its stabiliser $H$ in $\mathbb{Z}/p\mathbb{Z}$ is trivial. Hence
$$\displaystyle \left(\dfrac{Q^+[[X;\Sigma,\Delta]]_{(X-T)}}{I}\right)\ast H \cong Q_1^+[[y_1; \sigma_1, \delta_1]],$$
which is prime by assumption, as required.\qedhere
\end{enumerate}
\end{proof}

\subsection{Finding an invariant maximal ideal}\label{subsec: finding invariant max ideal}

Now we can apply the results of the previous subsection to the case of interest.

\textbf{Setup.} Let $(R,w_0)$ be a filtered $\mathbb{Z}_p$-algebra satisfying \eqref{filt}, and carrying a commuting skew derivation $(\sigma,\delta)$ which is compatible with $w_0$. Once again, we construct the filtered rings of \eqref{eqn: maps in localisation procedure 2}. Explicitly, to fix notation:

\begin{itemize}
\item Using Procedure \ref{proc: build a standard filtration}\textbf{(a)--(b)} (given an arbitrary choice of minimal prime ideal $\mathfrak{q}$ of $A\subseteq Z(\gr_{w_0}(R))$), construct the completion $\widehat{Q}$ of $Q(R)$ with respect to $w$.
\item Using Procedure \ref{proc: build a standard filtration}\textbf{(c)}, we fix $U = \widehat{Q}_{\geq 0}$ and an arbitrary regular normal element $z\in U$, and give $\widehat{Q}$ the filtration $v_{z,U}$. Fix a maximal ideal $M$ of $\widehat{Q}$.
\item By Lemma \ref{lem: compatibility is not enough}, we see that $(\sigma,\delta)$ extends uniquely to a continuous skew derivation of $(\widehat{Q}, v_{z,U})$, which we continue to denote $(\sigma, \delta)$. Then $\sigma$ permutes the maximal ideals of $\widehat{Q}$, and by Proposition \ref{propn: t is a power of p}, the $\sigma$-orbit of $M$ has size $p^n$ for some $n\in\mathbb{N}$.
\item For each $i=1,\dots,p^n$, set $M_i:=\sigma^{i-1}(M)$, and $Q_i:=\widehat{Q}/M_i$. Let $N = M_1 \cap \dots \cap M_{p^n}$, and let $Q$ be the semisimple artinian ring $\widehat{Q}/N$, so that $Q$ has a canonical decomposition $Q = Q_1 \times \dots \times Q_{p^n}$. Clearly $\sigma$ induces an automorphism of $Q$, which we also denote by $\sigma$, and $\sigma(Q_i)=Q_{i+1}$.
\item Apply \S \ref{subsec: sigma-orbit of maximal ideals and maximal orders}\textbf{(d*)--(e*)} to $Q$, and fix maximal orders $\O_i\subseteq Q_i$ with corresponding $J(\O_i)$-adic filtrations $u_i$ such that $\sigma(\O_i) = \O_{i+1}$ and $u_{i+1}\circ \sigma = u_i$ for all $1\leq i\leq p^n - 1$. Give $Q$ the product filtration $u:=\min\{u_i:i=1,\dots,p^n\}$.
\end{itemize}

\begin{lem}\label{lem: assumptions are satisfied}
Assume $\delta(N)\subseteq N$. Then $(\sigma,\delta)$ induces a skew derivation of $Q$, which we also denote by $(\sigma,\delta)$, and the data $(Q,u,\sigma,\delta)$ satisfies Assumptions 1, 2 and 5 of \S \ref{subsec: crossed}. Moreover, if Assumption 3 is satisfied, then Assumption 4 is also satisfied.
\end{lem}

\begin{proof}

We saw in the proof of Theorem \ref{thm: bounded in the semisimple case} that $(\sigma,\delta)$ induces a skew derivation of $Q$, and we have remarked already in the setup that Assumptions 1 and 2 are satisfied. 

To verify Assumption 5, recall that each $\O_i$ is isomorphic to $M_\ell(D_i)$ for some complete discrete valuation ring $D_i$, and take $\pi_i$ to be a uniformiser in $D_i$. Then $\pi_i I+F_1Q_i$ (where $I$ is the $\ell\times\ell$ identity matrix) is a homogeneous element of degree 1 in gr$_{u_i}$ $Q_i$, with inverse $\pi_i^{-1} I+F_0Q_i$. Moreover, gr$(Q_i)_0=\O_i/\pi_i\O_i\cong M_\ell(D_i/\pi_iD_i)$ is clearly a simple ring.

Now adopt Assumption 3, i.e.\ there exists $t\in Q$ satisfying $\sigma(t) = t$ and $u(t)\geq 0$ such that $\delta(q) = tq-\sigma(q)t$ for all $q\in Q$, and some $(\Sigma_m, \Delta_m)$ is compatible with $u$. Then setting $A:=\O_1\times\dots\times\O_{p^n}$ and $T\subseteq A$ as the set of all words in the $\Sigma_n$-orbit of $a$, we see that $T$ is a $\Sigma_n$-invariant denominator set, and $Q=T^{-1}A$. Moreover, $\gr_u(A)\cong \gr_{u_1}(\O_1)\times\dots\times \gr_{u_{p^n}}(\O_{p^n})$ is Noetherian by the remarks of Definition \ref{defn: standard filtrations} (or by Theorem \ref{thm: filtered localisation}(v)), and $A$, $S$ and $T$ satisfy the hypotheses of Theorem \ref{thm: restricted skew power series rings exist}(iii), so we see using Corollary \ref{cor: Noetherian skew power series} that $Q^+[[x;\sigma,\delta]]$ is a Noetherian ring. As $Q^+[[x; \sigma, \delta]]$ is finitely generated as a module over $Q^+[[X_r; \Sigma_r, \Delta_r]]$ for all $r\geq 0$, this latter ring is also Noetherian, i.e. Assumption 4 is satisfied.\end{proof}

\begin{lem}\label{lem: all assumptions}
Suppose there exists $t\in \widehat{Q}$ such that $\sigma(t)=t$ and $\delta(q)=tq-\sigma(q)t$ for all $q\in\widehat{Q}$, where $w(t)\geq 0$ if $\chr(Q)=0$. Then $\delta(N)\subseteq N$, and the data $(Q,u,\sigma,\delta)$ satisfies Assumptions 1--5 of \S \ref{subsec: crossed}.
\end{lem}

\begin{proof}

Since $\delta$ is an inner $\sigma$-derivation of $\widehat{Q}$, it follows from Lemma \ref{lem: when delta preserves N}(ii) that $\delta(N)\subseteq N$, and thus $(\sigma,\delta)$ induces a skew derivation of $Q$, and Assumptions 1,2 and 5 are satisfied by Lemma \ref{lem: assumptions are satisfied}. Moreover, also using Lemma \ref{lem: assumptions are satisfied}, if we can prove Assumption 3, then Assumption 4 will follow.

It is clear from our assumptions that $\sigma(t)=t$, and if $w(t) \geq 0$ then $u(t)\geq 0$ by Corollary \ref{cor: degree zero pieces map to degree zero pieces}, so it remains only to prove that $(\Sigma_m,\Delta_m)$ is compatible with $u$ for some $m\geq 0$.

But each $Q_i$ is the quotient of $\widehat{Q}$ by a maximal ideal $M_i$, so it follows from Proposition \ref{propn: strong compatibility} that $(\Sigma_m,\Delta_m)$ preserves $M_i$ for all sufficiently high $m$, and induces a skew derivation of $Q_i$ is compatible with $u_i$. It follows that $(\Sigma_m,\Delta_m)$ is compatible with the product filtration $u$ as required.\end{proof}

We can now deal with the case that $M$ is not $\sigma$-invariant, i.e. $n \neq 0$.

\begin{thm}\label{thm: can assume M is sigma-invariant}
Assume that there exists $t\in \widehat{Q}$ such that $\delta(q)=tq-\sigma(q)t$ for all $q\in \widehat{Q}$. If $\chr(Q) = 0$, assume additionally that $w(t) \geq 0$. Then for any $i=1,\dots,p^n$, $(\Sigma_n, \Delta_n)$ induces a quasi-compatible skew derivation on $Q_i$, and if $Q_i^+[[y; \Sigma_n|_{Q_i}, \Delta_n|_{Q_i}]]$ is prime, then $Q^+[[x; \sigma, \delta]]$ and $R^+[[x; \sigma, \delta]]$ are prime.
\end{thm}

\begin{proof}
If $n = 0$, then $Q_i^+[[y; \Sigma_n, \Delta_n]] = Q^+[[x; \sigma, \delta]]$. So assume $n \neq 0$: note that, in this case, $\sigma$ is not an inner automorphism of $Q$ (as it permutes the simple factors of $Q$ non-trivially), and $\sigma(t) = t$ by Lemma \ref{lem: power 1}. It follows from Lemma \ref{lem: all assumptions} that $(Q,u,\sigma,\delta)$ satisfies Assumptions 1--5 of \S \ref{subsec: crossed}. 

Now, $(\Sigma_n, \Delta_n)$ is quasi-compatible with $u$ on $Q$ by Theorem \ref{thm: skew power series subrings exist}, and it is clear that $(\Sigma_n, \Delta_n)$ induces a skew derivation on $Q_i$ by restriction (cf. Corollary \ref{cor: if Q_i[[x]] is prime then Q[[x]] is prime}(i)), so this restriction will also be quasi-compatible with $u_1$ on $Q_i$. Now we may apply Corollary \ref{cor: if Q_i[[x]] is prime then Q[[x]] is prime}(iii) and Theorem \ref{thm: bounded in the semisimple case}(iii) to obtain the desired result.
\end{proof}

Using this result, provided we are able to ensure that $u(t) \geq 0$ in characteristic $0$, we can safely assume that $\sigma$ fixes some maximal ideal in $\widehat{Q}$, and thus we can assume that $Q$ is a simple ring. We do not know how to remove the hypothesis that $u(t)\geq 0$ and extend this result to general inner skew derivations, but this result applies in particular in the Iwasawa case, when $\delta=\sigma-\id$.

\subsection{Skew polynomial subrings and simplicity}\label{subsec: simple induction}

As mentioned in the introduction, even if $Q$ is a simple ring, the skew polynomial ring $Q[x;\sigma,\delta]$ will typically not be simple. The details are quite different for skew power series rings $Q^+[[x; \sigma, \delta]]$, as we will demonstrate in \S \ref{subsec: simple skew} below. In this section, we will explore what happens when a sufficiently large subring $Q^+[[X_m;\Sigma_m,\Delta_m]]$ happens to be simple.

For now, we will work in generality. Throughout this subsection, we will suppose we are given a ring $S$ with an automorphism $\sigma\in\Aut(S)$ such that $\sigma^p$ is inner. We will also assume that $S$ contains no \emph{non-trivial} (by which we mean proper and non-zero) $\sigma$-invariant ideals, and suppose that we are given a crossed product $R$ of $S$ with $\mathbb{Z}/p\mathbb{Z}$: $$R=S\ast\mathbb{Z}/p\mathbb{Z}=\bigoplus_{i=0}^{p-1} Sg^i,$$
where $gs=\sigma(s)g$ for all $s\in S$, and $c:=g^p\in S$.

Let $I$ be any non-trivial ideal of $R$, i.e. $I\neq 0$ and $I\neq R$.

\begin{defn}\label{defn: length}

Take $r=s_0+s_1g+\dots+s_{p-1}g^{p-1}\in R = S\ast\mathbb{Z}/p\mathbb{Z}$, where each $s_i \in S$. The \emph{length} $l(r)$ of $r$ is the cardinality of $\{0\leq i\leq p-1 : s_i\neq 0\}$. Clearly $l(r)=0$ if and only if $r=0$.

\end{defn}

Let $m+1$ be the minimal length of a non-zero element of $I$, where $m\geq 0$. Then for some chain $0\leq i_0<i_1<\dots<i_m=:k<p$, we can find an element $f_0g^{i_0}+f_1g^{i_1}+\dots+f_{m-1}g^{i_{m-1}}+f_mg^k\in I$ with $f_i\neq 0$ for each $i$. Define:
\[
A:=\{s\in S:h_0g^{i_0}+h_1g^{i_1}+\dots+h_{m-1}g^{i_{m-1}}+sg^k\in I\text{ for some }h_j\in S\}
\]

\noindent Then $A$ is a $\sigma$-invariant two-sided ideal of $S$, so by our assumption it is either $0$ or $S$. But we know that $0\neq f_m\in A$, so $A=S$. In particular, $1\in A$, so we can find an element $$Z:=z_{i_0}g^{i_0}+z_{i_1}g^{i_1}+\dots+z_{i_{m-1}}g^{i_{m-1}}+g^k\in I,$$ where $z_{i_j}\neq 0$ for each $j<m$.

\textbf{Note:} if $m=0$ then $g^k\in I$ and so $I=R$, contradicting our assumption. Therefore $m>0$.

\begin{propn}\label{minimal}

Each $z_{i_j}$ is a $\sigma$-invariant unit in $S$, and $z_{i_j}^{-1}sz_{i_j}=\sigma^{i_j-k}(s)$ for all $s\in S$.

\end{propn}

\begin{proof}

Firstly, $gz_{i_j}=\sigma(z_{i_j})g$, so
$$gZ-Zg=(\sigma(z_{i_0})-z_{i_0})g\cdot g^{i_0}+(\sigma(z_{i_1})-z_{i_1})g\cdot g^{i_1}+\dots+(\sigma(z_{i_{m-1}})-z_{i_{m-1}})g\cdot g^{i_{m-1}}$$
Since $m>0$, we know that $0\leq i_{j}<k<p$, so $0\leq i_j+1\leq p-1$ for each $j$. It follows that $gZ-Zg\in I$ has length less than $m+1$, so it must be zero by our assumption on $m$, thus $\sigma(z_{i_j})=z_{i_j}$ for all $j$ as required.

Now, given $s\in S$, we have $g^is=\sigma^i(s)g^i$, and hence $g^k\sigma^{-k}(s)=sg^k$. So,
$$sZ-Z\sigma^{-k}(s)=(sz_{i_0}-z_{i_0}\sigma^{i_0-k}(s))g^{i_0}+\dots+(sz_{i_{m-1}}-z_{i_{m-1}}\sigma^{i_{m-1}-k}(s))g^{i_{m-1}}.$$
But $sZ-Z\sigma^{-k}(s)\in I$, so by our assumption on $m$ again, it follows that $sZ-Z\sigma^{-k}(s)=0$, and hence $sz_{i_j}=z_{i_j}\sigma^{i_j-k}(s)$ for all $j,s$.

This means that $0\neq z_{i_j}$ is a $\sigma$-invariant normal element in $S$, which implies that $z_{i_j}S$ is a $\sigma$-invariant two-sided ideal, so it is either $0$ or $S$ by assumption. But $z_{i_j}\neq 0$, so it must be a unit, and hence $z_{i_j}^{-1}sz_{i_j}=\sigma^{i_j-k}(s)$ for all $s\in S$.\end{proof}

\begin{cor}\label{cor: inner}

$\sigma$ is an inner automorphism of $S$: in fact, $\sigma(s)=usu^{-1}$ for some $\sigma$-invariant unit $u$ in $S$, commuting with each $z_{i_j}$.

\end{cor}

\begin{proof}

Let $j:=i_0-k$. Using Proposition \ref{minimal}, we see that $\sigma^j$ is an inner automorphism of $S$, defined by the $\sigma$-invariant unit $z_{i_0}$. But $\sigma^p$ is an inner automorphism of $S$, defined by the unit $c$, which is $\sigma$-invariant as $\sigma(c) = \sigma(g^p) = g\cdot g^p\cdot g^{-1} = g^p = c$. Since $\gcd(j,p) = 1$, it follows that $\sigma$ is inner and defined by a $\sigma$-invariant unit $u$. Then, for each $j<m$, $\sigma(z_{i_j})=z_{i_j}$, so $uz_{i_j}u^{-1}=z_{i_j}$ and $u$ commutes with $z_{i_j}$.
\end{proof}

It follows that $z_{i_j}^{-1}sz_{i_j}=\sigma^{i_j-k}(s)=u^{i_j-k}su^{-(i_j-k)}$ for each $0\leq j<m$ and each $s\in S$. Hence $z_{i_j}u^{i_j-k}$ is central in $S$ and $\sigma$-invariant. Define $F:=Z(S)^\sigma=\{\alpha\in Z(S):\sigma(\alpha)=\alpha\}$, and for each $j<m$, let $$\alpha_{i_j}:=z_{i_j}u^{i_j-k}=u^{i_j-k}z_{i_j}\in F,$$ so that $z_{i_j}=u^{k-i_j}\alpha_{i_j}$.

Setting $\alpha_i=0$ for all $i<k$ such that $i\not\in \{i_0, \dots, i_{m-1}\}$, we can rewrite our element $Z\in I$ as:
\begin{align*}
Z&=u^{k}\alpha_0+u^{k-1}\alpha_1g+\dots+u\alpha_{n-1}g^{k-1}+g^k\\
&=u^{k}(\alpha_0+\alpha_1(u^{-1}g)+\dots+\alpha_{k-1}(u^{-1}g)^{k-1}+(u^{-1}g)^k).
\end{align*}
The final equality here follows because $u$ commutes with each $\alpha_i$ and also with $g$. In particular, the non-zero element $u^{-k}Z\in I$ lies in Span$_F\{(u^{-1}g)^i:0\leq i\leq p-1\}$. So altogether, we have proved the following:

\begin{thm}\label{thm: intersection with skew group ring}
Suppose $S$ is a ring, $\sigma\in \Aut(S)$ such that $\sigma^p$ is inner and $S$ contains no non-trivial $\sigma$-invariant ideals. If $R:=S\ast\mathbb{Z}/p\mathbb{Z}=\bigoplus_{i=0}^{p-1} Sg^i$, where $g$ acts by $\sigma$, and $R$ contains a non-trivial ideal $I$, then:

\begin{itemize}
\item $\sigma$ is inner, i.e. there exists a $\sigma$-invariant unit $u\in S^{\times}$ such that $\sigma(s)=usu^{-1}$ for all $s\in S$.

\item $F:=Z(S)^\sigma$ is a field and $C:=$ Span$_F\{(u^{-1}g)^i:0\leq i\leq p-1\}$ is a subring of $Z(R)$.

\item $I\cap C\neq 0$ and $C$ is not a field.
\end{itemize}
\end{thm}

\begin{proof}
The first statement follows from Corollary \ref{cor: inner}, and the last statement is true because we have seen that $0\neq u^{-k}Z\in I\cap C$, so if $C$ was a field it would follow that $I$ contains a unit, and hence $I=R$, contradicting our assumption.

To prove the second statement, note that if $\alpha\in F=Z(S)^{\sigma}$, then $\alpha S$ is a $\sigma$-invariant two-sided ideal of $S$, so either $\alpha = 0$ or $\alpha S = S$ (in which case $\alpha$ is a unit), so $F$ is a field. Finally, as $u^{-1}g$ is central in $R$, we get that $F[u^{-1}g]$ is a subring of $Z(R)$. But $(u^{-1}g)^p=u^{-p}g^p \in F$, so $F[u^{-1}g]= \mathrm{Span}_F\{(u^{-1}g)^i:0\leq i\leq p-1\}=C$ as required.\end{proof}

\begin{cor}\label{cor: semisimple in char 0}
Suppose $S$ has no $p$-torsion, and $(S,R,\sigma)$ satisfy all the hypotheses of Theorem \ref{thm: intersection with skew group ring}. If $F=Z(S)^{\sigma}$ contains all $p$'th roots of unity, then $S$ is simple and $R\cong S^k$ for some $k\leq p$.
\end{cor}

\begin{proof}

First note that for any proper ideal $J$ of $S$, $J\cap\sigma(J)\cap\dots\cap\sigma^{p-1}(J)$ is $\sigma$-invariant, so it must be 0. Hence, taking $J$ to be a maximal ideal, it follows from the Chinese remainder theorem that $S$ is simple (if $J = 0$) or is the direct product of $p$ simple rings, $S \cong \prod_{i=0}^{p-1} S/\sigma^i(J)$. In either case, $S$ is $\sigma$-prime and hence semiprime, and using \cite[Theorem 4.4]{passmanICP}, we see that $R=S\ast\mathbb{Z}/p\mathbb{Z}$ is also semiprime. 

Now let $P$ be an arbitrary prime ideal of $R$. Then $P\cap S$ is a $\sigma$-prime ideal of $S$ \cite[Lemma 14.1]{passmanICP}, and in particular is $\sigma$-invariant, so must be $0$. This means that $P$ is a \emph{minimal} prime ideal of $R$ \cite[Theorem 16.2(i)]{passmanICP}, and since $P$ was arbitrary, it follows that all prime ideals of $R$ are maximal. There are $k$ of these for some number $k \leq p$ \cite[Theorem 16.2(ii)]{passmanICP}, and their intersection is a nilpotent ideal of $R$ \cite[Theorem 16.2(iii)]{passmanICP}, which (as $R$ is semiprime) must be 0. Hence, by the same argument as above, $R$ is the direct product of finitely many simple rings, say $R\cong S_1\times\dots\times S_k$.

Using Theorem \ref{thm: intersection with skew group ring}, any non-trivial ideal $I$ of $R$ has non-zero intersection with the subring $C:=$ Span$_F\{(u^{-1}g)^i:i<p\}$ of $Z(R)$, and $C$ is not a field. But $C$ can be realised as:
\[
C=F[Y]/(Y^p-u^{-p}c),
\]
so it follows from standard results in Galois theory that $u^{-p}c\in F$ has a $p$'th root in $F$, i.e. $u^{-p}c=f^p$ for some $0\neq f\in F$. But $c=g^p$, so if $u^{-p}c=f^p$ then $(f^{-1}u^{-1}g)^p=1$, i.e. $\zeta:=f^{-1}u^{-1}g$ is a $p$'th root of unity in $Z(R)$.

Write $\pi_i: R\to S_i$ for the natural projection map for each $1\leq i\leq k$. As $F$ is a field, the composite map $\pi_i: F\hookrightarrow R\twoheadrightarrow S_i$ is injective for each $i$. Hence $\pi_i(f) \in \pi_i(F)$, and since $\pi_i(\zeta)\in S_i$ is a $p$th root of unity, it must lie in $\pi_i(F)$. In particular, the images of $u$, $f$ and $\zeta$ all lie in $\pi_i(S)$, and so $\pi_i(g) = \pi_i(uf\zeta) \in \pi_i(S)$. Since $R$ is generated by $S$ and powers of $g$, it follows that $\pi_i(R) = \pi_i(S) = S_i$.

Now write $M_i = \ker \pi_i$, so that $R = S + M_i$ and hence $S_i \cong R/M_i = (S+M_i)/M_i \cong S/(M_i\cap S)$. But $M\cap S$ is a proper $\sigma$-invariant ideal of $S$, so it must be $0$, and hence $S_i \cong S$.
\end{proof}

Unfortunately, in characteristic $p$, the preceding argument does not hold, so we return to the skew power series setting.

Suppose that the data $(Q,u,\sigma,\delta)$ satisfies Assumptions 1--5 of \S \ref{subsec: crossed} with $n = 0$ (i.e. $Q$ is simple), and that $(\gr_{u}(Q))_0$ is a simple artinian ring. It is easy to check (using Lemma \ref{lem: power 2} and Definition \ref{defn: X_n etc in inner case}) that $(Q, u, \Sigma_\ell, \Delta_\ell)$ also satisfies Assumptions 1--5 for any $\ell\geq 0$, and it follows from Proposition \ref{propn: x-t is regular and normal} that $X_\ell-T_\ell$ is regular and normal, so that by Theorem \ref{thm: crossed} we have a crossed product decomposition for each $m\geq 0$:
\[
 Q^+[[X_\ell;\Sigma_\ell,\Delta_\ell]]_{(X_\ell-T_\ell)}=Q^+[[X_{\ell+m};\Sigma_{\ell+m},\Delta_{\ell+m}]]_{(X_{\ell+m}-T_{\ell+m})}\ast(\mathbb{Z}/p^m\mathbb{Z}),
\]
with crossed product basis $\{1, g, \dots, g^{p^m-1}\}$, where $g=X_\ell-T_\ell$.

We aim to deduce information inductively about the larger ring $Q^+[[X_\ell;\Sigma_\ell,\Delta_\ell]]_{(X_\ell-T_\ell)}$ from the smaller ring $Q^+[[X_{\ell+m}; \Sigma_{\ell+m}, \Delta_{\ell+m}]]_{(X_{\ell+m}-T_{\ell+m})}$, so for the purposes of this induction we will take $m = 1$, and without loss of generality we will assume that $\ell = 0$. So $(x, \sigma, \delta, t) = (X_\ell, \Sigma_\ell, \Delta_\ell, T_\ell)$, and for ease of notation we set $(X, \Sigma, \Delta, T) = (X_{\ell+1}, \Sigma_{\ell+1}, \Delta_{\ell+1}, T_{\ell+1})$. In summary, we have
\[
 Q^+[[x;\sigma,\delta]]_{(x-t)}=Q^+[[X;\Sigma,\Delta]]_{(X-T)}\ast(\mathbb{Z}/p\mathbb{Z}),
\]
with basis $\{1, g, \dots, g^{p-1}\}$, where $g = x-t$. We then have $g^p=X-T\in Q^+[[X;\Sigma,\Delta]]$ and $gf=\sigma(f)g$ for all $f\in Q^+[[X;\Sigma,\Delta]]$.

\begin{propn}\label{propn: cases}
Assume that $Q$ has characteristic $p$, and that $Q^+[[X;\Sigma,\Delta]]_{(X-T)}$ contains no non-trivial $\sigma$-invariant ideals. Then $Q^+[[x;\sigma,\delta]]_{(x-t)}$ is simple.
\end{propn}

\begin{proof}

Let $S:=Q^+[[X;\Sigma,\Delta]]_{(X-T)}$, $R:=Q^+[[x;\sigma,\delta]]_{(x-t)}=S\ast\mathbb{Z}/p\mathbb{Z}$.

Suppose for contradiction that there exists a non-trivial ideal $I$ of $R$. Then the data $(S,\sigma,R)$ satisfies all the hypotheses of Theorem \ref{thm: intersection with skew group ring}, so writing $F = Z(S)^\sigma$ and $u\in S^\times$ for the $\sigma$-invariant unit such that $\sigma(s) = usu^{-1}$ for all $s\in S$, we can conclude that the central subring $C:= \mathrm{Span}_F\{(u^{-1}g)^i:0\leq i\leq p-1\}$ of $R$ is not a field. But we can write $C$ as the twisted-group ring $C = F^t[\mathbb{Z}/p\mathbb{Z}] = \bigoplus_{i=0}^{p-1} Fh^i$, where $h=u^{-1}g$. We will prove that the set $\{h^{ip}:0\leq i\leq p-1\}$ is linearly independent over the field $F^p$, and it will follow from \cite[Lemma 4.1]{Rei76} that $C$ is a field, a contradiction.

Let $\alpha$ be an arbitrary $\sigma$-invariant element of $S$ that commutes with $T$. In particular, we may take $\alpha$ to be any element of $F=Z(S)^{\sigma}$, or $\alpha=u$ since $uTu^{-1}=\sigma(T)=T$.

We can write $\alpha=(X-T)^{-m}f$, where $f=\sum_{i\geq 0}\lambda_iX^i$ for some $\lambda_i \in Q$, and it follows that $\sigma(\lambda_i)=\lambda_i$ for each $i$. In particular, $\lambda_iX^i$ commutes with $X-T$ for each $i$. Also, since $\alpha$ and $X$ commute with $T$, it follows that $f=(X-T)^m\alpha$ commutes with $T$. So $$0=fT-Tf=\sum_{i\geq 0}\lambda_iX^iT-T\lambda_iX^i=\sum_{i\geq 0}(T\lambda_i-\lambda_iT)X^i$$ Therefore $\lambda_i$ commutes with $T$ for each $i$, so $\lambda_iX^i$ commutes with $X=X-T+T$, and thus $\lambda_i$ commutes with $X$. 

Let us assume further that $\lambda_iX^i$ commutes with $\alpha$, and hence with $f$, for each $i$. This condition is clearly satisfied if $\alpha$ is central, so this assumption is satisfied by any $\alpha\in F$. Moreover, since $u(\lambda_iX^i)u^{-1}=\sigma(\lambda_iX^i)=\lambda_iX^i$, it is also satisfied by $\alpha=u$. With this assumption, it follows that $$0=\lambda_jX^jf-f\lambda_jX^j=\sum_{i\geq 0}{(\lambda_j\lambda_i-\lambda_i\lambda_j)X^{i+j}},$$
and so $\lambda_i\lambda_j=\lambda_j\lambda_i$ for all $i,j$.

Therefore, since $S$ has characteristic $p$, we get $$\alpha^p=(X-T)^{-pm}f^p=(X-T)^{-pm}\underset{i\geq 0}{\sum}{(\lambda_iX^i)^p}=(X^p-T^p)^{-m}\underset{i\geq 0}{\sum}{\lambda_i^pX^{pi}}.$$
In particular, every element of $F^p$ lies in the subring $Q^+[[X^p;\Sigma^p,\Delta^p]]_{(X^p-T^p)}$, and moreover $u^p \in Q^+[[X^p;\Sigma^p,\Delta^p]]_{(X^p-T^p)}$. But $h^{ip}=u^{-ip}g^{ip}=u^{-ip}(X-T)^i$, and we know that $\{(X-T)^i:0\leq i\leq p-1\}$ is linearly independent over $Q^+[[X^p;\Sigma^p,\Delta^p]]$ by Corollary \ref{cor: powers of x-t also form a basis} applied to $Q^+[[X; \Sigma, \Delta]]$. So it follows that $\{h^{ip}:0\leq i\leq p-1\}$ is linearly independent over $Q^+[[X^p;\Sigma^p,\Delta^p]]_{(X^p-T^p)}$, and hence over $F^p$ as required.\end{proof}

Altogether, we have now proved the following important result.

\begin{thm}\label{thm: simple extension}

Suppose that the data $(Q,u,\sigma,\delta)$ satisfies Assumptions 1--5 of \S \ref{subsec: crossed}, and that $Q$ is simple. Suppose there exists $m\geq 0$ such that $Q^+[[X_m;\Sigma_m,\Delta_m]]_{(X_m-T_m)}$ is a simple ring, and assume $m$ is minimal with respect to this property.

\begin{enumerate}[label=(\roman*)]
\item If $Q$ has characteristic $p$ then $m=0$ and $Q^+[[x;\sigma,\delta]]_{(x-t)}$ is simple.

\item If $Q$ has characteristic 0 and $Z(Q)$ contains all $p$'th roots of unity, then there exists $r\leq p^m$ such that $$Q^+[[x;\sigma,\delta]]_{(x-t)}\cong\left(Q^+[[X_m;\Sigma_m,\Delta_m]]_{(X_m-T_m)}\right)^r.$$
\end{enumerate}
\end{thm}

\begin{proof}

For ease of notation, write $A_k := Q^+[[X_k;\Sigma_k,\Delta_k]]_{(X_k-T_k)}$ throughout. If $m = 0$, there is nothing to prove, so assume $m > 0$.

\begin{enumerate}[label=(\roman*)]
\item
By Theorem \ref{thm: crossed}, we have a crossed product decomposition $A_{m-1} = A_m * \mathbb{Z}/p\mathbb{Z} = \bigoplus_{i=0}^{p-1} A_m g^i$, where $g$ acts by $\Sigma_{m-1}$. By assumption, $A_m$ is simple, so contains no non-zero $\Sigma_{m-1}$-invariant ideals, and now it follows from Proposition \ref{propn: cases} that $A_{m-1}$ is simple, contradicting our assumption on $m$.

\item Write $S = A_m$. Let $\ell \geq 0$ be minimal such that there exists $r\leq p$ satisfying $A_\ell \cong S^r$. If $\ell = 0$, we are done, so assume $\ell > 0$.

Write $R = A_{\ell-1}$ and $S' = A_\ell$, so that $R = S'*\mathbb{Z}/p\mathbb{Z} \cong S^r*\mathbb{Z}/p\mathbb{Z} = \bigoplus_{i=0}^{p-1} S^r g^i$, where $g$ acts by $\tau := \Sigma_{\ell-1}$. For each maximal ideal $M$ of $S'$, let $O(M):=\{\tau^j(M):0\leq j\leq p-1\}$ be the complete $\tau$-orbit of $M$, and let $\mathcal{S}:=\{O(M):M$ a maximal ideal of $S'\}$. 

For each $O\in\mathcal{S}$, $N(O):=\underset{M\in O}{\cap}{M}$ is maximal as a $\tau$-invariant ideal of $S'$, and if $$\mathcal{X}:=\{N(O):O\in\mathcal{S}\}=\{N_1,\dots,N_h\},$$ then $S_j:=S/N_j$ has no non-trivial $\tau$-invariant ideals for each $j$, and $S'\cong S_1\times\dots\times S_h$, and the crossed product $R\cong S'\ast\mathbb{Z}/p\mathbb{Z}$ decomposes as $$R\cong (S_1\ast\mathbb{Z}/p\mathbb{Z})\times\dots\times(S_h\ast\mathbb{Z}/p\mathbb{Z})$$
via the map
\begin{align*}
S^r * \mathbb{Z}/p\mathbb{Z} &\to (S_1\ast\mathbb{Z}/p\mathbb{Z})\times\dots\times(S_h\ast\mathbb{Z}/p\mathbb{Z})\\
\underset{0\leq i\leq p-1}{\sum}{(s_{1,i},\dots,s_{h,i})g^i}&\mapsto\left(\underset{0\leq i\leq p-1}{\sum}s_{1,i}g_1^i,\dots,\underset{0\leq i\leq p-1}{\sum}s_{h,i}g_h^i\right),
\end{align*}
where $g_j$ acts by the restriction of $\tau$ to $S_j$. Since $\tau^p = \Sigma_\ell$ is conjugation by $X_\ell - T_\ell \in S'^\times$, it is inner as an automorphism of $S'$, and hence the action of each $g_j^p$ is inner on $S_j$.

Using Corollary \ref{cor: semisimple in char 0}, it follows that each $S_i\ast\mathbb{Z}/p\mathbb{Z}\cong S_i^{r_i}$ for some $r_i\leq p$. But we know that $S_i=S'/N_i\cong S^{q_i}$ for some $q_i\leq r$ such that $q_1+\dots+q_h=r$, thus $R\cong S^{r_1q_1+\dots+r_hq_h}$, and $r_1q_1+\dots+r_hq_h\leq p(q_1+\dots+q_h)=pr\leq p\cdot p^{m-\ell}=p^{m-(\ell-1)}$, contradicting the minimality of $\ell$.\qedhere
\end{enumerate}
\end{proof}

In particular, in the Iwasawa case when $t=-1$ and $\delta=\sigma-\id$, it follows in characteristic $p$ that if $Q^+[[x^{p^m};\sigma^{p^m},\sigma^{p^m}-\id]]$ is simple for some $m$, then $Q^+[[x;\sigma,\sigma-\id]]$ is simple.

\begin{rk}
Theorem \ref{thm: simple extension}(ii) is the strongest result we have to date in characteristic 0. If $(R,w_0)$ satisfies \eqref{filt}, $\chr(R)=0$ and we are in situation (2) of Theorem \ref{thm: strongly bounded}, then if we assume further that no power of $\sigma$ is an inner automorphism of a $\sigma$-standard completion $Q$ of $Q(R)$ (as we do in the next section), then we can apply this result to show that $R^+[[x;\sigma,\delta]]$ is semiprime.
\end{rk}

\subsection{Outer automorphisms}\label{subsec: simple skew}

In \S \ref{subsec: simple induction} we explored the implications for $Q^+[[x; \sigma, \delta]]_{(x-t)}$ of the simplicity of $Q^+[[X_m;\Sigma_m,\Delta_m]]_{(X_m - T_m)}$. We will now explore how, contrary to the skew polynomial case, it will often be the case that $Q^+[[X_m;\Sigma_m,\Delta_m]]_{(X_m - T_m)}$ is indeed simple.

Let $(R,w_0)$ be a filtered $\mathbb{Z}_p$-algebra satisfying \eqref{filt}, and carrying a commuting skew derivation $(\sigma,\delta)$ which is compatible with $w_0$, and once again construct the filtered rings of \eqref{eqn: maps in localisation procedure 2}. We will now assume further that $\chr(Q)=p$ and $\sigma$ is not an inner automorphism of $Q$.

Suppose there exists $t\in Q(R)$ such that $\delta(q) = tq - \sigma(q)t$ for all $q\in Q$. Since $Q$ is semisimple artinian, it follows from Lemma \ref{lem: power 1} that $\sigma(t)=t$. Define $X_r, \Sigma_r, \Delta_r, T_r$ for all $r\in\mathbb{N}$ as in Definition \ref{defn: X_n etc in inner case}, and using Lemma \ref{lem: all assumptions}, we see that Assumptions 1--5 of \S \ref{subsec: crossed} are satisfied by Lemma \ref{lem: assumptions are satisfied}. From now on, using Assumption 3, fix $m\geq 0$ such that $(\Sigma_m,\Delta_m)$ is compatible with $u$.

As usual, set $\O:=u^{-1}([0,\infty])$. The compatibility of $(\Sigma_m,\Delta_m)$ ensures that $$Q^+[[X_m;\Sigma_m,\Delta_m]]\cong Q\otimes_{\O} \O[[X_m;\Sigma_m,\Delta_m]]$$ by Theorem \ref{thm: restricted skew power series rings exist}(iii) (compare the proof of Lemma \ref{lem: assumptions are satisfied}).

A key result, which connects skew power series over $Q$ to skew polynomials, is \cite[Theorem C]{jones-woods-2}. This states that, if $(\Sigma_m, \Delta_m)$ is compatible with $u$, then any non-zero two-sided ideal of $Q^+[[X_m;\Sigma_m,\Delta_m]]$ has non-zero intersection with the skew polynomial ring $Q[X_m;\Sigma_m,\Delta_m]$.

For convenience of notation, set $\tau:=\Sigma_m$. Then an easy calculation shows that we can set $y = X_m - T_m$ to ``untwist" this skew polynomial ring to one of automorphic type: $$Q[X_m;\Sigma_m,\Delta_m]=Q[y;\tau].$$

The following is surely well-known.

\begin{propn}\label{propn: polynomial}
Suppose $Q$ is simple. If no positive power of $\sigma$ is inner, then any non-zero two-sided ideal of $Q[y,\tau]$ is generated by $y^n$ for some $n\in\mathbb{N}$.
\end{propn}

\begin{proof}
Let $I\neq 0$ be a two-sided ideal in $Q[y;\tau]$, and $n\geq 0$ be the minimal degree (as a polynomial in $y$) of a non-zero element of $I$. Define
\[
A:=\{q\in Q:b_0+b_1y+\dots+b_{n-1}y^{n-1}+qy^n\in I\text{ for some } b_i\in Q\}
\]
Then $A$ is a two-sided ideal of $Q$, and $A\neq 0$ by our choice of $n$, so we must have $1\in A$. Thus there exist $z_0,z_1,\dots,z_{n-1}\in Q$ such that $$r:=z_0+z_1y+\dots+z_{n-1}y^{n-1}+y^n\in I.$$
For any $q\in Q$, we have $qr-r\tau^{-n}(q)\in I$, and 
\[
qr-r\tau^{-n}(q)=(qz_0-z_0\tau^{-n}(q))+(qz_1-z_1\tau^{-(n-1)}(q))y+\dots+(qz_{n-1}-z_{n-1}\tau^{-1}(q))y^{n-1},
\] 
so by our assumption on $n$, it follows that $qz_i-z_i\tau^{-(n-i)}(q)=0$ for all $i$.

Therefore, if $z_i\neq 0$ for some $i$, then $z_i$ is a non-zero normal element of the simple ring $Q$, and hence it is a unit. So $\tau^{-(n-i)}(q)=z_i^{-1}qz_i$ for all $q\in Q$, and hence $\tau^{n-i}$ is inner -- contradiction. It follows that $z_i=0$ for all $i$, and hence $r=y^n\in I$, so $I$ contains the two-sided ideal $(y^n)$.

Finally, if $c_0+c_1y+\dots+c_my^m\in I$ for some $m\geq n$, then since $y^n, y^{n+1},\dots,y^m\in I$, it follows that $c_0+c_1y+\dots+c_{n-1}y^{n-1}\in I$, so again, by our assumption on $n$, it follows that $c_i=0$ for all $i<n$, so $I\subseteq (y^n)$, and hence $I=(y^n)$.\end{proof}

\begin{thm}\label{thm: simple when no power of sigma is inner}

Let $(Q,u)$ be a simple artinian ring of characteristic $p$ carrying a commuting skew derivation $(\sigma,\delta)$ which is quasi-compatible with $u$, and suppose that the data $(Q,u,\sigma,\delta)$ satisfies Assumptions 1--5 of \S \ref{subsec: crossed}. In particular, there exists $t\in Q$ such that $\delta(q) = tq - \sigma(q)t$ for all $q\in Q$, and $(\Sigma_m,\Delta_m)$ is compatible with $u$ for sufficiently large $m$. If no positive power of $\sigma$ is inner as an automorphism of $Q$, then the localisation $Q^+[[x;\sigma,\delta]]_{(x-t)}$ is a simple ring.
\end{thm}

\begin{proof}

Let $S := Q^+[[X_m;\Sigma_m,\Delta_m]]$ and $y := X_m-T_m$.

We show first that any non-zero ideal $I \lhd S$ contains some $y^n$. Indeed, since $(\Sigma_m,\Delta_m)$ is compatible with the standard filtration on $Q$, we get that $J:=I\cap Q[X_m;\Sigma_m,\Delta_m]\neq 0$ by \cite[Theorem C]{jones-woods-2}. Then, since $\Delta_m$ is inner, $Q[X_m;\Sigma_m,\Delta_m] = Q[y;\Sigma_m]$, so $J=(y^n)$ for some $n\in\mathbb{N}$ by Proposition \ref{propn: polynomial}, and hence $I$ contains $y^n$.

Next, let $I \neq 0$ be an arbitrary two-sided ideal in the localised ring $S_y$. We must have $I\cap S\neq 0$ \cite[Proposition 2.1.16(iii)]{MR}, and so by the above, $I\cap S$ contains some $(y^n)$. But as $y$ is a unit in $S_y$, it follows that $I = S_y$, and as $I$ was arbitrary, this means that $S_y$ is a simple ring. But now, by Theorem \ref{thm: simple extension}(i), $Q^+[[x; \sigma, \delta]]_{(x-t)}$ is simple.
\end{proof}

Recall that, as mentioned in the introduction, it will sometimes be the case that $t$ is invertible and $u(t) = u(t^{-1}) = 0$: in this case, $x-t$ is a unit, and so $Q^+[[x; \sigma, \delta]] = Q^+[[x; \sigma, \delta]]_{(x-t)}$ is simple. In either case, $Q^+[[x; \sigma, \delta]]$ is prime by Lemma \ref{lem: normal2}.

\textbf{Note:} The proof that $Q^+[[X_m;\Sigma_m,\Delta_m]]_{(X_m-T_m)}$ is a simple ring also applies when $Q$ has characteristic 0 and $u(t)\geq 0$, so applying Theorem \ref{thm: simple extension} gives $Q^+[[x; \sigma, \delta]]_{(x-t)}\cong \left(Q^+[[X_m;\Sigma_m,\Delta_m]]_{(X_m-T_m)}\right)^k$, which completes the proof of Theorem \ref{letterthm: simple when no power of sigma is inner}.

Now, finally, we can complete the proof of our most important result, Theorem \ref{letterthm: infinite order implies prime}.

\begin{thm}\label{thm: infinite order implies prime}
Suppose $(R,w_0)$ is a filtered ring of characteristic $p$ satisfying \eqref{filt}, and $(\sigma,\delta)$ is a commuting skew derivation on $R$ which is compatible with $w_0$. Fix a $\sigma$-standard filtered artinian completion $(Q,u)$ of $(R,w_0)$, and suppose also that 

\begin{itemize}
\item there exists $t\in Q(R)$ such that $\delta(r)=tr-\sigma(r)t$ for all $r\in R$, and

\item no positive power of $\sigma$ is inner as an automorphism of $Q$. 
\end{itemize}

Then $R^+[[x;\sigma,\delta]]$ is a prime ring.
\end{thm}

\begin{proof}

Since $Q$ is semisimple artinian, and $\sigma$ is not an inner automorphism of $Q$, it follows from Lemma \ref{lem: power 1} that $\sigma(t)=t$. It follows from Lemma \ref{lem: all assumptions} that $(\sigma,\delta)$ descends to a quasi-compatible skew derivation of $(Q,u)$, and the data $(Q,u,\sigma,\delta)$ satisfies Assumptions 1--5 of \S \ref{subsec: crossed}.

Using Assumption 1, we can decompose $Q$ into a direct product $Q_1\times\dots\times Q_{p^\ell}$, and using Theorem \ref{thm: can assume M is sigma-invariant}, if we can prove that $Q_i^+[[y;\Sigma_\ell|_{Q_i},\Delta_\ell|_{Q_i}]]$ is prime for any $i$, then it follows that $R^+[[x;\sigma,\delta]]$ is prime.

Assume for contradiction that some power of $\Sigma_\ell$ restricts to an inner automorphism of $Q_i$, i.e. there exists $r\geq 1$ and $a\in Q_i$ such that $\Sigma_{\ell}^r(q)=\sigma^{rp^\ell}(q)=aqa^{-1}$ for all $q\in Q_i$. Then for any $j=1,\dots,p^\ell$, if $q\in Q_j$ then $q=\sigma^{j-i}(q')$ for some $q'\in Q_i$, so $$\sigma^{rp^\ell}(q)=\sigma^{rp^\ell}(\sigma^{j-i}(q'))=\sigma^{j-i}(\sigma^{rp^\ell}(q'))=\sigma^{j-i}(aq'a^{-1})=\sigma^{j-i}(a)q\sigma^{j-i}(a)^{-1}.$$
So let $A\in Q=Q_1\times\dots\times Q_{p^\ell}$ have $j$'th component $\sigma^{j-i}(a)$, then $A\in Q^{\times}$ and $\sigma^{rp^\ell}(q)=AqA^{-1}$ for all $q\in Q$, so $\sigma^{rp^\ell}$ is an inner automorphism of $Q$, contradicting our assumption.

Therefore, no power of $\Sigma_\ell$ restricts to an inner automorphism of $Q_i$ for any $i$. Moreover, since $\Delta_\ell(q)=T_\ell q-\Sigma_\ell(q) T_\ell$ for any $q\in Q$, if we let $t_i$ be the $i$'th component of $T_\ell$, then the restriction of $\Delta_\ell$ to $Q_i$ satisfies $\Delta_\ell(q)=t_iq-\Sigma_\ell(q)t_i$ for all $q\in Q_i$.

In other words, the restricted skew derivation $(\Sigma_\ell|_{Q_i},\Delta_\ell|_{Q_i})$ of $Q_i$ satisfies the same assumptions as $(\sigma,\delta)$, so we may assume from now on that $\ell = 0$ and $Q$ is simple, and it remains to show that $Q^+[[x;\sigma,\delta]]$ is prime.

But since no power of $\sigma$ is inner, it follows from Theorem \ref{thm: simple when no power of sigma is inner} that $Q^+[[x;\sigma,\delta]]_{(x-t)}$ is a simple ring, and hence $Q[[x;\sigma,\delta]]$ is a prime ring by Lemma \ref{lem: normal2}, and this completes the proof.\end{proof}

To complete the proof that the Iwasawa-type skew power series ring $R[[x;\sigma,\sigma-\id]]$ is prime, the only remaining problem is to address the case where some power of $\sigma$ is an inner automorphism. This problem will be addressed in a later work.

\bibliography{biblio}
\bibliographystyle{plain}

\end{document}